\newif\ifdraft		
\newif\iflualatex	
\newif\ifpictures 
\draftfalse
\lualatexfalse
\picturestrue

\documentclass[
	11pt, 
	a4paper, 
	\ifdraft draft,\else final,\fi
	oneside%
]{article}

\iflualatex \usepackage{fontspec} \else \usepackage[T1]{fontenc} \fi
\usepackage{lmodern}
\usepackage[main=english, ngerman]{babel} 
\usepackage{fullpage}
\usepackage{tocloft}
\usepackage{caption}
\usepackage{xcolor} 
	\definecolor{darkblue}{RGB}{0,35,102}
	\definecolor{darkgreen}{RGB}{15,140,15}
\usepackage[final]{graphicx} 
	\ifdraft\ifpictures\else \setkeys{Gin}{draft} \fi\fi 
\usepackage{tabularx} 
\usepackage{csquotes} 
\usepackage{paralist} 

\usepackage{mathrsfs} 
\usepackage{amsmath, amsfonts, amssymb, amsthm} 
\usepackage{MnSymbol} 
	\DeclareSymbolFont{Symbols}{OMS}{cmsy}{m}{n}
	\SetSymbolFont{Symbols}{bold}{OMS}{cmsy}{b}{n}
	\DeclareMathSymbol{\setminus}{\mathbin}{Symbols}{"6E} 
\usepackage{mathtools} 
\usepackage{dsfont} 
\usepackage{xfrac} 
\usepackage{esint} 

\usepackage[color,notcite,notref]{showkeys} 
	\colorlet{labelkey}{orange}
	\colorlet{refkey}{orange}
	
\usepackage{todonotes} 
	\ifdraft\else\renewcommand{\todo}[2][]{} \fi 
	
\usepackage{changes} 
	\definechangesauthor[name=Kai, color=orange]{KR}

\usepackage{placeins} 
\usepackage{seqsplit} 
\usepackage{caption, subcaption} 
\usepackage{fontawesome5} 
\usepackage{tikz} 
	\tikzset{>=stealth}
	\usetikzlibrary{
		decorations.pathreplacing, 
		decorations.pathmorphing, 
		patterns, 
		mindmap, 
		trees, 
		calc,
		arrows}

\usepackage[final]{hyperref} 
	\hypersetup{
		colorlinks, 
		linkcolor=blue, 
		urlcolor=red, 
		filecolor=green, 
		citecolor=blue}
\usepackage[capitalise]{cleveref} 
\usepackage[open, openlevel=2, numbered]{bookmark}
\usepackage[backend=biber, style=alphabetic, giveninits=true, maxnames=5]{biblatex} 
\theoremstyle{plain} 
\newtheorem{theorem}{Theorem}[section]
\newtheorem{lemma}[theorem]{Lemma}
\newtheorem{proposition}[theorem]{Proposition}

\newtheorem{remark}[theorem]{Remark}

\theoremstyle{definition} 
\newtheorem{definition}[theorem]{Definition}
\newtheorem{assumption}[theorem]{Assumption}

\crefname{assumption}{Assumption}{Assumptions}

\DeclareMathOperator{\SO}{SO}
\DeclareMathOperator{\SL}{SL}
\DeclareMathOperator{\Gl}{Gl}
\DeclareMathOperator{\dist}{dist}

\DeclareMathOperator{\oclass}{o}
\DeclareMathOperator{\sym}{sym}
\DeclareMathOperator{\skewMat}{skew}
\newcommand{\R}{\mathbb{R}}

\newcommand{\N}{\mathbb{N}}

\newcommand{\Z}{\mathbb{Z}}
\newcommand{\eps}{\varepsilon}
\newcommand{\dd}{\mathrm{d}}
\newcommand{\wto}{\rightharpoonup}
\newcommand{\placeholder}{\makebox[1ex]{$\boldsymbol{\cdot}$}}
\newcommand{\grad}[1][]{\if\relax\detokenize{#1}\relax \nabla \else \nabla_{\!#1}\fi}

\newcommand{\el}{\mathrm{el}}
\newcommand{\pl}{\mathrm{pl}}
\DeclarePairedDelimiter{\abs}{\lvert}{\rvert} 
\DeclarePairedDelimiter{\norm}{\lVert}{\rVert}
\DeclarePairedDelimiter{\of}{\lparen}{\rparen}
\DeclarePairedDelimiterX{\compose}[2]{\lparen}{\rparen}{#1 \:\delimsize\vert\allowbreak\:\mathopen{} #2}
\newcommand\setSymbol[1][]{\nonscript\:#1\vert\allowbreak\nonscript\:\mathopen{}} 
\newcommand{\given}{\setSymbol}
\DeclarePairedDelimiterX{\set}[1]{\lbrace}{\rbrace}{\renewcommand{\given}{\setSymbol[\delimsize]}#1} 

\DeclareMathOperator{\Cont}{C}
\DeclareMathOperator{\SobH}{H}
\DeclareMathOperator{\SobW}{W}
\DeclareMathOperator{\Leb}{L}

\newcommand{\stepemph}[1]{\textsc{#1}}

\makeatletter
\newcommand{\matheqmargin}[1]{\setlength{\@mathmargin}{#1}}
\newcommand{\namedlabel}[2][]{\def\@currentlabelname{#1}\label{#2}}
\makeatother

	{\par\xdef\tpd{\the\prevdepth}\egroup\prevdepth=\tpd}
\makeatletter
	{\@fleqnfalse}%
	{\@fleqntrue\ignorespacesafterend}
\makeatother

\counterwithin{equation}{section} 

\allowdisplaybreaks 

\author{Stefan Neukamm, Kai Richter}
\title{Dimension reduction for elastoplastic rods in the bending regime}

\hypersetup{
	pdftitle={Dimension reduction for elastoplastic rods in the bending regime},
	pdfsubject={Nonlinear Elasticity},
	pdfauthor={Kai Richter},
	pdfkeywords={finite plasticity, finite elasticity, rate-independent system, evolutionary Gamma-convergence, dimension reduction, rod, bending, strain gradient elasticity},
	pdfdisplaydoctitle}

\DeclareDelimFormat[bib,biblist]{nametitledelim}{\addcolon\space}
\DeclareFieldFormat*{title}{\mkbibemph{#1}\isdot}
\DeclareFieldFormat*{journaltitle}{#1}

\AtEveryBibitem{%
  \clearfield{url}%
  \clearfield{issn}%
  \clearfield{isbn}%
  \clearfield{month}}
\begin{document}
\selectlanguage{english}%
\makeatletter%
\begin{flushleft}
	\def\refTUDsymbol{\textasteriskcentered}%
	\newcommand{\refTUD}{\hyperref[Title:TUD]{\textsuperscript{\refTUDsymbol}}}%
	\newcommand{\mail}[2]{\faEnvelope[regular] \hypersetup{hidelinks}\href{mailto:#1}{#1} (#2)}%
	\LARGE\bfseries
	\parbox[t]{0.7\linewidth}{\flushleft\@title}
	\par \bigskip

  \large\normalfont
  Stefan Neukamm\refTUD\ $\cdot$
  Kai Richter\refTUD
  \par

  \normalsize
	\mbox{} \hfill \today
	\par \bigskip
	
	\footnotesize 
	\textsuperscript{\refTUDsymbol}Faculty of Mathematics, Technische Universit{\"a}t Dresden, 01062 Dresden, Germany \label{Title:TUD}
	\par \medskip
	
	\mail{stefan.neukamm@tu-dresden.de}{S.~Neukamm} \\
	\mail{kai.richter@tu-dresden.de}{K.~Richter}
\end{flushleft}
\makeatother
\vspace{-2\baselineskip} \bigskip\bigskip

\paragraph*{Abstract.} 
We rigorously derive an effective bending model for elastoplastic rods starting from three-dimensional finite plasticity. For the derivation we lean on a framework of evolutionary $\Gamma$-convergence for rate-independent systems.
The main contribution of this paper is an ansatz for the mutual recovery sequence in the bending regime. In difference to previous works that deal with infinitesimal deformations in the limit, in the bending regime we are concerned with finite deformations that admit infinitesimally small strain but large rotations. To overcome these difficulties we provide new ideas based on a multiplicative decomposition for the Cosserat rod ansatz.
As regularizing terms, we introduce strain gradient terms into the energy that vanish as the thickness of the rod tends to zero.
\bigskip

\begin{footnotesize}
\begin{compactdesc}
	\item[\textbf{Keywords:}] finite plasticity $\cdot$ finite elasticity $\cdot$ rate-independent system $\cdot$ evolutionary $\Gamma$-convergence $\cdot$ dimension reduction $\cdot$ rod $\cdot$ bending $\cdot$ strain gradient elasticity.
	\item[\textbf{MSC-2020:}] 74C15 (primary) $\cdot$ 49J45 $\cdot$ 74K10 $\cdot$ 74Q15.
\end{compactdesc}
\end{footnotesize}

\paragraph{Acknowledgments.}
The authors received support from the German Research Foundation (DFG) via the research unit FOR 3013, “Vector- and tensor-valued surface PDEs” (project number 417223351).

{\hypersetup{hidelinks} \tableofcontents}
\ifdraft\listoftodos\fi

\ifdraft \renewcommand{\input}[1]{\include{#1}} \fi 

\section{Introduction}

The derivation of effective models for thin structures from three-dimensional models is a classical and fundamental topic in the field of continuum mechanics. In particular, the derivation via variational dimension reduction within the framework of $\Gamma$-convergence has been subject of considerabel attention in recent decades. In the case of static, nonlinear elasticity, based on the pioneering work of Friesecke, James and M\"{u}ller \cite{FJM02}, many results have been obtained, e.g., for bending rods \cite{MM03,Neu12,BNS20,BGKN22-PP}, plates in the von K{\'a}rm{\'a}n regime \cite{FJM06,NV13,GOW20}, plates in the bending regime \cite{AD20,BGNPP23,BNPS22}, and for thin films \cite{DKPS21}. The derivation of models in the evolutionary case is relatively recent and less developed. To our knowledge, existing results only address cases with geometrically linear effective models. In particular, in the von K{\'a}rm{\'a}n regime models for elastoplastic plates \cite{Dav14_02}, magnetoelastic plates \cite{BK23}, viscoelastic plates \cite{FK20}, and viscoelastic rods \cite{FM23} have been recently obtained. These results use concepts from evolutionary $\Gamma$-convergence and exploit the fact that the considered energy scaling leads to a linearization at identity. In the evolutionary case, such a linearization has been carried out for the first time in the seminal work by Mielke and Stefanelli \cite{MS13} where linearized elastoplasticity is obtained as a evolutionary $\Gamma$-limit of finite elastoplasticity.

In the present paper we consider the bending regime. In contrast to previous results our effective model is geometrically nonlinear and allows for deformations with large rotations.  More specifically, we discuss the derivation of an elastoplastic bending-torsion rod model in the form of an evolutionary rate-independent system (ERIS). The derivation of the effective model is based on the concept of evolutionary $\Gamma$-convergence for ERIS introduced in \cite{MRS08} and which has also been used in \cite{MS13,Dav14_02}. Our result requires new ideas to deal with the finite deformations in the limit.
Although we focus on elastoplasticity, we believe that the developed ideas can be adapted to different rod models, e.g., ERIS to describe damage, phase separation and swelling models, and rate-dependent models such as viscoelasticity.
\medskip

In the following we present the main setting and results.
Throughout this contribution, we consider an elastoplastic rod $\Omega^h := \omega \times hS$, where $\omega = (0,l)$ and $S \subset \R^2$ denotes the cross-section and $0 < h \ll 1$ is the thickness of the rod. The stored energy density $\mathsf{\Psi}$ of the rod depends on the deformation $\mathsf{y}:\Omega^h\to\R^3$ and on an internal variable $\mathsf{z}:\Omega^h\to\R^{3\times 3}$ depicting the linearized plastic strain:
\begin{equation*}
	\mathsf{\Psi}\big(x, \mathsf{y}, \mathsf{z}\big) = \mathsf{W}_\el\big(\mathsf{F}_\el(x)\big) + \mathsf{H}^\mathsf{C}_\el\big(\grad \sqrt{\mathsf{C_y}}(x)\big) + \mathsf{H}^\mathsf{R}_\el\big(\grad \mathsf{R_y}(x)\big) + \mathsf{W}_\pl\big(\mathsf{F}_\pl(x)\big),
\end{equation*}
where $F_\el$ and $F_\pl$ denote the elastic and plastic part of the strain and are defined, following \cite{Lee69}, via the multiplicative decomposition
\begin{equation*}
	\grad\mathsf{y}(x) = \mathsf{F}_\el(x)\mathsf{F}_\pl(x), \qquad
	\text{where } \mathsf{F}_\pl(x) := I + h \mathsf{z}(x).
\end{equation*}
Moreover, $\sqrt{\mathsf{C_y}(x)} \in \R^{3\times 3}_{\sym}$ and $\mathsf{R_y}(x) \in \SO(3)$ are obtained from $\mathsf{y}$ by the polar decomposition,
\begin{equation} \label{Eq:PolarDecomposition}
	\grad\mathsf{y}(x) = \mathsf{R_y}(x)\sqrt{\mathsf{C_y}(x)}, \qquad
	\text{where } \mathsf{C_y}(x) := \grad \mathsf{y}(x)^T\grad \mathsf{y}(x).
\end{equation}
The model is based on a decomposition of the energy into an elastic part and a hardening part. For the elastic part, we consider a standard elastic energy density $\mathsf{W}_\el$ which is frame-indifferent, minimized at the set of rotations and non-degenerated, see \cref{Ass:ElasticMaterialLaw} for the precise assumptions. The elastic energy is augmented by two strain-gradient terms $\mathsf{H}^\mathsf{C}_\el$ and $\mathsf{H}^\mathsf{R}_\el$ which act on the components of the polar decomposition of $\grad\mathsf{y}$ individually.
We consider these strain-gradient terms to achieve additional regularity required for handling the geometric and physical nonlinearities. However, thanks to the split of the strain-gradient terms into $\mathsf{H}^\mathsf{C}_\el$ and $\mathsf{H}^\mathsf{R}_\el$ and since we consider different scalings of these terms, the contribution of the strain-gradient terms vanishes in the limit, see \cref{Rem:StrainGradientTerms}. Although we introduce the strain-gradient terms mainly for technical reasons, the concept of strain-gradients originates in models for nonsimple materials (cf.\ \cite[Sect.~2.5]{KR19}) and has a long history that goes back to Toupin \cite{Tou62,Tou64} and Mindlin \cite{Min64}. In particular, already \cite{Min64} considers strain-gradients in the form of $\mathsf{H}^\mathsf{C}_\el$ and $\mathsf{H}^\mathsf{R}_\el$. The physical relevance and origin of nonsimple material models is still a point of discussion among experts, see e.g. \cite{BS24} for a recent contribution that justifies strain-gradient elasticity from microscopic effects in the context of phase-field crystal models. Nevertheless, strain-gradients are used in many recent variational studies of nonlinear material models, e.g.\ \cite{FK20,DKPS21,BFK23,MR20,OL24}.
\medskip

It is standard practice for thin structures to scale the reference domain to unit thickness. Therefore, we consider $\Omega := \Omega^1$ and introduce the rescaled variables $y(x_1, \bar{x}) = \mathsf{y}(x_1, h\bar{x})$ and $z(x_1,\bar{x}) := \mathsf{z}(x_1, h\bar{x})$, $x = (x_1,\bar{x}) \in \Omega$. The scaling to unit thickness yields a transformation of the gradient to $\grad[h] := \compose{\partial_1}{\frac{1}{h}\bar{\grad}}$, where $\bar{\grad} := (\partial_2, \partial_3)$ denotes the normal gradient and $\partial_1$ the tangential gradient. Since $R_y$ is not well-defined, if $\det \grad[h] y = 0$, we introduce the set of admissible deformations,
\begin{equation} \label{Eq:AdmissibleDeformations}
	\mathcal{A}^h_\mathrm{def} := \set*{y \in \SobW^{2,p}(\Omega, \R^3) \given \det \grad[h] y > 0 \text{ a.e.}}.
\end{equation} 
By combining the considerations above and suitable scaling of the terms, we are let to the following energies. In the total energy we include a time-dependent loading term $l(t): \Omega \to \R^3$, which drives the evolution.
\begin{subequations} \label{Eq:Definition3DEnergy}
\begin{align}
	\mathcal{E}_\el^h(y, z) &:= \begin{cases}
		\begin{aligned}[b] 
			h^{-2} \int_\Omega W_\el\big(\grad[h]y(x)(I + h z(x))^{-1}\big) \,\dd x& \\
			+ h^{-\alpha_C p} \int_{\Omega} H^C_\el\big(\grad \sqrt{C_y}(x)\big) \,\dd x& \\
			+ h^{\alpha_R p} \int_{\Omega} H^R_\el\big(\grad R_y(x)\big) \,\dd x&,
		\end{aligned} & \text{if } y \in \mathcal{A}^h_\mathrm{def}, \\
		\infty & \text{else},
	\end{cases} \\
	\mathcal{E}_\pl^h(z) &:= h^{-2} \int_\Omega W_\pl\big(I + h z(x)\big) \,\dd x, \\
	\mathcal{E}^h_{\el+\pl}(y,z) &:= \mathcal{E}_\el^h(y,z) + \mathcal{E}_\pl^h(z), \\
	\mathcal{E}^h(t,y,z) &:= \mathcal{E}_{\el+\pl}^h(y,z) - \int_\Omega l(t,x) \cdot y(x) \,\dd x, \label{Eq:DefinitionEnergy}
\end{align}
\end{subequations}
where $\alpha_C, \alpha_R > 0$ with $\alpha_R < \frac{2}{3}(1 - \alpha_C)$ and $p > 3$. Note that naturally also $\grad[h] \sqrt{C_y}$ and $\grad[h]R_y$ would appear in the energy. Thus, we implicitly consider different scaling for the tangential and normal components in these terms. A motivation for this scaling is presented in \cref{Rem:StrainGradientTerms}. The scaling $h^{-2}$ of the elastic energy and hardening energy implies that in the limit $h \to 0$ we obtain a bending theory.
\medskip

Additionally to the free energy, the dissipation of the plastic strain is relevant for the dynamics. As in \cite{MRS08,MS13,Dav14_02} we model the dissipation by the \emph{dissipation distance},
\begin{equation} \label{Eq:DefinitionDissipation}
	\mathcal{D}^h(z, \hat{z}) := h^{-1}\int_{\Omega} \mathscr{D}_\pl\big(I + hz(x), I + h\hat{z}(x)\big) \,\dd x.
\end{equation}
The dissipation distance describes the minimal amount of energy that is dissipated by passing from a state $z$ of the plastic strain to $\hat{z}$. For ERIS, the dissipation should not depend on the rate of change of the plastic strain but just on the path taken. Here, this is achieved by relating the dissipation distance density $\mathscr{D}_\pl$ to a positively 1-homogeneous map $\mathscr{R}_\pl$ as described later in \eqref{Eq:DefinitionDissipationDistance}.
An appropriate weak solution concept for ERIS composed of the energy and dissipation functionals $\mathcal{E}^h$ and $\mathcal{D}^h$ and a suitable state space $\mathcal{Q}^h$ has been proven to be the concept of \emph{energetic solutions}. The state space is here given as
\begin{equation} \label{Eq:DefinitionStateSpace}
	\mathcal{Q}^h := \SobW^{2,p}(\Omega, \R^3) \times \Leb^2(\Omega, \R^{3 \times 3}).
\end{equation}

\begin{definition}[Energetic solutions, {cf.\ \cite{MT04}}] \label{Def:EnergeticSolution}
Consider an ERIS given by $\mathcal{M} := (\mathcal{Q}, \mathcal{E}, \mathcal{D})$, where $\mathcal{Q} = \mathcal{Y} \times \mathcal{Z}$ is a set, $\mathcal{E}: [0,T] \times \mathcal{Q} \to (-\infty, \infty]$ and $\mathcal{D}: \mathcal{Z} \times \mathcal{Z} \to [0,\infty]$. We denote the \emph{set of stable states} by
\begin{equation}
	S_\mathcal{M}(t) := \set*{q=(y,z) \in \mathcal{Q} \given \mathcal{E}(t,q) < \infty, \mathcal{E}(t, q) \leq \mathcal{E}(t, \hat{q}) + \mathcal{D}(z,\hat{z}) \text{ for all } \hat{q}=(\hat{y},\hat{z}) \in \mathcal{Q}}.
\end{equation}
Moreover, we define the dissipation of a trajectory $z: [s,t] \to \mathcal{Z}$, $0 \leq s < t \leq T$ by
\begin{equation}
	\operatorname{Diss}_\mathcal{M}(z; [s,t]) := \sup \set*{\sum_{i=1}^N \mathcal{D}\big(z(t_{i-1}), z(t_i)\big) \given s=t_0 < \dots < t_N = t, N \in \N}.
\end{equation}
An \emph{energetic solution} to the ERIS is a trajectory $q=(y,z): [0,T] \to \mathcal{Q}$, such that
\begin{enumerate}[(a)]
	\item (Global stability). $q(t)$ is a stable state for any $t \in [0,T]$, i.e.
	\begin{equation} \label{Eq:GlobalStability}
		q(t) \in S_\mathcal{M}(t).
	\end{equation}
	\item (Global Energy balance). $[0,T] \ni t \mapsto \partial_t \mathcal{E}(t,q(t))$ is well-defined and integrable and for any $t \in [0,T]$,
	\begin{equation} \label{Eq:GlobalEnergyBalance}
		\mathcal{E}(t, q(t)) + \operatorname{Diss}_\mathcal{M}(z; [0,t]) = \mathcal{E}(0, q(0)) + \int_0^t \partial_s \mathcal{E}(s,q(s)) \,\dd s.
	\end{equation}
\end{enumerate}
\end{definition}

\begin{remark}
Note that in our case, since only the loading term is time-dependent, we have for any $q=(y,z) \in \mathcal{Q}^h$ with $\mathcal{E}^h_{\el+\pl}(q) < \infty$,
\begin{equation}
	\partial_t \mathcal{E}^h(t,q) = -\int_\Omega \partial_t l(t,x) \cdot y(x) \,\dd x. 
\end{equation}
\end{remark}

In this paper we study the convergence of energetic solutions to the ERIS $(\mathcal{Q}^h, \mathcal{E}^h, \mathcal{D}^h)$. We consider one-sided boundary conditions determined by $(v_\mathrm{bc}, R_\mathrm{bc}) \in \R^3 \times \SO(3)$. Such boundary conditions have already been used in the static case, cf.\ \cite{Neu12,BGKN22-PP}. We introduce these boundary conditions by restricting the state space to
\begin{equation} \label{Eq:DefinitionStateSpaceBC3D}
	\mathcal{Q}^h_{(v_\mathrm{bc}, R_\mathrm{bc})} := \set*{(y,z) \in \mathcal{Q}^h \given y(0, \bar{x}) = v_\mathrm{bc} + h x_2 R_\mathrm{bc}e_2 + hx_3 R_\mathrm{bc}e_3 \text{ for a.e.\ } \bar{x} \in S},
\end{equation}
Our main result is the following.

\begin{theorem} \label{Thm:MainResultConvergenceOfSolutions}
Consider \cref{Ass:ElasticMaterialLaw,Ass:PlasticMaterialLaw}. Let $(v_\mathrm{bc}, R_\mathrm{bc}) \in \R^3 \times \SO(3)$ and $(y^h, z^h): [0,T] \to \mathcal{Q}^h_{(v_\mathrm{bc}, R_\mathrm{bc})}$ be energetic solutions to the ERIS $(\mathcal{Q}^h_{(v_\mathrm{bc}, R_\mathrm{bc})}, \mathcal{E}^h, \mathcal{D}^h)$. Then, up to a subsequence we have for all $t \in [0,T]$, 
\begin{subequations} \label{Eq:ConvergenceOfEnergeticSolutionsBC}
\begin{align}
	z^h(t) &\wto z(t) && \text{in } \Leb^2(\Omega, \R^{3\times 3}), \\
\intertext{and up to a further $t$-dependent subsequence,}
	y^h(t) &\to v(t) && \text{in } \SobH^1(\Omega, \R^3), \\
	\grad[h]y^h(t) &\to R(t) && \text{in } \Leb^2(\Omega, \R^{3\times 3}),
\end{align}
\end{subequations}
where $(v,R,z): [0,T] \to \mathcal{Q}^0_{(v_\mathrm{bc}, R_\mathrm{bc})}$ is an energetic solution to the ERIS $(\mathcal{Q}^0_{(v_\mathrm{bc}, R_\mathrm{bc})}, \mathcal{E}^0, \mathcal{D}^0)$, see below. (For the proof see \cref{Sec:ConvergenceOfSolutions}.)
\end{theorem}

The limiting system is a bending model for rods, cf.\ \cite{MM03,Neu12,BNS20,BGKN22-PP} combined with linearized plasticity \cite{MS13,Dav14_02}. Limiting deformations consist of bending and twisting of the rod, which is described by the following set of possible rod configurations:
\begin{equation}
	\mathcal{A}_\mathrm{rod} := \set*{(v,R) \in \SobH^2(\omega, \R^3) \times \SobH^1(\omega, \R^{3\times 3}) \given R \in \SO(3) \text{ a.e.\ and } \partial_1 v = R e_1}.
\end{equation}
Combined with the linearized plastic strain $z$, this leads to the state spaces (without and with boundary conditions, respectively),
\begin{subequations}
\begin{align}
	\mathcal{Q}^0 &:= \mathcal{A}_\mathrm{rod} \times \Leb^2(\Omega, \R^{3\times 3}), \\
	\mathcal{Q}^0_{(v_\mathrm{bc}, R_\mathrm{bc})} &:= \set*{(v,R,z) \in \mathcal{Q}^0 \given v(0)=v_\mathrm{bc}, R(0)=R_\mathrm{bc}}.
\end{align}
\end{subequations}
The limiting energy and dissipation are defined as follows:
\begin{subequations} \label{Eq:LimitingModel}
\begin{align}
	\mathcal{E}^0_\el(v,R,z) &:= \int_\omega Q_\el^\mathrm{eff}\big(R^T\partial_1 R - K^\mathrm{eff}(z)\big) + \int_\Omega Q_\el\big(z^\mathrm{res}\big), \label{Eq:LimitingElasticEnergy}\\
	\mathcal{E}^0_\pl(z) &:= \int_\Omega Q_\pl(z), \\
	\mathcal{E}^0_{\el+\pl}(v,R,z) &:= \mathcal{E}^0_\el(v,R,z) + \mathcal{E}^0_\pl(z), \\
	\mathcal{E}^0(t,v,R,z) &:= \mathcal{E}^0(v,R,z) - \int_{\omega} l^\mathrm{eff}(t,x_1) \cdot v(x_1) \,\dd x_1, \\
	\mathcal{D}^0(z, \hat{z}) &:= \int_\Omega \mathscr{R}_\pl\of*{\hat{z} - z}.
\end{align}
\end{subequations}
The limiting energy admits the following effective quantities:
\begin{itemize}
	\item $Q_\el^\mathrm{eff}: \R^{3\times 3}_{\skewMat} \to \R$, a non-degenerate quadratic form which describes the bending-torsion energy of the rod,
	\item $K^\mathrm{eff}: \Leb^2(\Omega, \R^{3\times 3}) \to \Leb^2(\omega, \R^{3\times 3}_{\skewMat})$, the effective contribution of the plastic strain to the bending and torsion of the rod,
	\item $z^\mathrm{res} \in \Leb^2(\Omega, \R^{3\times 3}_{\sym})$, a residual plastic strain that leads to an energy which cannot be accommodated by bending and torsion of the rod,
	\item $l^\mathrm{eff}(t,x_1) := \int_S l(t,x_1, \bar{x}) \,\dd \bar{x}$, the effective loading.
\end{itemize}
The definition of these effective quantities is due to an orthogonal projection technique established in \cite{BNS20,BGKN22-PP}. See \cref{Sec:RepresentationFormulas} for more information and especially \cref{Def:EffectiveCoefficients} for the definition of the quantities.

\begin{remark} \label{Rem:UniqueSolution}
The 3D and limiting ERIS are non-convex and thus, solutions are in general non-unique. A prominent example for this observation is the buckling of a symmetric rod. Therefore, generally solutions might jump at any time. Consequently, we cannot expect temporal regularity of the non-dissipative variable $y$ and convergence of solutions cannot be expected without extraction of a time-dependent subsequence.
On the other hand, if the solution to the limiting ERIS is unique, we observe in \cref{Thm:MainResultConvergenceOfSolutions} convergence of the whole sequence $(y^h(t),\grad[h]y^h(t),z^h(t))$ by a standard procedure, see \cref{Rem:ConvergenceSubsequences} for more details. In \cref{Sec:Example} we study an example where uniqueness of the solution can be expected.
\end{remark}

It is important to note that we do not require nor prove existence of solutions to the 3D model. In fact, in \cref{Thm:MainResultConvergenceOfSolutions} we implicitly assume that solutions exist. However, existence of solutions to the finite system is generally unknown. 
Nevertheless, we can always obtain approximate solutions to a related time-incremental problem. Following \cite{MS13,Dav14_02} we prove that these approximate solutions also convergence to solutions to the limiting ERIS:

\begin{theorem} \label{Thm:MainResultConvergenceOfApproximateSolutions}
Consider \cref{Ass:ElasticMaterialLaw,Ass:PlasticMaterialLaw}. Let $(v_\mathrm{bc}, R_\mathrm{bc}) \in \R^3 \times \SO(3)$, $0 = t_0^h < \cdots < t_{N^h}^h = T$ partitions with $\lim_{h \to 0}\max_{i \leq N^h}(t_i^h - t_{i-1}^h) = 0$ and $\kappa: (0,\infty) \to (0,\infty)$ with $\lim_{h \to 0}\kappa(h) = 0$. Consider $(y^h_i, z^h_i) \in \mathcal{Q}^h_{(v_\mathrm{bc}, R_\mathrm{bc})}$, such that for $i=1,\dots,N^h$,
\begin{equation} \label{Eq:DiscreteSolutionScheme}
	\mathcal{E}^h(t_i^h, y^h_i, z^h_i) +  \mathcal{D}^h(z^h_{i-1}, z^h_i) \leq \kappa(h)(t^h_i - t^h_{i-1}) + \inf_{(y,z) \in \mathcal{Q}^h_{\mathrlap{(v_\mathrm{bc}, R_\mathrm{bc})}}} \left(\mathcal{E}^h(t_i^h, y, z) +  \mathcal{D}^h(z^h_{i-1}, z)\right).
\end{equation}
Let $(y^h, z^h): [0,T] \to \mathcal{Q}^h_{(v_\mathrm{bc}, R_\mathrm{bc})}$ be the associated right-continuous, piece-wise constant interpolants. Then, up to a subsequence, the convergences \eqref{Eq:ConvergenceOfEnergeticSolutionsBC} hold with $(v,R,z): [0,T] \to \mathcal{Q}^0_{(v_\mathrm{bc}, R_\mathrm{bc})}$ being an energetic solution to the ERIS $(\mathcal{Q}^0_{(v_\mathrm{bc}, R_\mathrm{bc})}, \mathcal{E}^0, \mathcal{D}^0)$. (For the proof see \cref{Sec:ConvergenceOfSolutions}.)
\end{theorem}

Especially, we obtain existence of solutions to the limiting model as an immediate consequence.
To obtain these results, we utilize the general theory of evolutionary $\Gamma$-convergence for energetic solutions to ERIS, which was introduced in \cite{MRS08}. Similar to the theory of regular $\Gamma$-convergence for minimization problems, the procedure consists of proving certain properties of the energy and dissipation functionals. We summarize the most important properties in the following theorem.

\begin{theorem} \label{Thm:EvolutionaryGammaConvergence}
Consider \cref{Ass:ElasticMaterialLaw,Ass:PlasticMaterialLaw}. Let $(v_\mathrm{bc}, R_\mathrm{bc}) \in \R^3 \times \SO(3)$ and $(y^h, z^h) \subset \mathcal{Q}^h$. Then, the following statements hold. Moreover, the statements remain true, if $(y^h, z^h) \subset \mathcal{Q}^h_{(v_\mathrm{bc}, R_\mathrm{bc})}$ and we replace $\mathcal{Q}^h$ by $\mathcal{Q}^h_{(v_\mathrm{bc}, R_\mathrm{bc})}$ and $\mathcal{Q}^0$ by $\mathcal{Q}^0_{(v_\mathrm{bc}, R_\mathrm{bc})}$. In this case we can also replace \eqref{Eq:ConvergenceOfEnergeticSolutions} below with \eqref{Eq:ConvergenceOfEnergeticSolutionsBC}.
\begin{enumerate}[(a)]
	\item (Compactness): For some constant $c>0$ (independent of $h$, $y^h$ and $z^h$), we find
	\begin{equation} \label{Eq:Coercivity}
		\norm*{y^h - {\textstyle\fint_S} y^h(0,\placeholder)}_{\Leb^2(\Omega)}^2 + \norm*{\grad[h]y^h}_{\Leb^2(\Omega)}^2 + \norm*{z^h}_{\Leb^2(\Omega)}^2 \leq c\left(1 + \mathcal{E}^h_{\el+\pl}(y^h,z^h)\right).
	\end{equation}
	Moreover, if
	\begin{equation} \label{Eq:BoundedEnergy}
		\limsup_{h \to 0}\, \mathcal{E}^h_{\el+\pl}(y^h,z^h) < \infty.
	\end{equation}
	we find $(v,R,z) \in \mathcal{Q}^0$, such that up to a subsequence (not relabeled),
	\begin{subequations} \label{Eq:ConvergenceOfEnergeticSolutions}
	\begin{align}
		y^h - {\textstyle\fint_S} y^h(0,\placeholder) &\to v - v_\mathrm{bc} && \text{in } \SobH^1(\Omega, \R^3), \\
		\grad[h]y^h &\to R && \text{in } \Leb^2(\Omega, \R^{3\times 3}), \\
		z^h &\wto z && \text{in } \Leb^2(\Omega, \R^{3\times 3}).
	\end{align}
	\end{subequations}
	(For the proof see \cref{Sec:Compactness}.)
\end{enumerate}
Additionally assume \eqref{Eq:BoundedEnergy} and that the convergences \eqref{Eq:ConvergenceOfEnergeticSolutions} hold for some $(v,R,z) \in \mathcal{Q}^0$.
\begin{enumerate}[(a)] \setcounter{enumi}{1}
	\item (Lower bound): Let $\hat{z}^h, \hat{z} \in \Leb^2(\Omega, \R^{3\times 3})$ with $\hat{z}^h \wto \hat{z}$ in $\Leb^2(\Omega, \R^{3\times 3})$. Then,
	\begin{subequations} \label{Eq:LiminfInequalities}
	\begin{align}
		\liminf_{h \to 0} \mathcal{E}^h_{\el+\pl}(y^h,z^h) &\geq \mathcal{E}^0_{\el+\pl}(v,R,z), \\
		\liminf_{h \to 0} \mathcal{D}^h(z^h, \hat{z}^h) &\geq \mathcal{D}^0(z,\hat{z}).
	\end{align}
	\end{subequations}
	(For the proof see \cref{Sec:LowerBound}.)
	
	\item (Mutual recovery sequence): Let $(\hat{v}, \hat{R}, \hat{z}) \in \mathcal{Q}^0$. Then, for each subsequence of $(h)$ we find a further subsequence (not relabeled) and a sequence $(\hat{y}^h, \hat{z}^h) \subset \mathcal{Q}^h$ satisfying the convergences \eqref{Eq:ConvergenceOfEnergeticSolutions} (with $(y^h, z^h)$ and $(v,R,z)$ replaced by $(\hat{y}^h, \hat{z}^h)$ and $(\hat{v}, \hat{R}, \hat{z})$), such that
	\begin{multline} \label{Eq:MututalLimsupInequality}
		\limsup_{h \to 0} \Big(\mathcal{E}_{\el+\pl}^h(\hat{y}^h, \hat{z}^h) + \mathcal{D}^h(z^h, \hat{z}^h) - \mathcal{E}_{\el+\pl}^h(y^h, z^h)\Big) \\
		\leq \mathcal{E}_{\el+\pl}^0(\hat{v}, \hat{R}, \hat{z}) + \mathcal{D}^0(z, \hat{z}) - \mathcal{E}_{\el+\pl}^0(v, R, z).
	\end{multline}
	(For the proof see \cref{Sec:RecoverySequence}.)	
\end{enumerate}
\end{theorem}

Especially interesting is the construction of the mutual recovery sequence. We use (up to some correction) a classical Cosserat ansatz
\begin{equation}
	\hat{y}^h(x) = \hat{v}^h(x_1) + hx_2\hat{R}^h(x_1)e_2 + hx_3\hat{R}^h(x_1)e_3,
\end{equation}
where $(\hat{v}^h, \hat{R}^h) \in \mathcal{A}_\mathrm{rod}$ is a suitable rod configuration, see \cref{Prop:ConstructionRecoverySequenceDeformation} for the precise ansatz. What separates this mutual recovery sequence construction from the constructions in static problems (cf.\ \cite{MM03,Neu12,BNS20}) is that in order to obtain \eqref{Eq:MututalLimsupInequality}, $\hat{v}^h$ and $\hat{R}^h$ (as well as the correction term) in general have to be chosen depending on $y^h$. The approaches developed in \cite{MS13,Dav14_02} rely on strong convergence of the differences of strains associated to $y^h$ and $\hat{y}^h$. 
We adapt this approach here, but are faced with an additional difficulty. In these papers infinitesimal deformations are considered in the limit. Thus, the linearity of the space of limiting displacements can be exploited in the construction. However, this is not the case in this paper, since bending deformations yield a non-convex state space. Hence, we have to use a new, entirely multiplicative approach. The main idea is to use a construction $\hat{R}^h = \tilde{R}R^h$ where $R^h$ belongs to a rod configuration related to $y^h$ and $\tilde{R} = \hat{R}R^T$. To our knowledge this is the first rigorous result in this direction dealing with finite deformations in the limit.
\medskip

\paragraph{Structure of the paper.}
In \cref{Sec:ModelSetting}, we discuss the 3D model in detail and introduce the required assumptions. Afterwards in \cref{Sec:RepresentationFormulas} we define the effective quantities of the limiting model and study their properties. Then, in \cref{Sec:Proofs} we prove our main results. This section is structured into separate subsections. In \cref{Sec:Compactness} we prove the compactness statement \cref{Thm:EvolutionaryGammaConvergence} (a), in \cref{Sec:LowerBound} the lower bound \cref{Thm:EvolutionaryGammaConvergence} (b), in \cref{Sec:RecoverySequence} we construct the mutual recovery sequence and finally in \cref{Sec:ConvergenceOfSolutions} we prove \cref{Thm:MainResultConvergenceOfSolutions,Thm:MainResultConvergenceOfApproximateSolutions,}.
We finish with an example in \cref{Sec:Example} where uniqueness of solutions to the limiting model can be expected.

\subsection{Notation}

\paragraph{Decomposition.} $e_1, e_2, e_3$ denotes the canonical basis of $\R^3$. We denote by $\compose{A_1}{\bar{A}}$ the $n \times 3$ matrix with $A_1 \in \R^n$ as the first column and $\bar{A} \in \R^{n \times 2}$ as the two last columns. Vise versa, given $A \in \R^{n \times 3}$ we denote the first and last two columns as $A_1 \in \R^n$ and $\bar{A} \in \R^{n \times 2}$, respectively. Especially, we denote $\grad = \compose{\partial_1}{\bar{\grad}}$. For $x \in \R^3$ we also use the notation $x = (x_1, \bar{x})$.

\paragraph{Normal part.} We denote the mapping of vectors in $\R^3$ to their normal part by $\mathbf{\bar{x}} : \R^3 \to \R^3$, $x \mapsto x_2e_2 + x_3e_3$ and identify it with the map $\mathbf{\bar{x}} : \R^2 \to \R^3$, $\bar{x} \mapsto x_2e_2 + x_3e_3$.

\paragraph{Subsets.} We introduce notations for certain subsets of $\R$ and $\R^{3\times 3}$. We denote $\R_+ := [0,\infty)$ and $\bar{\R}_+ := [0,\infty] := \R_+ \cup \set*{\infty}$. Moreover, we set 
\begin{alignat*}{2}
	\Gl(3) &:= \set*{F \in \R^{3\times 3} \given \det F \neq 0}, 
	&\quad \Gl_+(3) &:= \set*{F \in \R^{3\times 3} \given \det F > 0}, \\
	\SL(3) &:= \set*{F \in \R^{3\times 3} \given \det F = 1},
	&\quad \SO(3) &:= \set*{F \in \SL(3) \given F^T F = FF^T= I}, \\
	\R^{3\times 3}_\mathrm{sym} &:= \set*{F \in \R^{3\times 3} \given F^T = F},
	&\quad \R^{3\times 3}_\mathrm{dev} &:= \set*{F \in \R^{3\times 3}_\mathrm{sym} \given \operatorname{trace} F = 0}.
\end{alignat*}

\paragraph{Strain gradients.} We identify $\R^{3\times 3 \times 3}$ with $\R^{(3\times 3) \times 3}$, i.e.\ row vectors consisting of three matrices. Especially, we define the multiplication of $R \in \R^{3\times 3}$ and $\mathbb{F} \in \R^{3\times 3 \times 3}$ as the entry-wise matrix multiplication, i.e.\ $R\mathbb{F} := (R\mathbb{F}_1, R\mathbb{F}_2, R\mathbb{F}_3)$ where $\mathbb{F} = (\mathbb{F}_1, \mathbb{F}_2, \mathbb{F}_3)$ or more precise $(R\mathbb{F})_{ijk} := \sum_{l=1}^3 R_{il}\mathbb{F}_{ljk}$. Moreover, we denote the derivative of maps $F \in \SobW^{1,1}_\mathrm{loc}(\R^3, \R^{3\times 3})$ by $\grad F := (\partial_1 F, \partial_2 F, \partial_3 F)$ or in coordinates $(\grad F)_{ijk} := \partial_k F_{ij}$. We define the Hessian of a map $y \in \SobW^{2,1}_\mathrm{loc}(\R^3, \R^3)$ as $\grad^2 y := \grad (\grad y)$. The same notation is used for the scaled gradient $\grad[h]$ and Hessian $\grad[h]^2$.

\paragraph{Polar decomposition.} Given $h > 0$ and $y \in \SobH^1(\Omega, \R^3)$ with $\det\grad[h]y > 0$ a.e., we denote the components of the polar decomposition of $\grad[h] y$ by $R_y$ and $\sqrt{C_y}$, i.e.\ $C_y := \grad[h]y^T \grad[h]y$ and $R_y := \grad[h]y (\sqrt{C_y})^{-1}$. We suppress the dependence of $h$ in this notation, but when used, the choice of $h$ can be taken from the context. It is common knowledge that with this definition, $R_y$ minimizes the distance of $\grad[h]y$ to $\SO(3)$, i.e.\ $R_y \in \SO(3)$ a.e.\ and
\begin{equation}
	\dist\of[\big]{\grad[h]y, \SO(3)} = \abs*{\grad[h]y - R_y} = \abs*{\sqrt{C_y} - I}.
\end{equation}

\paragraph{Rotation fields.} We define the set $\SobH^1_{\SO(3)}(\omega, \R^{3\times 3}) := \SobH^1(\omega, \R^{3\times 3}) \cap \Leb^\infty(\omega, \SO(3))$. Given $R \in \SobH^1_{\SO(3)}(\omega, \R^{3\times 3})$, we introduce $K_R := R^T \partial_1 R \in \Leb^2(\omega, \R^{3\times 3}_{\skewMat})$.

\section{Model setting} \label{Sec:ModelSetting}

In this section we present the 3D model and the required assumptions in detail. Our model is given by an ERIS described by the free energy $\mathcal{E}^h$, dissipation distance $\mathcal{D}^h$ and state space $\mathcal{Q}^h$ as introduced in \eqref{Eq:DefinitionEnergy}, \eqref{Eq:DefinitionDissipation} and \eqref{Eq:DefinitionStateSpace}, respectively. These quantities are determined by
\begin{itemize}
	\item the reference domain $\Omega \subset \R^3$,
	\item the elastic energy densities $W_\el: \R^{3\times 3} \to \R$ and $H_\el^C, H_\el^R: \R^{3\times 3 \times 3} \to \R$,
	\item the hardening energy density $W_\pl: \R^{3\times 3} \to \R$,
	\item the dissipation distance density $\mathscr{D}_\pl: \R^{3 \times 3} \times \R^{3\times 3} \to \R$, which itself depends on the dissipation potential density $\mathscr{R}_\pl: \R^{3\times 3} \to \R$,
	\item and the time dependent external loading density $l: [0,T] \times \Omega \to \R^3$.
\end{itemize}
For the readers convenience, we summarize the quantities introduced in this section in \cref{Tab:Quantities,Tab:Quantities_intermediate}.

\begin{table}[!ht]
\centering
\framebox{
\begin{tabularx}{0.9\linewidth}{l@{\hspace{5pt}}l@{\hspace{5pt}}l@{}X}
\multicolumn{4}{l}{\textit{Sets and domains}} \\
	\multicolumn{3}{l@{}}{$\Omega=\omega \times S$\dotfill\dots} & reference domain with $\omega = (0,l)$ and cross-section $S \subset \R^2$, \\
	\multicolumn{3}{l@{}}{$K_\pl \subset \SL(3)$\dotfill\dots} & set of finite plastic energy. \\
\multicolumn{4}{l@{}}{\textit{Parameters}} \\
	\multicolumn{3}{l@{}}{$h > 0$\dotfill\dots} & thickness of the rod, \\
	\multicolumn{3}{l@{}}{$p > 3$\dotfill\dots} & growth of the elastic strain-gradient energy densities, cf.\ \cref{Ass:ElasticMaterialLaw}, \\
	\multicolumn{3}{l@{}}{$\alpha_C, \alpha_R > 0$\dotfill\dots} & scaling of the respective elastic strain-gradient energies, connected by $\alpha_R < \frac{2}{3}(1 - \alpha_C)$, \\
\multicolumn{4}{l}{\textit{Elastic material law}} \\
	$W_\el$ & $(F)$ &  $\in \bar{\R}_+$\dotfill\dots & elastic energy density, \\
	$Q_\el$ & $(G)$ & $\in \R_+$\dotfill\dots & linearized elastic energy density, \\
	$H^C_\el$, $H^R_\el$ & $(\mathbb{F})$ & $\in \R_+$\dotfill\dots & elastic strain-gradient energy densities, \\
\multicolumn{4}{l}{\textit{Plastic material law}} \\
	$W_\pl$ & $(F)$ & $\in \bar{\R}_+$\dotfill\dots & hardening energy density, \\
	$Q_\pl$ & $(G)$ & $\in \R_+$\dotfill\dots & linearized hardening energy density, \\
	$\mathscr{R}_\pl$ & $(\dot{F})$ & $\in \bar{\R}_+$\dotfill\dots & dissipation potential density, \\
	$\mathscr{D}_\pl$ & $(F,\hat{F})$ & $\in \bar{\R}_+$\dotfill\dots & dissipation distance density, \\
\multicolumn{4}{l}{\textit{External loading}} \\
	$l$ & $(t,x)$ & $\in \R^3$\dotfill\dots & external loading density.
\end{tabularx} }
\caption{Summary of the main quantities discussed in this paper. Here $t \in [0,T]$ represents a dependence on the time, $x \in \Omega$ on the (material) coordinates, $F,G \in \R^{3\times 3}$ on the (linearized) strain, $\dot{F} \in \R^{3\times 3}$ on the evolution of the strain and $\mathbb{F} \in \R^{3\times 3 \times 3}$ on the strain-gradient.}
\label{Tab:Quantities}
\end{table}

\begin{table}[!ht]
\centering
\framebox{
\begin{tabularx}{0.9\linewidth}{l@{\hspace{5pt}}l@{\hspace{5pt}}l@{}X}
\multicolumn{4}{l}{\textit{Elastic material law}} \\
	$\mathbb{C}_\el$ &  & $\in \R^{3\times 3 \times 3 \times 3}$\dotfill\dots & fourth-order tensor associated to $Q_\el$, \\
	$r_\el$ & $(\delta)$ & $\in \bar{\R}_+$\dotfill\dots & rest term for the quadratic expansion of $W_\el$. \\
\multicolumn{4}{l}{\textit{Plastic material law}} \\
	$\mathbb{C}_\pl$ & & $\in \R^{3\times 3 \times 3 \times 3} $\dotfill\dots & fourth-order tensor associated to $Q_\pl$, \\
	$r_\pl$ & $(\delta)$ & $\in \bar{\R}_+$\dotfill\dots & rest term for the quadratic expansion of $W_\pl$.
\end{tabularx} }
\caption{Summary of helping quantities and quantities relevant for the analysis. The interpretation is as in \cref{Tab:Quantities} and in addition $\delta \in \R_+$ corresponds to a dependence of the modulus of $G$.}
\label{Tab:Quantities_intermediate}
\end{table}

\paragraph{Reference domain.}

Let $S \subset \R^2$ a open, bounded and connected Lipschitz domain and $\omega := (0,l)$. For simplicity, we assume the positioning properties
\begin{equation} \label{Eq:ReferenceDomainCentered}
	\int_S x_2 \,\dd\bar{x} = \int_S x_3 \,\dd\bar{x} = \int_S x_2 x_3 \,\dd \bar{x} = 0.
\end{equation}
The (rescaled) reference domain is denoted by $\Omega := \omega \times S$.

\paragraph{Loading.}

We assume $l \in \SobW^{1,1}\of*{[0,T], \SobH^1(\Omega, \R^3)}$.

\paragraph{Elastic material law.} We consider the following properties for the energy densities $W_\el$, $H_\el^C$ and $H_\el^R$.

\begin{enumerate}
\renewcommand{\labelenumi}{(WE\arabic{enumi})}
\renewcommand{\theenumi}{WE\arabic{enumi}}
\makeatletter
\def\@currentlabelname{\upshape (WE1) -- (WE4)}
\makeatother

	\item \label{Ass:WE}\label{Ass:W_FrameIndifference}(Frame indifference):
	\begin{equation*}
		W_\el(RF) = W_\el(F) \qquad \text{for all } F \in \R^{3 \times 3}, R \in \SO(3);
	\end{equation*}
	
	\item \label{Ass:W_NonDegeneracy}(Non-degeneracy): There exists some $0 < c \leq C$ and $\rho > 0$, such that
	\begin{alignat*}{2}
		W_\el(F) &\geq c \dist^2(F, \SO(3)) &&\qquad \text{for all } F \in \R^{3\times 3}, \\
		W_\el(F) &\leq C \dist^2(F, \SO(3)) &&\qquad \text{for all } F \in \R^{3\times 3} \text{ with } \dist^2(F,\SO(3)) \leq \rho;
	\end{alignat*}
	
	\item \label{Ass:W_Expansion}(Quadratic expansion): There exists a quadratic form $Q_\el: \R^{3\times 3} \to \R$ and an increasing map $r_\el: \R_+ \to \bar{\R}_+$ with $\lim_{\delta \to 0} r_\el(\delta) = 0$, such that
	\begin{equation*}\label{Eq:TaylorExpansion}
		\abs*{W_\el(I+G) - Q_\el(G)} \leq \abs*{G}^2 r_\el(\abs*{G}) \qquad \text{for all } G \in \R^{3\times 3};
	\end{equation*}
	
	\item \label{Ass:W_Continuity}(Continuity): There exist $\gamma > 0$ and $L \in \R_+$, such that for all $F_1, F_2 \in B_\gamma(I)$ and $F \in \R^{3 \times 3}$,
	\begin{equation*}
		\abs*{W_\el(F_1 F F_2) - W_\el(F)} \leq L (W_\el(F) + 1)(\abs*{F_1 - I} + \abs*{F_2 - I});
	\end{equation*}
\end{enumerate}

\begin{enumerate}
\renewcommand{\labelenumi}{(HE\arabic{enumi})}
\renewcommand{\theenumi}{HE\arabic{enumi}}
\makeatletter
\def\@currentlabelname{\upshape (HE1) -- (HE3)}
\makeatother

	\item \label{Ass:HE}\label{Ass:HE_FrameIndifference}(Frame indifference):
	\begin{equation*}
		H_\el^R(R\mathbb{F}) = H_\el^R(\mathbb{F}) \qquad \text{for all } \mathbb{F} \in \R^{3 \times 3 \times 3}, R \in \SO(3);
	\end{equation*}
	
	\item \label{Ass:HE_Growth}($p$-growth): For each $H_\el \in \set*{H_\el^C, H_\el^R}$ there exist $0 < c \leq C$, such that
	\begin{equation*}
		c \abs*{\mathbb{F}}^p \leq H_\el(\mathbb{F}) \leq C \left( \abs*{\mathbb{F}}^p + 1\right) \qquad \text{for all } \mathbb{F} \in \R^{3 \times 3 \times 3};
	\end{equation*}
	
	\item \label{Ass:HE_Continuity}($p$-Lipschitz continuity): There exists $L_C, L_R \in \R_+$, such that
	\begin{align*}
		\abs*{H_\el^C(\mathbb{F}) - H_\el^C(\hat{\mathbb{F}})} &\leq L_C \left(\abs*{\mathbb{F}}^{p-1} + \abs*{\hat{\mathbb{F}}}^{p-1}\right) \abs*{\mathbb{F} - \hat{\mathbb{F}}}, \\
		\abs*{H_\el^R(\mathbb{F}) - H_\el^R(\hat{\mathbb{F}})} &\leq L_R \left(\abs*{\mathbb{F}}^{p-1} + \abs*{\hat{\mathbb{F}}}^{p-1} + 1\right) \abs*{\mathbb{F} - \hat{\mathbb{F}}}  && \text{for all } \mathbb{F}, \hat{\mathbb{F}} \in \R^{3 \times 3 \times 3}.
	\end{align*}
\end{enumerate}

\begin{assumption}[Elastic material law] \label{Ass:ElasticMaterialLaw}
We assume that $W_\el$ satisfies \nameref{Ass:WE} and $H_\el^C, H_\el^R$ satisfy \nameref{Ass:HE} for some $p > 3$. Moreover, we fix $0 < \alpha_C, \alpha_R < 1$ with $\alpha_R < \frac{2}{3}(1- \alpha_C)$.
\end{assumption}

We denote by $\mathbb{C}_\el \in \R^{3\times 3 \times 3 \times 3}$ the symmetric fourth-order tensor associated to the quadratic form $Q_\el$ by the polarization formula
\begin{equation}
	G : \mathbb{C}_\el \hat{G} = \tfrac{1}{2} \left( Q_\el(G + \hat{G}) - Q_\el(G) - Q_\el(\hat{G})\right), \qquad G,\hat{G} \in \R^{3\times 3}.
\end{equation}

\begin{remark}~
\begin{enumerate}[(i)]
	\item The assumptions \nameref{Ass:WE} are standard assumptions for the elastic energy density and used in a similar form in various contributions, e.g.\ \cite{BNPS22,Neu12,MS13,Dav14_02}. The condition \eqref{Ass:W_Continuity} is usually given in the form $\abs{F^T \grad W_\el(F)} \leq c(W_\el(F) + 1)$ which is stronger as shown in \cite[Lem.~4.1]{MS13}. Since we only require \eqref{Ass:W_Continuity}, we choose to assume this condition directly.
	
	\item The quadratic form $Q_\el$ is uniquely defined from \eqref{Ass:W_Expansion} as $Q_\el(G) = \lim_{h \to 0} \tfrac{1}{h^2} W_\el(I + hG)$. It is well-known in elasticity that \nameref{Ass:WE} imply that $Q_\el$ satisfies
	\begin{equation}
		c \abs*{\sym G}^2 \leq Q_\el(G) \leq C \abs*{\sym G}^2 \quad \text{for all } G \in \R^{3\times 3},
	\end{equation}
	with the same constants as in \eqref{Ass:W_NonDegeneracy}. Especially, $Q_\el(G)$ and $\hat{G} : \mathbb{C}_\el G$ depend only on the symmetric part of $G$ and $\hat{G}$.
	
	\item The specific choice of $r_\el$ is not important. However, we may use the explicit choice
	\begin{equation}
		r_\el(\delta) := \sup\set*{\tfrac{1}{\abs*{G}^2} \abs[\big]{W_\el(I + G) - Q_\el(G)} \given G \in \R^{3\times 3}, 0 < \abs{G} \leq \delta}, \qquad \delta > 0.
	\end{equation}
\end{enumerate}
\end{remark}

\begin{remark} \label{Rem:StrainGradientTerms}
The growth condition \eqref{Ass:HE_Growth} and $p$-Lipschitz continuity \eqref{Ass:HE_Continuity} are usual assumptions for strain-gradient elasticity, see e.g. \cite{FK18,FK20,DKPS21,BFK23}. It is natural that the frame-indifference \eqref{Ass:HE_FrameIndifference} is only assumed for the part depending on $R_y$, since $\sqrt{C_y}$ is already independent of the frame, cf.\ \cite[Sect.~2(b)]{DSV09}. The introduction of the strain-gradient terms is mainly for analytical reasons. Thus, we seek to choose them in such a way that they vanish in the limit. We motivate that this is expected with our choice of the individual scalings $h^{-\alpha_C}$ and $h^{\alpha_R}$: In the limit we expect deformations that are close to the classical Cosserat ansatz
\begin{equation*}
	y^h(x) = v(x_1) + hR(x_1)\mathbf{\bar{x}}(x),
\end{equation*}
for some $(v,R) \in \mathcal{A}_{\rm rod}$. We observe
\begin{equation*}
	\grad[h] y^h = R + h R\compose*{K_R \mathbf{\bar{x}}}{0}, \qquad
	C_{y^h} = I + 2h \sym \compose*{K_R \mathbf{\bar{x}}}{0} + h^2 \compose*{K_R \mathbf{\bar{x}}}{0}^T \compose*{K_R \mathbf{\bar{x}}}{0}.
\end{equation*}
Together with \cref{Lem:EquivalenceRestTerms} this shows that we can expect $h^{-\alpha_C} \grad \sqrt{C_{y^h}} \to 0$ and $h^{\alpha_R} \grad R_{y^h} \to 0$ whenever $0 < \alpha_C, \alpha_R < 1$. The assumption $\alpha_R < \frac{2}{3}(1- \alpha_C)$ is technical and allows us to infer sufficient regularity for the correction term $\phi^h$, see \cref{Rem:alpha_q}.
\end{remark}

\paragraph{Plastic material law.}

We introduce the following two material classes for the hardening energy density and the dissipation potential which are adapted from \cite{MS13,Dav14_02}.

\begin{definition}[{cf.\ \cite{MS13,Dav14_02}}] \label{Def:MaterialClass_WP}
Let $K_\pl \subset \operatorname{SL}(3)$ compact, such that $I$ is a relatively interior point of $K_\pl$. We denote by $\mathfrak{W}_\pl(K_\pl)$ the set of all functions $W_\pl:\R^{3\times 3} \to \bar{\R}_+$, which are of the form
\begin{equation*}
	W_\pl(F) = \begin{cases} \widetilde{W}_\pl(F) & \text{if } F \in K_\pl, \\ \infty & \text{else,} \end{cases}
\end{equation*}
with $\widetilde{W}_\pl: \R^{3\times 3} \to \R_+$ satisfying the following properties.
\begin{enumerate}
\renewcommand{\labelenumi}{(WP\arabic{enumi})}
\renewcommand{\theenumi}{WP\arabic{enumi}}
\makeatletter
\def\@currentlabelname{\upshape (HE1) -- (HE3)}
\makeatother

	\item \label{Ass:WP_Continuity}(Continuity): We find $L \in \R_+$ and an open set $O \subset \R^{3\times 3}$ with $K_\pl \subset\subset O$, such that $\tilde{W}_\pl$ is $L$-Lipschitz continuous on $O$, i.e.\
	\begin{equation*}
		\abs*{\widetilde{W}_\pl(F) - \widetilde{W}_\pl(\hat{F})} \leq L \abs*{F - \hat{F}} \qquad \text{for all } F,\hat{F} \in O.
	\end{equation*}
	\item \label{Ass:WP_NonDegeneracy}(Non-degeneracy): There exists $c > 0$, such that
	\begin{equation*}
		\widetilde{W}_\pl(I + G) \geq c \abs*{G}^2 \qquad \text{for all } G \in \R^{3\times 3}.
	\end{equation*}
	\item \label{Ass:WP_Expansion}(Quadratic expansion): There exists a quadratic form $Q_\pl: \R^{3\times 3} \to \R$ and an increasing map $r_\pl: \R_+ \to \bar{\R}_+$ with $\lim_{\delta \to 0} r_\pl(\delta) = 0$, such that
	\begin{equation*}
		\abs*{\widetilde{W}_\pl(I+G) - Q_\pl(G)} \leq \abs*{G}^2 r_\pl(\abs*{G}) \qquad \text{for all } G \in \R^{3\times 3}.
	\end{equation*}
\end{enumerate}
\end{definition}

\begin{definition}[{cf.\ \cite{MS13,Dav14_02}}] \label{Def:DissipationClass_RP}
We denote by $\mathfrak{R}_\pl$ the set of all functions $\mathscr{R}_\pl: \R^{3\times 3} \to \bar{\R}_+$, which are of the form 
\begin{equation*}
	\mathscr{R}_\pl(\dot{F}) = \begin{cases} \widetilde{\mathscr{R}}_\pl(\dot{F}) & \text{if } \dot{F} \in \R^{3\times 3}_\mathrm{dev}, \\ \infty & \text{else,} \end{cases}
\end{equation*}
where $\widetilde{\mathscr{R}}_\pl: \R^{3\times 3}_\mathrm{dev} \to \R_+$ is a convex, positively 1-homogeneous function, such that there exist $0 < c \leq C$ satisfying
\begin{equation}
	c \abs*{\dot{F}} \leq \widetilde{\mathscr{R}}_\pl(\dot{F}) \leq C \abs*{\dot{F}} \qquad \text{for all } \dot{F} \in \R^{3\times 3}_\mathrm{dev}.
\end{equation}
\end{definition}

\begin{assumption}[Plastic material law] \label{Ass:PlasticMaterialLaw}
We assume that there exists a compact set $K_\pl \subset \operatorname{SL}(3)$ with $I$ as an relatively interior point, such that $W_\pl \in \mathfrak{W}_\pl(K_\pl)$ and $\mathscr{R}_\pl \in \mathfrak{R}_\pl$.
\end{assumption}

Given $\mathscr{R}_\pl$, we define the associated dissipation distance density $\mathscr{D}_\pl: \R^{3\times 3} \times \R^{3\times 3} \to \bar{\R}_+$ by
\begin{subequations} \label{Eq:DefinitionDissipationDistance}
\begin{align}
	\mathscr{D}_\pl(F, \hat{F}) &:= \begin{cases} \mathscr{D}_\pl(I, \hat{F} F^{-1}) & \text{if } F, \hat{F} \in \Gl_+(3), \\ \infty & \text{else,} \end{cases} \\
	\mathscr{D}_\pl(I, F) &:= \inf_{\substack{P \in \Cont^1([0,1],\Gl_+(3)), \\ P(0)=I, P(1)=F}} \int_0^1 \mathscr{R}_\pl(\dot{P}(t)P(t)^{-1}) \,\dd t.
\end{align}
\end{subequations}
Moreover, we denote by $\mathbb{C}_\pl \in \R^{3\times 3 \times 3 \times 3}$ the fourth-order tensor, associated to $Q_\pl$ by the polarization formula
\begin{equation}
	G : \mathbb{C}_\pl \hat{G} = \tfrac{1}{2} \left( Q_\pl(G + \hat{G}) - Q_\pl(G) - Q_\pl(\hat{G})\right), \qquad G,\hat{G} \in \R^{3\times 3}.
\end{equation}

\begin{remark}~ \label{Rem:PlasticProperties}
\begin{enumerate}[(i)]
	\item Similarly as for the elastic quantity, $Q_\pl$ is uniquely defined from \eqref{Ass:WP_Expansion} as $Q_\pl(G) = \lim_{h \to 0} \tfrac{1}{h^2} W_\pl(I + hG)$ and $r_\pl$ admits the explicit choice
	\begin{equation}
		r_\pl(\delta) := \sup\set*{\tfrac{1}{\abs*{G}^2} \abs*{W_\pl(I + G) - Q_\pl(G)} \given G \in \R^{3\times 3}, \abs*{G} \leq \delta}, \qquad \delta > 0.
	\end{equation}
	Moreover, from \eqref{Ass:WP_NonDegeneracy} we infer that for some constants $0 < c \leq C$,
	\begin{equation}
		c \abs*{G}^2 \leq Q_\pl(G) \leq C \abs*{G}^2 \qquad \text{for all } G \in \R^{3\times 3}.
	\end{equation}
	
	\item The properties of the set $K_\pl$ ensure that there exist constants $0 < c \leq C$, such that
	\begin{equation}\label{Eq:BoundednessInSetK}
		\abs*{F} \leq C, \qquad
		\abs*{F^{-1}} \leq C, \qquad
		\abs*{\hat{F} - I} \geq c, \qquad \text{for all } F \in K_\pl, \hat{F} \in \SL(3) \setminus K_\pl.
	\end{equation}
	While it is a very technical assumption that $W_\pl$ is infinite outside such a set, it is an important property in order to ensure that, as established in \cite[Lem.~3.1]{MS13},
	\begin{equation}
		\limsup_{h \to 0} \mathcal{E}_{\el+\pl}(y^h, z^h) < \infty \quad\Rightarrow\quad
		\limsup_{h \to 0} h^{-2} \int_\Omega \dist^2\of[\big]{\grad[h]y^h, \SO(3)} < \infty,
	\end{equation}
	i.e.\ sequences of finite energy also have finite bending energy and thus we are able to pass to a bending theory in the limit, cf.\ \cite{FJM02,MM03}.
	
	\item It is not hard to show that the dissipation distance $\mathscr{D}_\pl$ satisfies the triangle inequality
	\begin{equation}
		\mathscr{D}_\pl(F_1, F_3) \leq \mathscr{D}_\pl(F_1, F_2) + \mathscr{D}_\pl(F_2, F_3), \qquad \text{for all } F_1, F_2, F_3 \in \R^{3\times 3}.
	\end{equation}
	Moreover, it is proven in \cite[Rem.~2.4]{Dav14_02} that there exists a constant $c > 0$ such that
	\begin{equation}
		\mathscr{D}_\pl(F, \hat{F}) \leq c, \qquad \mathscr{D}_\pl(I, F) \leq c \abs*{F - I} \qquad \text{for all } F,\hat{F} \in K_\pl.
	\end{equation}
\end{enumerate}
\end{remark}


\section{Identification of the limiting model} 
\label{Sec:RepresentationFormulas}

The identification of the limiting elastic energy according to \eqref{Eq:LimitingElasticEnergy} has already been provided in \cite{BNS20,BGKN22-PP}. The form of the limiting energy is due to a representation of the limit of the nonlinear strain $E^h := h^{-1}(\sqrt{C_{y^h}} - I)$. According to \cite[Prop.~5.1]{BNS20}, the nonlinear strain admits a limit of the form
\begin{equation}
	E^h \wto \sym\compose*{K_R \mathbf{\bar{x}}}{0} + \chi \qquad \text{in } \Leb^2(\Omega, \R^{3\times 3}),
\end{equation}
where $R \in \SobH^1_{\SO(3)}(\omega, \R^{3\times 3})$ is given as in \eqref{Eq:ConvergenceOfEnergeticSolutions} and $\chi$ lives in the relaxation space $\Leb^2(\omega, \mathbb{H}_\mathrm{rel})$ with
\begin{equation}
	\mathbb{H}_\mathrm{rel} := \set*{\sym\compose*{a e_1}{\bar{\grad}\varphi} \given a \in \R, \varphi \in \SobH^1(S, \R^3)}.
\end{equation}
We shall see that this representation justifies that the limiting elastic energy can be written in the form
\begin{align}
	\mathcal{E}^0_\el(v,R,z) &= \min_{\chi \in \Leb^2(\omega, \mathbb{H}_\mathrm{rel})}\int_\Omega Q_\el\big( \sym\compose{K_R\mathbf{\bar{x}}}{0} + \chi - \sym(z)\big) \label{Eq:CharacterizationLimitingEnergy}\\
	&= \int_\omega \min_{\chi \in \mathbb{H}_\mathrm{rel}} \int_S Q_\el\big( \sym\compose{K_R\mathbf{\bar{x}}}{0} + \chi - \sym(z)\big). \notag
\end{align}
Note that the second equality can be obtained by appealing to the respective Euler-Lagrange equations of the left- and right-hand side. The right-hand side can be treated using orthogonal projections in the Hilbert space $\mathbb{H} := \Leb^2(S, \R^{3\times 3}_{\sym})$ equipped with the scalar product $(\chi,\psi)_Q := \int_S \chi: \mathbb{C}_\el \psi$. The details have already been carried out in \cite[Sect.~2.4]{BNS20} and more specifically \cite[Sect.~2.3]{BGKN22-PP} and are thus omitted here. To introduce the effective coefficients, we consider the subspaces $\mathbb{H}_\mathrm{micro}$, $\mathbb{H}_\mathrm{macro}$ and $\mathbb{H}_\mathrm{res}$ given by the orthogonal decompositions (in $\mathbb{H}$),
\begin{gather}
	\mathbb{H}_\mathrm{micro} := \set*{\sym\compose*{K\mathbf{\bar{x}}}{0} + \chi \given K \in \R^{3\times 3}_{\skewMat}, \chi \in \mathbb{H}_\mathrm{rel}}, \\
	\mathbb{H} = \mathbb{H}_\mathrm{micro} \oplus \mathbb{H}_\mathrm{res}, \qquad
	\mathbb{H}_\mathrm{micro} = \mathbb{H}_\mathrm{macro} \oplus \mathbb{H}_\mathrm{rel}.
\end{gather}
Moreover, we denote orthogonal projections in $\mathbb{H}$ onto some closed subspace $\mathcal{A} \subset \mathbb{H}$ by $\operatorname{P}_\mathcal{A}$. The procedure allows us to justify \eqref{Eq:LimitingElasticEnergy} by the representation
\begin{equation}
	\min_{\chi \in \mathbb{H}_\mathrm{rel}} \int_S Q_\el\big( \sym\compose{K\mathbf{\bar{x}}}{0} + \chi - z\big) = Q_\el^\mathrm{eff}\big(K - K^\mathrm{eff}(z)\big) + \int_S Q_\el\big(z^\mathrm{res}\big), 
\end{equation}
where $Q_\el^\mathrm{eff}$, $K^\mathrm{eff}(z)$ and $z^\mathrm{res}$ are defined as follows:

\begin{definition}\label{Def:EffectiveCoefficients}
Given $K \in \R^{3\times 3}_{\skewMat}$ and $z \in \Leb^2(S, \R^{3\times 3}_{\sym})$, we set
\begin{align}
	Q_\el^\mathrm{eff}(K) &:= \norm*{\operatorname{P}_{\mathbb{H}_\mathrm{macro}}\big(\sym\compose*{K\mathbf{\bar{x}}}{0}\big)}_Q^2, \\
	z^\mathrm{res} &:= \operatorname{P}_{\mathbb{H}_\mathrm{res}}(z),
\end{align}
and define $K^\mathrm{eff}(z) \in \R^{3\times 3}_{\skewMat}$ as the unique solution to
\begin{equation}
	 \operatorname{P}_{\mathbb{H}_\mathrm{macro}}\Big(\sym\compose*{K^\mathrm{eff}(z)\mathbf{\bar{x}}}{0}\Big) = \operatorname{P}_{\mathbb{H}_\mathrm{macro}}(z).
\end{equation}
Moreover, for $x \in \Omega$ and $z \in \Leb^2(\Omega, \R^{3\times 3})$, we set $z_{x_1} := \sym(z(x_1,\placeholder)) \in \Leb^2(S, \R^{3\times 3}_{\sym})$ and
\begin{equation}
	K^\mathrm{eff}(z)(x_1) := K^\mathrm{eff}(z_{x_1}), \qquad
	z^\mathrm{res}(x_1,\bar{x}) := z_{x_1}^\mathrm{res}(\bar{x}).
\end{equation}
\end{definition}

It has been shown in \cite[Lem.~2.10]{BNS20} (see also \cite[Lem.~2.8]{BGKN22-PP}) that $Q_\el^\mathrm{eff}$ is non-degenerate:

\begin{lemma}
There exist constants $0 < c \leq C$, such that for all $K \in \R^{3\times 3}_{\skewMat}$, $z \in \Leb^2(S, \R^{3\times 3}_{\sym})$,
\begin{equation*}
	c \abs*{K}^2 \leq Q_\el^\mathrm{eff}(K) \leq C \abs*{K}^2, \qquad
	\abs*{K^\mathrm{eff}(z)} \leq C\norm*{z}_{\Leb^2(S)}.
\end{equation*}
\end{lemma}

As in \cite{BNS20,BGKN22-PP} we obtain an alternative representation of these quantities using solutions to elliptic problems. For this we utilize the following basis of $\R^{3\times 3}_{\skewMat}$,
\begin{equation} \label{Eq:BasisSkewMatrices}
	K_1 := \tfrac{1}{\sqrt{2}}\left( 0 \;\middle|\; -e_3 \;\middle|\; e_2 \right), \qquad
	K_2 := \tfrac{1}{\sqrt{2}}\left( -e_2 \;\middle|\; e_1 \;\middle|\; 0 \right), \qquad
	K_3 := \tfrac{1}{\sqrt{2}}\left( -e_3 \;\middle|\; 0 \;\middle|\; e_1 \right).
\end{equation}

\begin{proposition} \label{Prop:EffectiveQuantities}
Let $K \in \R^{3\times 3}_{\skewMat}$ and $z \in \Leb^2(S, \R^{3\times 3}_{\sym})$. Define $(a_i, \varphi_i)$, $i=1,2,3$, as the unique minimizer satisfying $\fint_S \varphi_i = 0$ and $\fint_S \partial_3\varphi_i \cdot e_2 - \partial_2 \varphi_i \cdot e_3 = 0$ of
\begin{equation*}
	\R \times \SobH^1(S, \R^3) \ni (a,\varphi) \mapsto \int_S Q_\el\compose*{K_i\mathbf{\bar{x}} + ae_1}{\bar{\grad}\varphi},
\end{equation*}
and $(a_z, \varphi_z)$ as the unique minimizer satisfying $\fint_S \varphi_z = 0$ and $\fint_S \partial_3\varphi_z \cdot e_2 - \partial_2 \varphi_z \cdot e_3 = 0$ of
\begin{equation*}
	\R \times \SobH^1(S, \R^3) \ni (a,\varphi) \mapsto \int_S Q_\el\big(\compose*{ae_1}{\bar{\grad}\varphi} + z\big),
\end{equation*}
Let $\Psi_i := \sym\compose*{K_i\mathbf{\bar{x}} + a_ie_1}{\bar{\grad}\varphi_i}$ and $\mathbf{C} \in \R^{3\times 3}$ by $\mathbf{C}_{ij} = (\Psi_i, \Psi_j)_Q$, $i,j=1,2,3$. Then,
\begin{subequations}
\begin{align}
	Q_\el^\mathrm{eff}(K) &= \sum_{i,j=1}^3 \mathbf{C}_{ij} (K \cdot K_i) (K \cdot K_j) \\
	K^\mathrm{eff}(z) &= \sum_{i,j=1}^3 (z,\Psi_j)_Q (\mathbf{C}^{-1})_{ij} K_i, \\
	z^\mathrm{res} &= z + \sym\compose*{a_z e_1}{\bar{\grad}\varphi_z} - \sum_{i,j=1}^3 (z,\Psi_j)_Q (\mathbf{C}^{-1})_{ij} \Psi_i.
\end{align}
\end{subequations}

\end{proposition}

\section{Proofs} \label{Sec:Proofs}

This section is dedicated to the proof of \cref{Thm:MainResultConvergenceOfSolutions,Thm:MainResultConvergenceOfApproximateSolutions,Thm:EvolutionaryGammaConvergence}. We start with the proof of \cref{Thm:EvolutionaryGammaConvergence} with the compactness, lower bound and mutual recovery sequence statements discussed in separate subsections. We shall see that the compactness is due to the rigidity estimate of \cite{FJM02} and the lower bound statement follows from a careful Taylor expansion of $W_\el$ and $W_\pl$. The main difficulty is the mutual recovery sequence statement which requires a suitable choice for the ansatz as discussed in the introduction. Finally, we conclude \cref{Thm:MainResultConvergenceOfSolutions,Thm:MainResultConvergenceOfApproximateSolutions} from the previous observations using the theory established in \cite{MRS08}.

\subsection{Compactness; Proof of Theorem \ref{Thm:EvolutionaryGammaConvergence} (a)}
\label{Sec:Compactness}

This section is dedicated to the proof of \cref{Thm:EvolutionaryGammaConvergence} (a). The statement is essentially due to a combination of \cite[Thm.~2.1]{MM03} with \cite[Lem.~3.1]{MS13}. The second paper establishes the a priori estimate,
\begin{equation}\label{Eq:AprioriEstimateEnergy}
	h^{-2}\norm*{\dist\of[\big]{\grad[h]y^h, \SO(3)}}_{\Leb^2(\Omega)}^2 + \norm*{z^h}_{\Leb^2(\Omega)}^2 + h^2\norm*{z^h}_{\Leb^\infty(\Omega)}^2 \leq c \left( \mathcal{E}^h_{\el+\pl}(y^h,z^h) + 1\right),
\end{equation}
which already establishes most of \eqref{Eq:Coercivity}. The estimate follows essentially from \eqref{Ass:W_NonDegeneracy}, \eqref{Ass:WP_NonDegeneracy} and \cref{Rem:PlasticProperties}. Especially, we obtain that sequences $(y^h, z^h) \in \mathcal{Q}^h$ satisfying \eqref{Eq:BoundedEnergy} have finite bending energy in the sense that
\begin{equation}\label{Eq:FiniteBendingEnergy}
	\limsup_{h \to 0} h^{-2} \int_\Omega \dist^2\of[\big]{\grad[h]y^h, \SO(3)} < \infty.
\end{equation}
Then, the first paper establishes the convergences of $(y^h)$ for sequences with finite bending energy. However, we require stronger estimates for the mutual recovery sequence statement. We obtain those from the strain gradient terms in the energy. Thus, we are lead to reconstructing the arguments for the compactness. Our general procedure is as described above but we follow \cite{Neu12} whose argumentation slidly deviates from the one provided in \cite{MM03}. As in many similar contributions a crucial ingredient is the rigidity estimate \cite[Thm.~3.1]{FJM02}. We provide a version of the theorem, showing that the rotation can be chosen independent of the considered $\Leb^q$ space.

\begin{proposition}[cf.\ {\cite[Thm.~3.1]{FJM02}}] \label{Prop:RigidityEstimate}
Let $1 < q < \infty$ and $U \subset \R^3$ a bounded Lipschitz domain. There exists a constant $c > 0$, such that for all $\mathsf{y} \in \SobW^{1,\infty}(U, \R^3)$, we find a constant rotation $R \in \SO(3)$ independent of $q$, such that
\begin{equation} \label{Eq:RigidityEstimate}
	\norm*{\grad \mathsf{y} - R}_{\Leb^q(U)} \leq c \norm*{\dist(\grad \mathsf{y}, \SO(3))}_{\Leb^q(U)}.
\end{equation}
\end{proposition}

\begin{proof}
We show that the rotation $R$ can be chosen independent of $q$. In fact we may use the (semi-) explicit choice $R' \in \operatorname{Arg\,min}_{R \in \SO(3)} \abs*{\fint_U \grad \mathsf{y} - R}$ (see \cite{NR24-PP}). Indeed, let $R \in \SO(3)$ satisfy \eqref{Eq:RigidityEstimate}. Then, 
\begin{align*}
	\norm*{\grad \mathsf{y} - R'}_{\Leb^q(U)} &\leq \norm*{\grad \mathsf{y} - R}_{\Leb^q(U)} + \abs*{U}^{1/q}\abs*{R' - \textstyle\fint_U \grad \mathsf{y}} + \abs*{U}^{1/q}\abs*{\textstyle\fint_U \grad\mathsf{y} - R} \\
	&\leq \norm*{\grad \mathsf{y} - R}_{\Leb^q(U)} + \abs*{U}^{1/q}\abs*{R - \textstyle\fint_U \grad \mathsf{y}} + \abs*{U}^{1/q}\abs*{\textstyle\fint_U \grad\mathsf{y} - R} \\
	&\leq 3 \norm*{\grad \mathsf{y} - R}_{\Leb^q(U)} \leq c_1 \norm*{\dist(\grad \mathsf{y}, \SO(3))}_{\Leb^q(U)}. \qedhere
\end{align*}
\end{proof}
\bigskip 

As a corollary we get the following rigidity statements for rods, which is adapted from \cite[Prop.~3.6]{Neu12}.

\begin{proposition} \label{Prop:RigidityRods}
Let $1 < q < \infty$ and $h_0 > 0$. There exists a constant $c > 0$, such that for all $y \in \SobW^{1,\infty}(\Omega, \R^3)$ and any $0 < h \leq h_0$, we find a rotation field $R \in \SobW^{1,\infty}(\omega, \SO(3))$, such that $K_R = R^T\partial_1 R$ is piece-wise constant with jump-set contained in $h\Z$ and
\begin{subequations}
\begin{align}
	\norm*{\grad[h] y - R}_{\Leb^q(\Omega)} + \norm*{h \partial_1 R}_{\Leb^q(\omega)} &\leq c \norm*{\dist\of[\big]{\grad[h]y, \SO(3)}}_{\Leb^q(\Omega)}, \label{Eq:RigidityRodsLq}\\
	\norm*{h \partial_1 R}_{\Leb^\infty(\omega)} & \leq c \norm*{\dist\of[\big]{\grad[h]y, \SO(3)}}_{\Leb^\infty(\Omega)}. \label{Eq:RigidityRodsLinf}
\end{align}
\end{subequations}
Moreover, $R$ can be chosen independently of $q$ and if $(y, z) \in \mathcal{Q}^h_{(v_\mathrm{bc}, R_\mathrm{bc})}$ for some $z \in \Leb^2(\Omega, \R^{3\times 3})$ and $(v_\mathrm{bc}, R_\mathrm{bc}) \in \R^3 \times \SO(3)$, then we may choose $R$, such that $R(0) = R_\mathrm{bc}$.
\end{proposition}

\begin{proof}
Repeating the arguments for \cite[Prop.~3.6]{Neu12}, which already establishes \eqref{Eq:RigidityRodsLq} in the case $q = 2$, one can show that \eqref{Eq:RigidityRodsLq} holds for arbitrary $1 < q < \infty$. What is more, since $R$ is constructed using the rigidity estimate, $R$ can be chosen independently of $q$ by \cref{Prop:RigidityEstimate}. Especially, note that from the procedure one can establish the inequality
\begin{equation*}
	\sup_{x_1 \in (\xi, \xi+h)} \abs*{h\partial_1 R(x_1)}^2 \leq c_1 h^{-1} \norm*{\dist\of[\big]{\grad[h]y_*, \SO(3)}}_{\Leb^2([\xi-h,\xi+h)\times S)}^2, \qquad 
	\xi \in h\Z \cap [-h,l+h],
\end{equation*}
where $y^*$ is an extension of $y$ to $(-2h,l+2h) \times S$ with
\begin{equation*}
	\norm*{\dist\of[\big]{\grad[h] y_*, \SO(3)}}_{\Leb^2((-2h,0) \cup (l,l+2h) \times S)} \leq c_2 \norm*{\dist\of[\big]{\grad[h] y, \SO(3)}}_{\Leb^2((0,h) \cup (l-h,l) \times S)}.
\end{equation*}
From this, we obtain \eqref{Eq:RigidityRodsLinf} by comparing the $\Leb^2$ norm against the $\Leb^\infty$ norm. \cite[Prop.~3.6]{Neu12} also shows that one may choose $R$, such that $R(0) = R_\mathrm{bc}$, if $(y,z) \in \mathcal{Q}^h_{(v_\mathrm{bc}, R_\mathrm{bc})}$.
\end{proof}

We proceed by proving the a priori estimate \eqref{Eq:Coercivity}. By \eqref{Eq:AprioriEstimateEnergy} it remains to show the estimate for $y^h - \fint_S y^h(0,\placeholder)$. This is due to the following version of the Poincaré inequality. We show a stronger version than necessary for \eqref{Eq:Coercivity}, which we use in the proposition afterwards.

\begin{lemma}[Poincaré-Morrey] \label{Lem:PoincareMorrey}
Let $1 \leq q < \infty$. We find a constant $c > 0$, such that for all $y \in \SobW^{1,q}(\Omega)$ we have
\begin{equation}
	\norm*{y - {\textstyle \fint_S} y(0,\placeholder)}_{\Leb^q(\Omega)} \leq c \left( \norm*{\partial_1 y}_{\Leb^1(\Omega)} + \norm*{\bar{\grad} y}_{\Leb^q(\Omega)}\right).
\end{equation}
\end{lemma}

\begin{proof}
In the proof we follow \cite[Prop.~3.6]{Neu12}. Consider $\bar{y}(x_1) := \fint_S y(x_1,\bar{x})\,\dd \bar{x}$, $x_1 \in \omega$. Then, Poincaré's inequality yields
\begin{equation*}
	\norm*{y - \bar{y}}_{\Leb^q(\Omega)}^q = \int_\omega \norm*{y(x_1,\placeholder) - {\textstyle \fint_S}y(x_1,\placeholder)}_{\Leb^q(S)}^q \,\dd x_1 \leq c_1 \int_\omega \norm*{\bar{\grad} y(x_1,\placeholder)}_{\Leb^q(S)}^q \,\dd x_1 = c_1 \norm*{\bar{\grad} y}_{\Leb^q(\Omega)}^q.
\end{equation*}
Moreover, in one dimension the continuous representative satisfies $\bar{y}(x_1) - \bar{y}(0) = \int_0^{x_1} \partial_1 \bar{y}$. Thus, we can conclude the claim with the estimate,
\begin{equation*}
	\norm*{\bar{y} - \bar{y}(0)}_{\Leb^\infty(\omega)} \leq \norm*{\partial_1 \bar{y}}_{\Leb^1(\omega)} \leq \abs*{S}^{-1} \norm*{\partial_1 y}_{\Leb^1(\Omega)}.
	\qedhere
\end{equation*}
\end{proof}
\bigskip 

Now, \cref{Thm:EvolutionaryGammaConvergence} (a) is an immediate consequence of the following proposition, also in the variant with boundary conditions. 
The proposition shows that 3D deformation with finite energy in the sense of \eqref{Eq:BoundedEnergy} can be represented by means of a standard Cosserat ansatz with a small remainder $\phi^h$, see \eqref{eq:cosserat} below. The proposition also establishes estimates that we need for the construction of the mutual recovery sequence. In particular, using the regularizing properties of the strain gradient terms, the proposition yields smallness of the remainder $\phi^h$ in $L^q$ for large $q$.

\begin{proposition} \label{Prop:Compactness}
Let $(v_\mathrm{bc}, R_\mathrm{bc}) \in \R^3 \times \SO(3)$ and $(y^h, z^h) \subset \mathcal{Q}^h$ 
having bounded energy, i.e.\ satisfying \eqref{Eq:BoundedEnergy}. Then, $y^h$ admits a representation,
\begin{equation}\label{eq:cosserat}
	y^h(x) = v^h(x_1) + hR^h(x_1) \mathbf{\bar{x}}(x) + h\phi^h(x), \qquad x \in \Omega,
\end{equation}
for some rod configuration $(v^h, R^h) \in \mathcal{A}_\mathrm{rod}$ and $\phi^h \in \SobW^{1,\infty}(\Omega, \R^3)$, such that the following statements hold.
\begin{enumerate}[(a)]
	\item We have $v^h(0) = \fint_S y^h(0, \bar{x}) \,\dd \bar{x}$ and $\fint_S \phi^h(0,\bar{x}) \,\dd\bar{x} = 0$. Moreover, if $(y^h, z^h) \in \mathcal{Q}^h_{(v_\mathrm{bc}, R_\mathrm{bc})}$, then $v^h(0) = v_\mathrm{bc}$, $R^h(0) = R_\mathrm{bc}$ and $\phi^h(0, \bar{x}) = 0$ for a.e.\ $\bar{x} \in S$.
	
	\item We have $I + hz^h \in K_\pl$ a.e.\ in $\Omega$ for all $0 < h \ll 1$.

	\item For all $q \in [1,\infty]$ and $\tilde{q} \in [1,\infty)$, we have
	\begin{subequations}
	\begin{alignat}{10}
		&\limsup_{h \to 0}\; && && \norm*{z^h}_{\Leb^2(\Omega)} &&+ h && \norm*{z^h}_{\Leb^\infty(\Omega)} && && &&< \infty, \label{Eq:CompactnessZhBound}\\
		&\limsup_{h \to 0}\; && h^{-\alpha_C} && \norm[\Big]{\grad \sqrt{C_{y^h}}}_{\Leb^p(\Omega)} &&+ h^{\alpha_R} && \norm*{\grad R_{y^h}}_{\Leb^p(\Omega)} &&+ h^{\alpha_R} && \norm*{\grad \grad[h] y^h}_{\Leb^p(\Omega)} &&< \infty, \label{Eq:CompactnessStrainGradientsBound}\\
		&\limsup_{h \to 0}\; && h^{-\alpha_q} && \norm[\Big]{\sqrt{C_{y^h}} - I}_{\Leb^q(\Omega)} &&+h^{-\alpha_{\tilde{q}}} && \norm*{\grad[h]y^h - R^h}_{\Leb^{\tilde{q}}(\Omega)} &&+ h^{1-\alpha_q} &&\norm*{\partial_1 R^h}_{\Leb^q(\omega)} && < \infty, \label{Eq:CompactnessStrainBound} \\
		&\limsup_{h \to 0}\; && h^{1-\alpha_q^*} && \norm*{\phi^h}_{\Leb^q(\Omega)} &&+ h^{1-\alpha_{\tilde{q}}} && \norm*{\grad[h]\phi^h}_{\Leb^{\tilde{q}}(\Omega)} && && &&< \infty, \label{Eq:CompactnessRestBound}
	\end{alignat}
	\end{subequations}
	where 
	$\alpha_q := \left\{\begin{smallmatrix*}[l] 1 & q \leq 2, \\ \alpha_C + 2\frac{1 - \alpha_C}{q} & q \in (2,\infty), \\ \alpha_C & q = \infty\end{smallmatrix*}\right.$ and 
	$\alpha_q^* = 1$ for $q < \infty$ and $\alpha_q^* < \alpha_3$ arbitrary for $q = \infty$.
	
	\item We find $(v,R,z) \in \mathcal{Q}^0$, $\phi \in \SobH^1(\omega, \R^3)$ and $\psi \in \Leb^2(\omega, \SobH^1(S,\R^3))$, such that up to a subsequence (not relabeled),
	\begin{subequations} \label{Eq:ConvergencesFromCompactnessAdvanced}
	\begin{align}
		z^h & \wto z && \text{in } \Leb^2(\Omega, \R^{3\times 3}), \label{Eq:CompactnessZhConvergence} \\
		v^h - v^h(0) & \wto v - v(0) && \text{in } \SobH^2(\omega, \R^3), \\
		R^h & \wto R && \text{in } \SobH^1(\omega, \R^{3\times 3}), \\
		\phi^h &\wto \phi && \text{in } \SobH^1(\Omega, \R^3), \label{Eq:CompactnessRestConvergence}\\
		\grad[h] \phi^h &\wto \compose{\partial_1\phi}{\bar{\grad} \psi} && \text{in } \Leb^2(\Omega, \R^{3\times 3}).
	\end{align}
	\end{subequations}
	Moreover, if $(y^h, z^h) \in \mathcal{Q}^h_{(v_\mathrm{bc}, R_\mathrm{bc})}$, then $(v,R,z) \in \mathcal{Q}^0_{(v_\mathrm{bc}, R_\mathrm{bc})}$.
\end{enumerate}
\end{proposition}

\begin{proof}
From \eqref{Eq:BoundedEnergy}, we immediately obtain,
\begin{alignat*}{5}
	(*_1)\;\;\limsup_{h \to 0}\, &h^{-2} \int_\Omega W_\el\big(\grad[h]y^h (I + hz^h)^{-1}\big) &&< \infty, &&\qquad
	(*_2)\;\;\limsup_{h \to 0}\, &h^{-2} \int_\Omega W_\pl\big(I + hz^h\big) &&< \infty, \\
	(*_3)\;\;\limsup_{h \to 0}\, &h^{-\alpha_Cp} \int_\Omega H^C_\el\big(\grad[h]\sqrt{C_{y^h}}\big)&& < \infty, &&\qquad
	(*_4)\;\;\limsup_{h \to 0}\, &h^{\alpha_Rp} \int_\Omega H^R_\el\big(\grad[h]R_{y^h}\big)&& < \infty.
\end{alignat*}
$(*_2)$ implies (b) by \cref{Def:MaterialClass_WP} and the a priori estimate \eqref{Eq:AprioriEstimateEnergy} yields \eqref{Eq:CompactnessZhBound} and that $(y^h)$ has finite bending energy in the sense of \eqref{Eq:FiniteBendingEnergy}, i.e.
\begin{equation*} \tag{$*_5$}
	\limsup_{h \to 0} h^{-1}\norm*{\sqrt{C_{y^h}} - I}_{\Leb^2(\Omega)} = \limsup_{h \to 0} h^{-1}\norm*{\dist\of[\big]{\grad[h]y^h, \SO(3)}}_{\Leb^2(\Omega)} < \infty.
\end{equation*}
$(*_3)$ and $(*_4)$ yield \eqref{Eq:CompactnessStrainGradientsBound} in view of the growth condition \eqref{Ass:HE_Growth} and the polar decompositon $\grad[h]y^h = R_{y^h}\sqrt{C_{y^h}}$. Now since $p > 3$, from Morrey's and Poincaré's inequality, combined with $(*_5)$ and \eqref{Eq:CompactnessStrainGradientsBound}, we conclude the $\Leb^\infty$ bound
\begin{align*}
	\limsup_{h \to 0} h^{-\alpha_C}\norm*{\sqrt{C_{y^h}} - I}_{\Leb^\infty(\Omega)}
	&\leq \limsup_{h \to 0} h^{-\alpha_C}\left(\norm*{\sqrt{C_{y^h}} - {\textstyle\fint_\Omega \sqrt{C_{y^h}}}}_{\Leb^\infty(\Omega)} + \abs*{{\textstyle\fint_\Omega \sqrt{C_{y^h}}} -I}\right) \\
	&\leq c_1 \limsup_{h \to 0} h^{-\alpha_C} \left(\norm*{\grad \sqrt{C_{y^h}}}_{\Leb^p(\Omega)} + \norm*{\sqrt{C_{y^h}} -I}_{\Leb^2(\Omega)}\right) < \infty.
\end{align*}
Note that for $2 < q < \infty$, $\alpha_q$ is defined such that $\alpha_q q = \alpha_C(q - 2) + 2$. Thus, we have the trivial inequality
\begin{equation*}
	h^{-\alpha_q q}\norm*{\sqrt{C_{y^h}} - I}_{\Leb^q(\Omega)}^q
	\leq \left(h^{-\alpha_C}\norm*{\sqrt{C_{y^h}} - I}_{\Leb^\infty(\Omega)}\right)^{q-2} \left(h^{-1}\norm*{\sqrt{C_{y^h}} - I}_{\Leb^2(\Omega)}\right)^2,
\end{equation*}
which establishes
\begin{equation*} \tag{$*_6$}
	\limsup_{h \to 0} h^{-\alpha_q}\norm*{\dist\of[\big]{\grad[h]y^h, \SO(3)}}_{\Leb^q(\Omega)} = \limsup_{h \to 0} h^{-\alpha_q}\norm*{\sqrt{C_{y^h}} - I}_{\Leb^q(\Omega)} < \infty.
\end{equation*}
Note that for $1\leq q < 2$, $(*_6)$ follows immediately from $(*_5)$ by Hölder's inequality. Let us now define the rotation field $R^h$ according to \cref{Prop:RigidityRods}. Then, \eqref{Eq:CompactnessStrainBound} follows immediately from $(*_6)$. Following \cite[Prop.~3.6]{Neu12} we define
\begin{align*}
	v^h(x_1) &:= \fint_S y^h(0,\bar{x}) \,\dd \bar{x} + \int_0^{x_1} R^h(s)e_1 \,\dd s, \\
	\phi^h(x) &:= h^{-1} \Big( y^h(x) - v^h(x_1) - hR^h(x_1)\mathbf{\bar{x}}(x) \Big).
\end{align*}
Then, $(v^h, R^h) \in \mathcal{A}_\mathrm{rod}$, (a) is satisfied by the centering of $S$, see \eqref{Eq:ReferenceDomainCentered}, and a quick calculation shows $\grad[h]\phi^h = h^{-1}(\grad[h]y^h - R^h) - \compose*{\partial_1 R^h \mathbf{\bar{x}}}{0}$. From this, we conclude $\phi^h \in \SobW^{1,\infty}$, since $y^h \in \SobW^{2,p}(\Omega, \R^3)$ and $R^h \in \SobW^{1,\infty}(\Omega, \R^{3\times 3})$. Moreover, \eqref{Eq:CompactnessStrainBound} yields,
\begin{equation*} \tag{$*_7$}
	\limsup_{h \to 0}\, h^{1-\alpha_{\tilde{q}}} \norm*{\grad[h]\phi^h}_{\Leb^{\tilde{q}}(\Omega)} < \infty.
\end{equation*}
It remains to show the first part of \eqref{Eq:CompactnessRestBound}. We start with the case $1 \leq q < \infty$. Since $\fint_S \phi^h(0,\bar{x}) \,\dd \bar{x} = 0$, \cref{Lem:PoincareMorrey} implies
\begin{equation*}
	\norm*{\phi^h}_{\Leb^q(\Omega)} \leq c_3 \left(\norm*{\partial_1 \phi^h}_{\Leb^1(\Omega)} + \norm*{\bar{\grad}\phi^h}_{\Leb^q(\Omega)}\right) \leq c_3 \left(\norm*{\grad[h] \phi^h}_{\Leb^1(\Omega)} + h\norm*{\grad[h]\phi^h}_{\Leb^q(\Omega)}\right).
\end{equation*}
Thus, \eqref{Eq:CompactnessRestBound} follows from $(*_7)$. The case $q = \infty$ now can be obtained from the fact that Poincaré's and Morrey's inequality show that for any $\bar{q} > 3$,
\begin{equation*}
	\norm*{\phi^h}_{\Leb^\infty(\Omega)} \leq c_4 \left( \norm*{\grad[h]\phi^h}_{\Leb^{\bar q}(\Omega)} + \norm*{\phi^h}_{\Leb^1(\Omega)} \right).
\end{equation*}
The statement (d) follows immediately from the boundedness of the sequences in question, which we established in (c). The only remaining part is the identification of the limit of the sequence $\grad[h]\phi^h$. The proof is analogous to the proof of \cite[Thm.~3.1(i)]{MM03} and \cite[Thm.~3.5(a)]{Neu12}. For the readers convenience we provide the arguments here. From \eqref{Eq:CompactnessRestBound} we know that $(\grad[h]\phi^h)$ converges (up to a subsequence) weakly in $\Leb^2(\Omega, \R^{3\times 3})$. Moreover, since $\grad[h] = \compose*{\partial_1}{\frac{1}{h}\bar{\grad}}$, \eqref{Eq:CompactnessRestBound} implies that the limit is of the form $\compose*{\partial_1 \phi}{d}$ for some $d \in \Leb^2(\Omega, \R^{3 \times 2})$. Let $\bar{\phi}^h: \omega \to \R^3, \bar{\phi}^h(x_1) := \fint_S \phi^h(x_1,\bar{x})\,\dd \bar{x}$ and $\psi^h := h^{-1}(\phi^h - \bar{\phi}^h)$. Then, we can write $d$ as the weak limit of $(\bar{\grad} \psi^h)$ in $\Leb^2(\Omega, \R^{3\times 2})$. Moreover, the proof of \cref{Lem:PoincareMorrey} shows
\begin{equation*}
	\limsup_{h \to 0} \norm*{\psi^h}_{\Leb^2(\Omega)} = \limsup_{h \to 0} h^{-1}\norm*{\phi^h - \bar{\phi}^h}_{\Leb^2(\Omega)} \leq c_5\limsup_{h \to 0} h^{-1}\norm*{\bar{\grad} \phi^h}_{\Leb^2(\Omega)} < \infty.
\end{equation*}
Thus, $(\psi^h)$ converges (up to a subsequence) to some $\psi \in \Leb^2(\Omega, \R^3)$. It follows that $\psi \in \Leb^2(\omega, \SobH^1(S,\R^3))$ and $d = \bar{\grad}\psi$, since
\begin{equation*}
	\int_\Omega d_i \varphi  = \lim_{h \to 0} \int_\Omega \partial_{i+1}\psi^h \varphi = \lim_{h \to 0} \int_\Omega \psi^h \partial_{i+1}\varphi = \int_\Omega \psi \partial_{i+1}\varphi \qquad
	\text{for all } \varphi \in \Cont^\infty_c(\Omega),\; i=1,2.
\qedhere
\end{equation*}
\end{proof}
\bigskip 

\begin{remark} \label{Rem:alpha_q}~
\begin{enumerate}[(i)]
	\item The ratio behind the definition of $\alpha_q$ in \cref{Prop:Compactness} is the following: As shown in the proof of \cref{Prop:Compactness}, equi-boundedness of the energy in form of
	\eqref{Eq:BoundedEnergy} yields (thanks to the geometric rigidity estimate \cite{FJM02} and the presence of the strain gradient term $H_\el^C$) the bound
	\begin{equation*}
		\limsup_{h \to 0} \left(h^{-1} \norm*{\dist\of[\big]{\grad[h]y^h, \SO(3)}}_{\Leb^2(\Omega)} + h^{-\alpha_C} \norm*{\dist\of[\big]{\grad[h]y^h, \SO(3)}}_{\Leb^\infty(\Omega)}\right) < \infty.
	\end{equation*}
	Combined with an interpolation estimate between $\Leb^2$ and $\Leb^\infty$ we obtain for all $q \in [1,\infty]$ the bound
	\begin{equation*}
		\limsup_{h \to 0} h^{-\alpha_q} \norm*{\dist\of[\big]{\grad[h]y^h, \SO(3)}}_{\Leb^q(\Omega)} < \infty.
	\end{equation*}
	
	\item \eqref{Eq:CompactnessRestBound} clearly implies $\lim_{h \to 0} h^{1-\delta} \norm*{\phi^h}_{\Leb^\infty(\Omega)} = 0$ for any $\delta < \alpha_q$ with $q=3$. This is a critical estimate in the construction of the mutual recovery sequence. In particular, thanks to the assumption $\alpha_R < \frac{2}{3}(1 - \alpha_C)$ we have  $\alpha_C + \alpha_R < \alpha_3$ and thus,
	\begin{equation*}
		\lim_{h \to 0} h^{1-\alpha_c - \alpha_R} \norm*{\phi^h}_{\Leb^\infty(\Omega)} \to 0.
	\end{equation*}
\end{enumerate}
\end{remark}

\subsection{Lower bound; Proof of Theorem \ref{Thm:EvolutionaryGammaConvergence} (b)}
\label{Sec:LowerBound}

In this section we show \cref{Thm:EvolutionaryGammaConvergence} (b). \cite{MS13} already establishes the lower bound for the plastic part, i.e.\
\begin{align*}
	\liminf_{h \to 0} \mathcal{E}^h_\pl(z^h) &\geq \mathcal{E}^0_\pl(z), \\
	\liminf_{h \to 0} \mathcal{D}^h(z^h, \hat{z}^h) &\geq \mathcal{D}^0(z, \hat{z}),
\end{align*}
provided $z^h \wto z$ and $\hat{z}^h \wto \hat{z}$ in $\Leb^2(\Omega, \R^{3\times 3})$ (the lower bound for the energy needs that $(y^h, z^h)$ has bounded energy to be rigorous, but this is not a restriction as we shall see in \cref{Prop:LowerBoundElastic}). It remains to prove the lower bound for the elastic energy. For this, we mimic the procedure of \cite{Neu12} appealing to \cref{Prop:Compactness}. See also \cite{MM03,BNS20,BGKN22-PP}, where similar ideas are used. We start by providing a characterization of the limit of the nonlinear strain, the proof of which follows immediately from \cref{Prop:Compactness} and \cite[Lem.~4.5]{Neu12}. Recall that for $R \in \SobH^1_{\SO(3)}(\omega)$ we defined $K_R := R^T \partial_1 R$.

\begin{proposition} \label{Prop:LowerBoundStrain}
Let $(y^h, z^h) \subset \mathcal{Q}^h$ with bounded energy in the sense of \eqref{Eq:BoundedEnergy}. Let $(v^h, R^h) \in \mathcal{A}_\mathrm{rod}$ and $\phi^h \in \SobW^{1,\infty}(\Omega, \R^3)$ be as in \cref{Prop:Compactness}. Consider
\begin{subequations}
\begin{alignat}{3}
	G^h &:= h^{-1}\Big((R^h)^T \grad[h]y^h &&- I\Big) &&= \compose{K_{R^h}\mathbf{\bar{x}}}{0} + (R^h)^T\grad[h]\phi^h, \\
	E^h &:= h^{-1}\Big(\sqrt{(\grad[h]y^h)^T\grad[h]y^h} &&- I\Big) &&= h^{-1} \Big(\sqrt{C_{y^h}} - I\Big). \label{Eq:NonlinearStrainLowerBound}
\end{alignat}
\end{subequations}
Then, for all $q \in [1,\infty]$,
\begin{equation}
	\limsup_{h \to 0} h^{1 - \alpha_q} \norm*{G^h}_{\Leb^q(\Omega)} + h^{1 - \alpha_q} \norm*{E^h}_{\Leb^q(\Omega)} < \infty.
\end{equation}
Moreover, if we find $z \in \Leb^2(\Omega, \R^{3\times 3})$, $(v,R) \in \mathcal{A}_\mathrm{rod}$, $\phi \in \SobH^1(\omega, \R^3)$ and $\psi \in \Leb^2(\omega, \SobH^1(S,\R^3))$ such that the convergences \eqref{Eq:ConvergencesFromCompactnessAdvanced} hold, then 
\begin{subequations}
\begin{align}
	G^h &\wto G := \compose[\big]{K_R \mathbf{\bar{x}} + R^T\partial_1\phi}{R^T\bar{\grad} \psi} && \text{in } \Leb^2(\Omega,\R^{3\times 3}), \\
	E^h &\wto E := \sym\compose[\big]{K_R \mathbf{\bar{x}} + ae_1}{\bar{\grad} \varphi} && \text{in } \Leb^2(\Omega,\R^{3\times 3}), \label{Eq:ConvergenceNonlinearStrainLowerBound}
\end{align}
\end{subequations}
where $a:= R^T\partial_1\phi \cdot e_1 \in \Leb^2(\omega)$ and $\varphi := R^T\psi + (R^T\partial_1\phi \cdot \mathbf{\bar{x}})e_1 \in \Leb^2(\omega, \SobH^1(S, \R^3))$.
\end{proposition}

The next essential ingredient is the following lemma resulting from a careful Taylor expansion.

\begin{lemma}[Expansion] \label{Lem:CarefulExpansion}
Let $(\Phi^h), (\Psi^h) \subset \Leb^2(\Omega, \R^{3\times 3})$ bounded and $\kappa: (0,\infty) \to (0,\infty)$ with $\lim_{h \to 0} \kappa(h) = 0$. Then, we find $O^h \subset \Omega$ with $\lim_{h \to 0} \abs*{\Omega \setminus O^h} = \lim_{h \to 0} \kappa(h)h^{-2}\abs*{\Omega \setminus O^h} = 0$, such that $\det(I + h\Psi^h(x)) > 0$ for all $x \in O^h$ and
\begin{equation}
	\lim_{h \to 0} \abs*{h^{-2} \int_{O^h} W_\el\big( (I + h\Phi^h)(I + h\Psi^h)^{-1} \big) - \int_{O^h} Q_\el \big(\Phi^h - \Psi^h\big)} = 0
\end{equation}
\end{lemma}

\begin{proof}
Define $\tilde{\kappa}(h) := \kappa(h)$ if $\lim_{h \to 0} h^2\kappa(h)^{-1} = 0$ and $\tilde{\kappa}(h) := h$ else. We define the set $O^h := \set[\Big]{x \in \Omega \given \abs*{h\Phi^h(x)} + \abs*{h\Psi^h(x)} \leq \min\set*{\frac{1}{2}, \tilde{\kappa}(h)^{1/4}}}$. Then, Markov's inequality shows $\lim_{h \to 0}h^{-2}\tilde{\kappa}(h) \abs*{\Omega \setminus O^h} = 0$. Moreover, we find from the Neumann series that given $x \in O^h$,
\begin{equation*}
	\big(I + h\Psi^h(x)\big)^{-1} = I - h\Psi^h(x) + h\rho^h(x),
\end{equation*}
where $\rho^h(x) := h^{-1}\sum_{k = 2}^\infty (-h\Psi^h(x))^k$. The definition of $O^h$ shows that $(\rho^h)$ satisfies 
\begin{equation*} \tag{$*$}
	\norm*{\mathds{1}_{O^h}\rho^h}_{\Leb^2(\Omega)} \leq \tfrac{\tilde{\kappa}(h)^{1/4}}{1 - \tilde{\kappa}(h)^{1/4}} \norm*{\Psi^h}_{\Leb^2(\Omega)} \to 0, \qquad
	\norm*{\mathds{1}_{O^h}h\rho^h}_{\Leb^\infty(\Omega)} \leq \tfrac{\tilde{\kappa}(h)^{1/2}}{1 - \tilde{\kappa}(h)^{1/4}} \to 0.
\end{equation*}
Let $G^h(x) := h^{-1}\big((I + h\Phi^h(x))(I + h\Psi^h(x))^{-1} - I\big)$. Then, \eqref{Ass:W_Expansion} yields
\begin{equation*}
	\abs*{h^{-2} \int_{O^h} W_\el\big( (I + h\Phi^h)(I + h\Psi^h)^{-1} \big) - \int_{O^h} Q_\el(\Phi^h - \Psi^h)} \leq \mathrm{(I)}^h + \mathrm{(II)}^h,
\end{equation*}
where
\begin{align*}
	\mathrm{(I)}^h &:= \norm*{\mathds{1}_{O^h}G^h}_{\Leb^2(\Omega)} r_\el\big(\norm*{\mathds{1}_{O^h} hG^h}_{\Leb^\infty(\Omega)}\big), \\
	\mathrm{(II)}^h &:= \abs*{\int_{O^h} Q_\el(G^h) - Q_\el(\Phi^h - \Psi^h)}.
\end{align*}
From $(*)$ we obtain the estimates
\begin{equation*}
	\norm*{\mathds{1}_{O^h}(G^h - \Phi^h + \Psi^h)}_{\Leb^2(\Omega)} \to 0, \qquad
	\norm*{\mathds{1}_{O^h}hG^h}_{\Leb^\infty(\Omega)} \to 0.
\end{equation*}
Hence, $\mathrm{(I)}^h \to 0$ and also by the Cauchy-Schwarz inequality
\begin{equation*}
	\mathrm{(II)}^h = \abs*{\int_{O^h} (G^h - \Phi^h + \Psi^h) : \mathbb{C}_\el (G^h + \Phi^h - \Psi^h)} \to 0.
\qedhere
\end{equation*}
\end{proof}
\bigskip 

With these two contributions we are able to show the lower bound statement for the elastic energy and thus conclude \cref{Thm:EvolutionaryGammaConvergence} (b).

\begin{proposition} \label{Prop:LowerBoundElastic}
Let $(y^h, z^h) \subset \mathcal{Q}^h$ and $(v,R,z) \in \mathcal{Q}^0$ such that the convergences \eqref{Eq:ConvergenceOfEnergeticSolutions} hold (respectively \eqref{Eq:ConvergenceOfEnergeticSolutionsBC} provided $(y^h, z^h) \subset \mathcal{Q}^h_{(v_\mathrm{bc}, R_\mathrm{bc})}$ and $(v,R,z) \in \mathcal{Q}^0_{(v_\mathrm{bc}, R_\mathrm{bc})}$ for some $(v_\mathrm{bc}, R_\mathrm{bc}) \in \R^3 \times \SO(3)$). Then,
\begin{equation} \label{Eq:LowerBoundEnergy}
	\liminf_{h \to 0} \mathcal{E}^h_{\el+\pl}(y^h, z^h) \geq \mathcal{E}^0_{\el+\pl}(v,R,z).
\end{equation}
\end{proposition}

\begin{proof}
Without loss, we may assume that the left-hand side of \eqref{Eq:LowerBoundEnergy} is finite and restrict to a subsequence (not relabeled) which establishes the $\liminf$ as a limit. Then, the sequence $(y^h, z^h)$ has bounded energy and by \cref{Prop:Compactness} we may restrict to a further subsequence (still not relabeled) such that the convergences \eqref{Eq:ConvergencesFromCompactnessAdvanced} hold for some $\phi \in \SobH^1(\omega, \R^3)$ and $\psi \in \Leb^2(\omega, \SobH^1(S,\R^3))$. Since then, $\liminf_{h \to 0} \mathcal{E}^h_\pl(z^h) \geq \mathcal{E}^0_\pl(z)$ can be rigorously established as discussed at the beginning of this section and trivially,
\begin{equation*}
	\liminf_{h \to 0} h^{-\alpha_C p} \int_{\Omega} H^C_\el\big(\grad \sqrt{C_{y^h}}\big) + h^{\alpha_R p} \int_{\Omega} H^R_\el\big(\grad R_{y^h}\big) \geq 0,
\end{equation*}
in view of the characterization of the limiting elastic energy \eqref{Eq:CharacterizationLimitingEnergy} it remains to show 
\begin{equation*} \tag{$*$}
	\liminf_{h \to 0} h^{-2} \int_\Omega W_\el\big(\grad[h]y^h(I + h z^h)^{-1}\big) \geq \inf_{\chi \in \Leb^2(\omega, \mathbb{H}_\mathrm{rel})}\int_\Omega Q_\el\big(\compose{K_R\mathbf{\bar{x}}}{0} + \chi - z\big),
\end{equation*}
Applying \cref{Lem:CarefulExpansion} with $\Phi^h := E^h$ (defined as in \eqref{Eq:NonlinearStrainLowerBound}) and $\Psi^h := z^h$, provides sets $O^h \subset \Omega$ with $\abs*{\Omega \setminus O^h} \to 0$, such that
\begin{align*}
	&\liminf_{h \to 0} h^{-2} \int_\Omega W_\el\big(\grad[h]y^h(I + h z^h)^{-1}\big) 
	\overset{\eqref{Ass:W_FrameIndifference}}{=} \liminf_{h \to 0} h^{-2} \int_\Omega W_\el\big((I + hE^h)(I + h z^h)^{-1}\big) \\
	&\qquad\overset{\eqref{Ass:W_NonDegeneracy}}{\geq} \liminf_{h \to 0} h^{-2} \int_{O^h} W_\el\big((I + hE^h)(I + h z^h)^{-1}\big)
	= \liminf_{h \to 0} \int_{O^h} Q_\el\big(E^h - z^h\big).
\end{align*}
Now since by \eqref{Eq:ConvergenceNonlinearStrainLowerBound}, $E^h \wto E = \sym\compose*{K_R \mathbf{\bar{x}}}{0} + \chi$ for $\chi := \sym\compose*{ae_1}{\bar{\grad}\varphi} \in \Leb^2(\omega, \mathbb{H}_\mathrm{rel})$ and $z^h \wto z$ in $\Leb^2(\Omega, \R^{3\times 3})$, we find $\mathds{1}_{O^h}(E^h - z^h) \wto E - z$ in $\Leb^2(\Omega, \R^{3\times 3})$. Hence, we conclude $(*)$ using the weak lower semi-continuity of the quadratic functional as follows,
\begin{align*}
	&\liminf_{h \to 0} \int_{O^h} Q_\el\big(E^h - z^h\big)
	= \liminf_{h \to 0} \int_{\Omega} Q_\el\big(\mathds{1}_{O^h}(E^h - z^h)\big) \\
	&\qquad\geq \int_\Omega Q_\el\big( \compose*{K_R \mathbf{\bar{x}}}{0} + \chi - z \big)
	\geq \inf_{\chi \in \Leb^2(\omega, \mathbb{H}_\mathrm{rel})}\int_\Omega Q_\el\big(\compose{K_R\mathbf{\bar{x}}}{0} + \chi - z\big).
\qedhere
\end{align*}
\end{proof}
\bigskip 

\subsection{Mutual recovery sequence; Proof of Theorem \ref{Thm:EvolutionaryGammaConvergence} (c)}
\label{Sec:RecoverySequence}

The main novel contribution of this paper is the construction of a recovery sequence, which allows to prove \cref{Thm:EvolutionaryGammaConvergence} (c). In this section we provide this construction. As in \cite{MS13,Dav14_02} we construct the recovery sequence for the plastic strain and the deformation independently and then show convergence of the dissipation, the hardening energy and the elastic energy individually. First, we treat the plastic terms. The following proposition is due to \cite{MS13}. We use a slight variation of the statement in order to circumvent the necessity to assume that $\hat{z} - z$ is smooth. One can easily check that all the arguments used in \cite{MS13} still hold with the adjustment. We leave the details to the reader and state the result without proof:

\begin{proposition}[cf.\ {\cite[Lem.~3.6]{MS13}}] \label{Prop:LimsupInequalityPlastic}
Let $z, \hat{z} \in \Leb^2(\Omega, \R^{3\times 3})$ with $\tilde{z} := \hat{z} - z \in \Leb^2(\Omega, \R^{3\times 3}_\mathrm{dev})$ and $(z^h) \subset \Leb^2(\Omega, \R^{3\times 3})$ such that $z^h \wto z$ in $\Leb^2(\Omega, \R^{3\times 3})$ and
\begin{equation}
	\limsup_{h \to 0} \mathcal{E}^h_\pl(z^h) < \infty.
\end{equation}
Choose a sequence $(\tilde{z}^h) \subset \Cont^\infty_\mathrm{c}(\Omega, \R^{3\times 3}_\mathrm{dev})$ with $\tilde{z}^h \to \tilde{z}$ in $\Leb^2(\Omega, \R^{3\times 3})$ and $\lim_{h \to 0} h^{1/2}\norm{\tilde{z}^h}_{\Leb^\infty(\Omega)} = 0$.
Let $U^h := \set*{x \in \Omega \given \exp(h\tilde{z}^h(x))(I + hz^h(x)) \in K_\pl}$ and define
\begin{equation} \label{Eq:RecoverySequencePlastic}
	\hat{z}^h := h^{-1}\mathds{1}_{U^h} \big(\exp(h\tilde{z}^h)(I + hz^h) - I\big) + \mathds{1}_{\Omega \setminus U^h} z^h.
\end{equation}
Then, $\hat{z}^h \wto \hat{z}$ and $\hat{z}^h - z^h \to \tilde{z}$ in $\Leb^2(\Omega, \R^{3\times 3})$ and
\begin{subequations}
\begin{alignat}{2}
	\limsup_{h \to 0}\;& \mathcal{D}^h(z^h, \hat{z}^h) &&\leq \mathcal{D}^0(z,\hat{z}), \\
	\limsup_{h \to 0}\;& \Big(\mathcal{E}^h_\pl(\hat{z}^h) - \mathcal{E}^h_\pl(z^h)\Big) &&\leq \mathcal{E}^0_\pl(\hat{z}) - \mathcal{E}^0_\pl(z).
\end{alignat}
\end{subequations}
\end{proposition}

We turn to the construction of the recovery sequence for the deformation, that is, the construction of a sequence $(\hat{y}^h)$, such that
\begin{equation} \label{Eq:LimsupInequalityElastic}
	\limsup_{h \to 0} \left( \mathcal{E}^h_\el(\hat{y}^h,\hat{z}^h) -  \mathcal{E}^h_\el(y^h,z^h)\right) \leq \mathcal{E}^0_\el(\hat{v},\hat{R},\hat{z}) - \mathcal{E}^0_\el(v,R,z),
\end{equation}
for $(y^h, z^h)$, $(v,R,z)$ and $(\hat{v},\hat{R},\hat{z})$ as in \cref{Thm:EvolutionaryGammaConvergence} (c) and $(\hat{z}^h)$ as constructed above. We establish \eqref{Eq:LimsupInequalityElastic} by showing separately,
\begin{subequations} \label{Eq:LimsupInequalityElasticSeparated}
\begin{alignat}{2}
	&\limsup_{h \to 0} h^{-2}\int_\Omega W_\el \big( \grad[h]\hat{y}^h(I + h\hat{z}^h)^{-1}\big) - W_\el \big( \grad[h]y^h(I + hz^h)^{-1}\big) \leq \mathcal{E}^0_\el(\hat{v},\hat{R},\hat{z}) - \mathcal{E}^0_\el(v,R,z), \label{Eq:LimsupInequalityElasticBending}\\
	&\limsup_{h \to 0} h^{-p \alpha_C} \int_\Omega H^C_\el\big(\grad\sqrt{C_{\hat{y}^h}}\big) - H^C_\el\big(\grad\sqrt{C_{y^h}}\big) \leq 0, \label{Eq:LimsupInequalityElasticStrainGradientC}\\
	&\limsup_{h \to 0} h^{p \alpha_R} \int_\Omega H^R_\el\big(\grad R_{\hat{y}^h}\big) - H^R_\el\big(\grad R_{y^h}\big) \leq 0. \label{Eq:LimsupInequalityElasticStrainGradientR}
\end{alignat}
\end{subequations}
The following lemma is used to construct the rotational part of the deformation.

\begin{lemma} \label{Lem:MultiplicativeSmootheningOfRotations}
Let $\kappa: (0,\infty) \to [0,\infty)$ with $\lim_{h \to 0} \kappa(h) = 0$, $R \in \SobH^1_{\SO(3)}(\omega, \R^{3\times 3})$ and $(R^h) \subset \SobH^1_{\SO(3)}(\omega, \R^{3\times 3})$ with $R^h \wto R$ in $\SobH^1(\omega,\R^{3\times 3})$. Then, given another $\hat{R} \in \SobH^1_{\SO(3)}(\omega, \R^{3\times 3})$, we find a sequence $\hat{R}^h \subset \SobH^1_{\SO(3)}(\omega, \R^{3\times 3})$, such that
\begin{subequations}
\begin{align}
	\hat{R}^h &\wto \hat{R} && \text{in } \SobH^1(\omega, \R^{3\times 3}), \\
	K_{\hat{R}^h} - K_{R^h} &\to K_{\hat{R}} - K_R && \text{in } \Leb^2(\omega, \R^{3\times 3}),
\end{align}
\begin{equation} \label{Eq:ScalingRotation}
	\lim_{h \to 0} \kappa(h) \left(\norm*{K_{\hat{R}^h} - K_{R^h}}_{\Leb^\infty(\omega)} + \norm*{K_{\tilde{R}^h}}_{\Leb^\infty(\omega)} + \norm*{\partial_1 K_{\tilde{R}^h}}_{\Leb^\infty(\omega)}\right) = 0,
\end{equation}
\end{subequations}
where $\tilde{R}^h := \hat{R}^h (R^h)^T$. Moreover, if $R^h(0)=R(0)=\hat{R}(0) = R_\mathrm{bc} \in \SO(3)$, then also $\hat{R}^h(0) = R_\mathrm{bc}$.
\end{lemma}

\begin{proof}
Let $\tilde{R} := \hat{R}R^T$. By an approximation argument, we may choose $(\tilde{K}^h) \subset \Cont^\infty(\overline{\omega}, \R^{3\times 3}_{\skewMat})$ satisfying
\begin{equation*}
	\tilde{K}^h \to K_{\tilde{R}} \text{ in } \Leb^2(\Omega, \R^{3\times 3}), \qquad
	\lim_{h \to 0}\kappa(h)\norm*{\tilde{K}^h}_{\Leb^\infty(\omega)} = 0.
\end{equation*}
Then, by solving a linear ODE, we find $\tilde{R}^h \in \Cont^\infty(\overline{\omega}, \R^{3\times 3}) \cap \Cont(\overline{\omega},\SO(3))$, such that $\tilde{K}^h = K_{\tilde{R}^h}$ and $\tilde{R}^h(0) = \tilde{R}(0)$, see \cite[Lem.~2.3]{Neu12} for details. Moreover, since $\tilde{R}$ is given uniquely by $K_{\tilde{R}}$ and $\tilde{R}(0)$, we get $\tilde{R}^h \to \tilde{R}$ in $\SobH^1(\omega, \R^{3\times 3})$. Set $\hat{R}^h := \tilde{R}^h R^h$. Then, one easily computes the identities
\begin{equation*}
	K_{\hat{R}^h} - K_{R^h} = (R^h)^T K_{\tilde{R}^h} R^h, \qquad
	K_{\hat{R}} - K_R = R^T K_{\tilde{R}} R.
\end{equation*}
From these, the claim is easily checked.
\end{proof}

The following proposition establishes a general construction of a deformation via a corrected standard Cosserat ansatz. The recovery sequence shall be of this form.

\begin{proposition} \label{Prop:ConstructionRecoverySequenceDeformation}
Consider $(y^h, z^h) \subset \mathcal{Q}^h$ with bounded energy in the sense of \eqref{Eq:BoundedEnergy} and let $(v,R,z) \in \mathcal{Q}^0$. Let $(v^h, R^h) \in \mathcal{A}_\mathrm{rod}$ and $\phi^h \in \SobW^{1,\infty}(\Omega, \R^{3\times 3})$ be as in \cref{Prop:Compactness} and assume that $(v^h, R^h, \phi^h)$ converges to $(v,R,\phi,\psi)$ in the sense of \eqref{Eq:ConvergencesFromCompactnessAdvanced}. Let $(\hat{v},\hat{R}) \in \mathcal{A}_\mathrm{bc}$, $\hat{a} \in \Leb^2(\omega)$ and $\hat{\varphi} \in \Leb^2(\omega, \SobH^1(S, \R^3))$. Then, there exist sequences as follows:
\begin{enumerate}[(a)]
	\item $(\hat{a}^h) \subset \Cont^\infty_\mathrm{c}(\omega)$, such that 
	\begin{equation*}
		\hat{a}^h \to \hat{a} \text{ in } \Leb^2(\omega) \quad\text{and}\quad
		\lim_{h \to 0} \left(h^{1 - \alpha_C - \alpha_R} + h^{2\frac{1 - \alpha_C}{p}}\right) \norm{\hat{a}^h}_{\SobW^{1,\infty}(\omega)} = 0.
	\end{equation*}
	
	\item $(\hat{\phi}^h) \subset \Cont^\infty_\mathrm{c}(\Omega, \R^3)$ with $\hat{\phi}^h(0,\bar{x}) = 0$ for a.e.\ $\bar{x} \in S$, such that
	\begin{align*}
		&\hat{\phi}^h \to 0 \text{ in } \SobH^1(\Omega, \R^3), \quad 
		\grad[h]\hat{\phi}^h \to \compose{0}{\hat{R}\bar{\grad}\hat{\varphi} - \hat{R}R^T\bar{\grad}\psi} \text{ in } \Leb^2(\Omega, \R^{3\times 3}) \quad\text{and} \\
		& \lim_{h \to 0} h^{1 - \alpha_C - \alpha_R}\left(\norm{\grad[h]\hat{\phi}^h}_{\Leb^\infty(\Omega)} + \norm{\grad\grad[h]\hat{\phi}^h}_{\Leb^\infty(\Omega)}\right) = 0
	\end{align*}
	
	\item $(\phi^h_*) \subset \Cont^\infty(\overline{\omega}, \R^3)$ with $\phi^h_*(0) = 0$, such that
	\begin{equation*}
		\phi^h_* \to \phi \text{ in } \SobH^1(\omega, \R^3) \quad\text{and}\quad
		\lim_{h \to 0} h^{1 - \alpha_C - \alpha_R}\norm{\phi^h_*}_{\SobW^{2,\infty}(\omega)} = 0,
	\end{equation*}
	
	\item $(\hat{R}^h) \subset \SobH^1_{\SO(3)}(\omega, \R^{3\times 3})$ as in \cref{Lem:MultiplicativeSmootheningOfRotations} with 
	\begin{equation*}
		\kappa(h) := \norm{\phi^h - \phi^h_*}_{\Leb^2(\Omega)} + h^{1 - \alpha_C - \alpha_R}\left(1 + \norm*{\phi^h_*}_{\Leb^\infty(\omega)} + \norm*{\phi^h}_{\Leb^\infty(\omega)}\right) + h^{2\frac{1 - \alpha_C}{p}} + h^{\alpha_R}.
	\end{equation*}
\end{enumerate}
Moreover, we introduce $\hat{v}^h(x_1) := \hat{v}(0) + \int_0^{x_1} \hat{R}^h(s)e_1 \,\dd s$, $x_1 \in \omega$, $\tilde{R}^h := \hat{R}^h (R^h)^T$ and $\tilde{R} := \hat{R}R^T$. Define the deformation,
\begin{multline}
	\hat{y}^h(x) := \hat{v}^h(x_1) + h\hat{R}^h(x_1)\mathbf{\bar{x}}(x) \\
	+ h \left(\tilde{R}^h(x_1)(\phi^h(x) - \phi^h_*(x_1)) + \hat{\phi}^h(x) + \int_0^{x_1} \hat{a}^h(s)\hat{R}^h(s)e_1 \,\dd s \right).
\end{multline}
Then, $(\hat{y}^h) \subset \SobW^{2,p}(\Omega, \R^3)$ satisfies $\hat{y}^h \to \hat{v}$ in $\SobH^1(\Omega, \R^3)$ and $\grad[h]\hat{y}^h \to \hat{R}$ in $\Leb^2(\Omega, \R^{3\times 3})$. Recall $G^h, G \in \Leb^2(\Omega, \R^{3\times 3})$ from \cref{Prop:LowerBoundStrain} and consider $\hat{G}^h := h^{-1}((\hat{R}^h)^T\grad[h]\hat{y}^h - I)$. Then,
\begin{subequations}
\begin{align}
	\hat{G}^h &\wto \hat{G} := \compose[\big]{K_{\hat{R}}\mathbf{\bar{x}} + \hat{a}e_1}{\bar{\grad}\hat{\varphi}} && \text{in } \Leb^2(\Omega,\R^{3\times 3}), \\
	\hat{G}^h - G^h &\to \hat{G} - G && \text{in } \Leb^2(\Omega,\R^{3\times 3}),
\end{align}
\end{subequations}
Furthermore, $\Delta^h$ and $\hat{\Delta}^h$ defined by
\begin{equation}
	\grad[h]\hat{y}^h = \tilde{R}^h(\grad[h]y^h + h\Delta^h), \qquad
	\grad[h]\hat{y}^h = \tilde{R}^h(I + h\hat{\Delta}^h)\grad[h]y^h,
\end{equation}
satisfy $\Delta^h \to R(\hat{G} - G)$ and $\hat{\Delta}^h \to R(\hat{G} - G)R^T$ in $\Leb^2(\Omega, \R^{3\times 3})$, and
\begin{subequations} \label{Eq:EstimatesDeltah}
\begin{alignat}{2}
	&\lim_{h \to 0} h^{1 - \alpha_C - \alpha_R} \norm*{\Delta^h}_{\Leb^\infty(\Omega)} + h^{1 - \alpha_C} \norm*{\grad \Delta^h}_{\Leb^p(\Omega)} &&= 0, \\
	&\lim_{h \to 0} h^{1 - \alpha_C - \alpha_R} \norm*{\hat{\Delta}^h}_{\Leb^\infty(\Omega)} &&= 0.
\end{alignat}
\end{subequations}
\end{proposition}

\begin{proof}
Recall $\alpha_q$, $q \in [1,\infty]$ from \cref{Prop:Compactness} and \cref{Rem:alpha_q}. Then, the proposed sequences exist by standard approximation arguments and the estimates obtained in \cref{Prop:Compactness}. Note that if $\hat{R}(\hat\varphi - R^T\psi)$ is smooth, we can choose
$\hat{\varphi}^h := h\hat{R}(\hat\varphi - R^T\psi) - \fint_S h\hat{R}(\hat\varphi - R^T\psi)$. An elementary calculation using $\hat{R}^h = \tilde{R}^hR^h$ shows
\begin{equation*}
	\hat{G}^h = \compose*{K_{\hat{R}^h}\mathbf{\bar{x}} - (R^h)^T\partial_1\phi^h_* + (R^h)^TK_{\tilde{R}^h}(\phi^h - \phi^h_*) + \hat{a}^h e_1}{0} + (\hat{R}^h)^T \grad[h]\hat{\phi}^h + (R^h)^T\grad[h]\phi^h.
\end{equation*}
Note that by construction $\norm{K_{\tilde{R}^h}(\phi^h - \phi^h_*)}_{\Leb^2(\Omega)} \leq \kappa(h)\norm{K_{\tilde{R}^h}}_{\Leb^\infty(\omega)} \to 0$. Thus, we obtain from \cref{Lem:MultiplicativeSmootheningOfRotations},
\begin{align*}
	\hat{G}^h - G^h &= \compose*{(K_{\hat{R}^h} - K_{R^h})\mathbf{\bar{x}} - (R^h)^T\partial_1\phi^h_* + (R^h)^T K_{\tilde{R}^h} (\phi^h - \phi^h_*) + \hat{a}^h e_1}{0} + (\hat{R}^h)^T \grad[h]\hat{\phi}^h \\
	&\to \compose*{K_{\hat{R}} - K_R) \mathbf{\bar{x}} - R^T\partial_1\phi + \hat{a} e_1}{\bar{\grad}\hat{\varphi} - R^T\bar{\grad}\psi} = \hat{G} - G,
\end{align*}
strongly in $\Leb^2(\Omega, \R^{3\times 3})$. From this, we immediately obtain $\hat{G}^h \wto \hat{G}$ in $\Leb^2(\Omega, \R^{3\times 3})$ by \cref{Prop:LowerBoundStrain}, as well as $\hat{y}^h \to \hat{v}$ in $\SobH^1(\Omega, \R^3)$ and $\grad[h]y^h \to \hat{R}$ in $\Leb^2(\Omega, \R^{3\times 3})$. Furthermore, we have
\begin{equation*}
	\Delta^h = R^h(\hat{G}^h - G^h), \qquad
	\hat{\Delta}^h = \Delta^h (\grad[h]y^h)^{-1}.
\end{equation*}
Since by \cref{Prop:Compactness}, $\norm*{\dist\of[\big]{\grad[h]y^h, \SO(3)}}_{\Leb^\infty(\Omega)} \to 0$, it follows that $(\grad[h]y^h)^{-1}$ is bounded in $\Leb^\infty(\Omega, \R^{3\times 3})$ and converging in $\Leb^2(\Omega, \R^{3\times 3})$ to $R^T$. Hence, $\Delta^h \to R(\hat{G} - G)$ and $\hat{\Delta}^h \to R(\hat{G} - G)R^T$ in $\Leb^2(\Omega, \R^{3\times 3})$. Moreover, let $\rho := 1 - \alpha_C - \alpha_R$. Then, in view of $K_{\hat{R}^h} - K_{R^h} = (R^h)^T K_{\tilde{R}^h} R^h$ the representation of $\hat{G}^h - G^h$ above yields,
\begin{equation*}
	\Delta^h = \compose*{K_{\tilde{R}^h}R^h\mathbf{\bar{x}} - \partial_1\phi^h_* + K_{\tilde{R}^h} (\phi^h - \phi^h_*) + R^h\hat{a}^h e_1}{0} + (\tilde{R}^h)^T \grad[h]\hat{\phi}^h.
\end{equation*}
Hence, we obtain the estimate,
\begin{equation*}
	h^\rho\abs*{\Delta^h} + h^\rho \abs*{\hat{\Delta}^h} \leq c \left(
		\kappa(h)\abs*{K_{\tilde{R}^h}} 
		+ h^\rho\abs*{\partial_1\phi^h_*} 
		+ h^\rho\abs*{\hat{a}^h} 
		+ h^\rho\abs*{\grad[h]\hat{\phi}^h},
	\right),
\end{equation*}
and the right-hand side converges to $0$ in $\Leb^\infty(\Omega)$ by construction of the sequences. Finally, we calculate
\begin{align*}
	\partial_i\Delta^h &= \compose*{
	\begin{aligned}
		&(\partial_i K_{\tilde{R}^h})R^h\mathbf{\bar{x}} + K_{\tilde{R}^h}(\partial_i R^h)\mathbf{\bar{x}} + K_{\tilde{R}^h} R^h(\partial_i \mathbf{\bar{x}}) \\
			&- \partial_{i1}\phi^h_* \\
			&+ (\partial_i K_{\tilde{R}^h})(\phi^h - \phi^h_*) + K_{\tilde{R}^h} (\partial_i \phi^h - \partial_i\phi^h_*) \\
			&+ (\partial_i R^h)\hat{a}^he_1 + R^h(\partial_i \hat{a}^h)e_1
	\end{aligned}
	}{0}
	+ (\partial_i \tilde{R}^h)^T \grad[h]\hat{\phi}^h + (\tilde{R}^h)^T (\partial_i \grad[h]\hat{\phi}^h).
\end{align*}
Thanks to \eqref{Eq:ScalingRotation} and using the definition of the sequences combined with \cref{Prop:Compactness}, it is not hard to check that $h^{1-\alpha_C}\norm*{\grad \Delta^h}_{\Leb^p(\Omega)} \to 0$, . Especially, since $\grad[h]\hat{y}^h = \tilde{R}^h(\grad[h]y^h + h\Delta^h)$, we find $\hat{y}^h \in \SobW^{2,p}(\Omega, \R^3)$.
\end{proof}

With this construction, we are able to show \eqref{Eq:LimsupInequalityElasticSeparated}. We start with the estimate for the elastic energy.

\begin{proposition} \label{Prop:LimsupInequlityElasticBending}
Consider the situation of \cref{Prop:ConstructionRecoverySequenceDeformation} and let $(\hat{z}^h) \subset \Leb^2(\Omega, \R^{3\times 3})$ be as in \cref{Prop:LimsupInequalityPlastic}. Then,
\begin{multline}
	\limsup_{h \to 0} \left(h^{-2} \int_\Omega W_\el\big(\grad[h] \hat{y}^h(I + h\hat{z}^h)^{-1}\big) - h^{-2}\int_\Omega W_\el\big(\grad[h] y^h(I + hz^h)^{-1}\big)\right) \\
	\leq \int_\Omega Q_\el\big(\hat{G} - \hat{z}\big) - \int_\Omega Q_\el\big(G - z\big).
\end{multline}
\end{proposition}

\begin{proof}
Given sets $O^h \subset \Omega$, which we define below, we separate the left-hand side into a \enquote{good} part and a \enquote{bad} part
\begin{equation*}
	h^{-2} \int_\Omega W_\el\big(\grad[h] \hat{y}^h(I + h\hat{z}^h)^{-1}\big) - h^{-2}\int_\Omega W_\el\big(\grad[h] y^h(I + hz^h)^{-1}\big)
	 = g^h + b^h,
\end{equation*}
where
\begin{align*}
	g^h :&= h^{-2} \int_{O^h} W_\el\big(\grad[h] \hat{y}^h(I + h\hat{z}^h)^{-1}\big) - h^{-2}\int_{O^h} W_\el\big(\grad[h] y^h(I + hz^h)^{-1}\big) \\
	&= h^{-2} \int_{O^h} W_\el\big((I + h\hat{G}^h)(I + h\hat{z}^h)^{-1}\big) - h^{-2}\int_{O^h} W_\el\big((I + hG^h)(I + hz^h)^{-1}\big) \\
	b^h :&= h^{-2} \int_{\Omega \setminus O^h} W_\el\big(\grad[h] \hat{y}^h(I + h\hat{z}^h)^{-1}\big) - W_\el\big(\grad[h] y^h(I + hz^h)^{-1}\big).
\end{align*}
Recall from \cref{Prop:ConstructionRecoverySequenceDeformation} that $\norm*{h\hat{\Delta}^h}_{\Leb^\infty(\Omega)} \to 0$ and $\norm*{\hat{\Delta}^h}_{\Leb^2(\Omega)}$ is bounded. Set $I + hA^h := (I + hz^h)(I + h\hat{z}^h)^{-1}$. Recall \eqref{Eq:RecoverySequencePlastic} and note that
\begin{equation*}
	I + hA^h = \begin{cases} \exp(-h\tilde{z}^h) & \text{on } U^h, \\ I & \text{on } \Omega \setminus U^h, \end{cases}
\end{equation*}
and thus also $\norm*{hA^h}_{\Leb^\infty(\Omega)} \to 0$ and $\norm*{A^h}_{\Leb^2(\Omega)}$ is bounded. We seek to apply the expansion \eqref{Ass:W_Expansion} to treat the \enquote{good} part $g^h$ and show that the \enquote{bad} part $b^h$ vanishes. Therefore, we define $O^h$ according to \cref{Lem:CarefulExpansion}, such that $\abs*{\Omega \setminus O^h} \to 0$, $\left(\norm*{hA^h}_{\Leb^\infty(\Omega)} + \norm*{h\hat{\Delta}^h}_{\Leb^\infty(\Omega)}\right)h^{-2}\abs*{\Omega \setminus O^h} \to 0$ and
\begin{align*}
	\limsup_{h \to 0} g^h 
	&= \limsup_{h \to 0} \left( \int_{O^h} Q_\el(\hat{G}^h - \hat{z}^h) - \int_{O^h} Q_\el(G^h - z^h) \right) \\
	&= \limsup_{h \to 0} \int_{O^h} (\hat{G}^h - G^h + \hat{z}^h - z^h) : \mathbb{C}_\el (\hat{G}^h + G^h - \hat{z}^h - z^h).
\end{align*}
Since $\hat{G}^h - G^h \to \hat{G} - G$ and $\hat{z}^h - z^h \to \hat{z} - z$ strongly in $\Leb^2(\Omega, \R^{3 \times 3})$, the latter is a product of a strongly and a weakly converging sequence. Thus, we can pass to the limit and obtain
\begin{equation*}
	\limsup_{h \to 0} g^h = \int_\Omega (\hat{G} - G + \hat{z} - z) : \mathbb{C}_\el (\hat{G} + G - \hat{z} - z) = \int_\Omega Q_\el\big(\hat{G} - \hat{z}\big) - Q_\el\big(G - z\big).
\end{equation*}
Note that we use that $\mathds{1}_{O^h}(\hat{G}^h + G^h - \hat{z}^h - z^h) \wto (\hat{G} + G - \hat{z} - z)$, which is a consequence of $(\mathds{1}_{O^h})$ being bounded in $\Leb^\infty(\Omega)$ and $\mathds{1}_{O^h} \to 1$ in $\Leb^2(\Omega)$. To argue on the \enquote{bad} part, we utilize the representation $\grad[h] \hat{y}^h = \tilde{R}^h(I + h\hat{\Delta}^h)\grad[h]y^h$.
Let $F^h_\el := \grad[h] y^h(I + hz^h)^{-1}$. From \eqref{Ass:W_Continuity} we obtain
\begin{align*}
	\limsup_{h \to 0} b^h 
	&= \limsup_{h \to 0} h^{-2}\left( \int_{\Omega \setminus O^h} W_\el\big((I + h\hat{\Delta}^h)F^h_\el(I + hA^h)\big) - W_\el\big(F^h_\el\big)\right) \\
	&\leq L\limsup_{h \to 0} \left( \int_{\Omega \setminus O^h} h^{-2}\Big( W\big(F^h_\el\big) + 1 \Big)\Big( \abs*{h\hat{\Delta}^h} + \abs*{hA^h} \Big) \right) \\
	&\leq L\limsup_{h \to 0} \begin{aligned}[t]
		\bigg(\Big(\mathcal{E}^h_\el(y^h, z^h) + h^{-2}\abs*{\Omega \setminus O^h}\Big) \Big( \norm[\big]{h\hat{\Delta}^h}_{\Leb^\infty(\Omega)} + \norm[\big]{hA^h}_{\Leb^\infty(\Omega)} \Big)\bigg) = 0,
	\end{aligned}
\end{align*}
which finishes the proof.
\end{proof}

Next, we show that the strain gradient terms vanish in the limit along the recovery sequence.

\begin{proposition} \label{Prop:LimsupInequlityElasticStrainGradient}
Consider the situation of \cref{Prop:ConstructionRecoverySequenceDeformation} and let $(\hat{z}^h) \subset \Leb^2(\Omega, \R^{3\times 3})$ as in \cref{Prop:LimsupInequalityPlastic}. Then, \eqref{Eq:LimsupInequalityElasticStrainGradientC} and \eqref{Eq:LimsupInequalityElasticStrainGradientR} hold.
\end{proposition}

\begin{proof}
\stepemph{Step 1 -- Proof of \eqref{Eq:LimsupInequalityElasticStrainGradientC}:}
From \eqref{Ass:HE_Continuity} and Hölder's inequality, we infer
\begin{align*}
	&h^{-\alpha_C p} \int_\Omega H_\el^C\big(\grad \sqrt{C_{\hat{y}^h}}\big) - H_\el^C \big(\grad \sqrt{C_{y^h}}\big) \\
	&\quad\leq c_1 \int_\Omega \left( \abs[\Big]{h^{-\alpha_C}\grad \sqrt{C_{\hat{y}^h}}}^{p-1} + \abs[\Big]{h^{-\alpha_C}\grad \sqrt{C_{y^h}}}^{p-1} \right) h^{-\alpha_C}\abs[\Big]{\grad \sqrt{C_{\hat{y}^h}} - \grad \sqrt{C_{y^h}}} \\
	&\quad\leq c_1 \left( \norm[\Big]{h^{-\alpha_C}\grad \sqrt{C_{\hat{y}^h}}}_{\Leb^p(\Omega)}^{p-1} + \norm[\Big]{h^{-\alpha_C}\grad \sqrt{C_{y^h}}}_{\Leb^p(\Omega)}^{p-1} \right) h^{-\alpha_C} \norm[\Big]{\grad \sqrt{C_{\hat{y}^h}} - \grad \sqrt{C_{y^h}}}_{\Leb^p(\Omega)}.
\end{align*}
By \cref{Prop:Compactness}, we know $\limsup_{h \to 0} \norm[\Big]{h^{-\alpha_C}\grad \sqrt{C_{y^h}}}_{\Leb^p(\Omega)}^{p-1} < \infty$. Thus, to conclude the claim it remains to show
\begin{equation*}
	\lim_{h \to 0} h^{-\alpha_C}\norm[\Big]{\grad \sqrt{C_{\hat{y}^h}} - \grad \sqrt{C_{y^h}}}_{\Leb^p(\Omega)} = 0.
\end{equation*}
From the representation $\grad[h] \hat{y}^h = \tilde{R}^h(\grad[h]y^h + h\Delta^h)$, defined in \cref{Prop:ConstructionRecoverySequenceDeformation}, we infer,
\begin{align*}
	C_{\hat{y}^h} - C_{y^h} &= 2h \sym\big( (\Delta^h)^T \grad[h]y^h \big) + h^2 (\Delta^h)^T \Delta^h, \\
	\partial_i (C_{\hat{y}^h} - C_{y^h}) &= 2h\sym\big( (\partial_i \Delta^h)^T \grad[h]y^h + (\Delta^h)^T (\partial_i\grad[h]y^h) \big) + 2h^2 \sym \big( (\partial_i\Delta^h)^T \Delta^h \big).
\end{align*}
We claim that
\begin{equation*}
	\limsup_{h \to 0} h^{-\alpha_C} \norm*{C_{\hat{y}^h} - I}_{\Leb^\infty(\Omega)} < \infty.
\end{equation*}
Indeed, for $C_{\hat{y}^h}$ replaced by $C_{y^h}$ this follows from \cref{Prop:Compactness,Lem:EquivalenceRestTerms} and then, we may use the identity for $C_{\hat{y}^h} - C_{y^h}$ together with \eqref{Eq:EstimatesDeltah} to establish the claimed estimate.
Thus, we may apply \cref{Lem:EquivalenceRestTerms} and it suffices to establish
\begin{equation*}
	\lim_{h \to 0} h^{-\alpha_C}\norm[\Big]{\grad C_{\hat{y}^h} - \grad C_{y^h}}_{\Leb^p(\Omega)} = 0.
\end{equation*}
But, using the identity for $\partial_i (C_{\hat{y}^h} - C_{y^h})$ above, this is a consequence of the estimates \eqref{Eq:EstimatesDeltah} and the fact that $\limsup_{h \to 0} h^{\alpha_R}\norm*{\grad\grad[h]y^h}_{\Leb^p(\Omega)} < \infty$ by another application of \cref{Lem:EquivalenceRestTerms}.
\medskip

\stepemph{Step 2 -- Proof of \eqref{Eq:LimsupInequalityElasticStrainGradientR}:}
Analogously to the previous step, \eqref{Ass:HE_Continuity} and Hölder's inequality yield
\begin{align*}
	&h^{\alpha_R p} \int_\Omega H_\el^R\big(\grad R_{\hat{y}^h}\big) - H_\el^R \big(\grad R_{y^h}\big) \\
	&\quad = h^{\alpha_R p} \int_\Omega H_\el^R\big((\hat{R}^h)^T\grad R_{\hat{y}^h}\big) - H_\el^R \big((R^h)^T\grad R_{y^h}\big) \\
	&\quad\leq c_1 \left( \norm*{h^{\alpha_R}\grad R_{\hat{y}^h}}_{\Leb^p(\Omega)}^{p-1} + \norm*{h^{\alpha_R}\grad R_{y^h}}_{\Leb^p(\Omega)}^{p-1} \right) h^{\alpha_R} \norm*{(\hat{R}^h)^T\grad R_{\hat{y}^h} - (R^h)^T\grad R_{y^h}}_{\Leb^p(\Omega)}.
\end{align*}
It remains to show
\begin{equation*}
	\lim_{h \to 0} h^{\alpha_R} \norm*{(\hat{R}^h)^T\grad R_{\hat{y}^h} - (R^h)^T\grad R_{y^h}}_{\Leb^p(\Omega)} = 0.
\end{equation*}
In view of \cref{Lem:EquivalenceRestTerms}, it suffices to establish
\begin{equation*}
	\lim_{h \to 0} h^{\alpha_R} \norm*{(\hat{R}^h)^T\grad\grad[h]\hat{y}^h - (R^h)^T\grad\grad[h] y^h}_{\Leb^p(\Omega)} = 0.
\end{equation*}
From the representation $\grad[h] \hat{y}^h = \tilde{R}^h(\grad[h] y^h + h\Delta^h)$, we obtain
\begin{align*}
	\partial_i \grad[h]\hat{y}^h &= (\partial_i \tilde{R}^h)(\grad[h]y^h + h\Delta^h) + \tilde{R}^h(\partial_i \grad[h]y^h + h \partial_i \Delta^h), \\
	(\hat{R}^h)^T\partial_i \grad[h]\hat{y}^h - (R^h)^T\partial_i \grad[h]y^h &= (\hat{R}^h)^T(\partial_i \tilde{R}^h)(\grad[h]y^h + h\Delta^h) + h(R^h)^T \partial_i \Delta^h.
\end{align*}
Thus, the claim follows from the convergences
\begin{equation*}
	h^{\alpha_R} \norm*{\partial_1 \tilde{R}^h}_{\Leb^p(\Omega)} \to 0, \qquad
	h \norm*{\grad \Delta^h}_{\Leb^p(\Omega)} \to 0.
\qedhere
\end{equation*}
\end{proof}
\medskip 

\begin{remark}
The critical point in the argument above is that the term $h^{1-\alpha_C}(\Delta^h)^T(\partial_i \grad[h]y^h)$ in the identity for $\partial_i (C_{\hat{y}^h} - C_{y^h})$ vanishes. We conclude this by appealing to $h^{1-\alpha_C-\alpha_R}\norm*{\Delta^h}_{\Leb^\infty(\Omega)} \to 0$ which we in turn, as discussed in \cref{Rem:alpha_q}, obtain from the assumption $\alpha_R < \frac{2}{3}(1- \alpha_C)$.
\end{remark}

Now, \cref{Thm:EvolutionaryGammaConvergence} (c) is a straight-forward corollary of \cref{Prop:LimsupInequalityPlastic,Prop:LimsupInequlityElasticBending,Prop:LimsupInequlityElasticStrainGradient}. For the readers convenience we carry out the arguments.

\begin{proof}[Proof of \cref{Thm:EvolutionaryGammaConvergence} (c)]
Let $(y^h, z^h) \subset \mathcal{Q}^h$ satisfy \eqref{Eq:BoundedEnergy} and $(v,R,z) \in \mathcal{Q}^0$ be such that the convergences \eqref{Eq:ConvergenceOfEnergeticSolutions} hold. Consider $(v^h, R^h) \in \mathcal{A}_\mathrm{rod}$ and $\phi^h \in \SobW^{1,\infty}(\Omega, \R^{3\times 3})$ as in \cref{Prop:Compactness}. By \cref{Prop:Compactness}, for any subsequence, we find a further subsequence (not relabeled), such that the convergences \eqref{Eq:ConvergencesFromCompactnessAdvanced} hold for some $\phi \in \SobH^1(\omega, \R^3)$ and $\psi \in \Leb^2(\omega, \SobH^1(S, \R^3))$. We construct the recovery sequence for this subsequence. Let $(\hat{v},\hat{R},\hat{z}) \in \Leb^2(\Omega, \R^{3\times 3})$. By \eqref{Eq:CharacterizationLimitingEnergy} we find $\hat{a} \in \Leb^2(\omega)$ and $\hat{\varphi} \in \Leb^2(\omega, \SobH^1(S, \R^3))$, such that
\begin{equation*}
	\int_\Omega Q_\el\Big(\compose*{K_{\hat{R}} \mbox{$\mathbf{\bar{x}}$} + \hat{a}e_1}{\bar{\grad}\hat{\varphi}} - \hat{z}\Big) = \mathcal{E}^0_\el(\hat{v},\hat{R},\hat{z}).
\end{equation*}
Define $(\hat{y}^h, \hat{z}^h)$ as in \cref{Prop:LimsupInequalityPlastic,Prop:ConstructionRecoverySequenceDeformation}. Note that without loss of generality we may assume that $\hat{z} - z \in \Leb^2(\Omega, \R^{3\times 3}_\mathrm{dev})$, since otherwise $\mathcal{D}^0(z,\hat{z}) = \infty$. Then, \cref{Prop:LimsupInequalityPlastic,Prop:LimsupInequlityElasticBending,Prop:LimsupInequlityElasticStrainGradient} yield
\begin{align*}
	&\limsup_{h \to 0} \Big(\mathcal{E}_{\el+\pl}^h(\hat{y}^h, \hat{z}^h) + \mathcal{D}^h(z^h, \hat{z}^h) - \mathcal{E}_{\el+\pl}^h(y^h, z^h)\Big) \\
	&\qquad\leq \int_\Omega Q_\el\big(\hat{G} - \hat{z}\big) + \mathcal{E}^0_\pl(\hat{z}) + \mathcal{D}^0(z,\hat{z}) - \int_\Omega Q_\el\big(G - z\big) - \mathcal{E}^0_\pl(z) \\
	&\qquad = \mathcal{E}^0_{\el+\pl}(\hat{v},\hat{R},\hat{z}) + \mathcal{D}^0(z,\hat{z}) - \int_\Omega Q_\el\big( \compose*{K_R\mbox{$\mathbf{\bar{x}}$} + ae_1}{\bar{\grad}\varphi} - z\big) - \mathcal{E}^0_\pl(z),
\end{align*}
where $G$ and $\hat{G}$ as in \cref{Prop:ConstructionRecoverySequenceDeformation}. This shows \eqref{Eq:MututalLimsupInequality}, since by \eqref{Eq:CharacterizationLimitingEnergy}
\begin{equation*}
	\int_\Omega Q_\el\big( \compose*{K_R\mbox{$\mathbf{\bar{x}}$} + ae_1}{\bar{\grad}\varphi} - z\big) \geq \mathcal{E}^0_\el(v,R,z).
\qedhere
\end{equation*}
\end{proof}
\bigskip 

\subsection{Convergence of solutions; Proof of Theorems \ref{Thm:MainResultConvergenceOfSolutions} and \ref{Thm:MainResultConvergenceOfApproximateSolutions}}
\label{Sec:ConvergenceOfSolutions}

With \cref{Thm:EvolutionaryGammaConvergence} at hand, \cref{Thm:MainResultConvergenceOfSolutions,Thm:MainResultConvergenceOfApproximateSolutions} can be shown following the general strategy of \cite{MS13,Dav14_02} which themselves depend on the general theory introduced in \cite{MRS08}. We also use ideas of \cite[Thm.~2.1.6]{MR15} in the proof. To unify the statement about the convergences of solutions and time-discretized (approximate) solutions, we introduce the following notion of approximate solutions.

\begin{definition}[Approximate solutions]
Consider a sequence of ERIS given by $\mathcal{M}^h := (\mathcal{Q}, \mathcal{E}^h, \mathcal{D}^h)$, where $\mathcal{Q} = \mathcal{Y} \times \mathcal{Z}$ is a set, $\mathcal{E}^h: [0,T] \times \mathcal{Q} \to (-\infty, \infty]$ and $\mathcal{D}^h: \mathcal{Z} \times \mathcal{Z} \to [0,\infty]$. We call trajectories $q^h=(y^h,z^h): [0,T] \to \mathcal{Q}$ a sequence of \emph{approximate solutions}, if we find a map $\kappa: [0,\infty) \to [0,\infty)$ with $\lim_{h \to 0} \kappa(h) = 0$, such that
\begin{enumerate}[(a)]
	\item (Approximate stability). For all $h > 0$ and $t \in [0,T]$, $\mathcal{E}^h(t,q^h(t)) < \infty$ and we find a sequence $(t^h) \subset [0,T]$ with $t^h \to t$, such that for all $\hat{q}=(\hat{y},\hat{z}) \in \mathcal{Q}$, we have
	\begin{equation} \label{Eq:ApproxStability}
		\mathcal{E}^h(t^h,\hat{q}) + \mathcal{D}^h(z^h(t), \hat{z}) - \mathcal{E}^h(t^h,q^h(t)) \geq -\kappa(h).
	\end{equation}
	\item (Approximate energy balance). $[0,T] \ni t \mapsto \partial_t \mathcal{E}^h(t,q^h(t))$ is well-defined and integrable and for any $t \in [0,T]$,
	\begin{equation} \label{Eq:ApproxEnergyBalance}
		\mathcal{E}^h(t, q^h(t)) + \operatorname{Diss}_{\mathcal{M}^h}(z^h; [0,t]) \leq \mathcal{E}^h(0, q^h(0)) + \int_0^t \partial_s \mathcal{E}^h(s,q^h(s)) \,\dd s + \kappa(h).
	\end{equation}
\end{enumerate}
\end{definition}

\begin{remark}
Obviously any sequence of energetic solutions is a sequence of approximate solutions. Moreover, it is not hard to show that (approximate) solutions to the time-discretized problem, see \cref{Thm:MainResultConvergenceOfApproximateSolutions}, are also approximate solutions, cf.\ \cite[Thm.~3.4]{MRS08} and \cite[Thm.~6.2]{Dav14_02}, if the initial data $q^h_0 = (y^h_0, z^h_0) \in \mathcal{Q}^h_{(v_\mathrm{bc}, R_\mathrm{bc})}$ is approximately stable, i.e.\ satisfies
\begin{equation} \label{Eq:InitialData}
	\mathcal{E}^h(0, \hat{q}) + \mathcal{D}^h(z^h_0, \hat{z}) - \mathcal{E}^h(0, q^h_0) \geq -\kappa(h) \quad
	\text{for all } \hat{q} = (\hat{y}, \hat{z}) \in \mathcal{Q}^h_{(v_\mathrm{bc}, R_\mathrm{bc})}.
\end{equation}
Especially, since (approximate) solutions to the time-discretized problem always exist, provided \eqref{Eq:InitialData}, the set of sequences of approximate solutions is not empty.
\end{remark}

We unify and refine the statements of \cref{Thm:MainResultConvergenceOfSolutions,Thm:MainResultConvergenceOfApproximateSolutions} in the following proposition.

\begin{proposition} \label{Prop:ConvergenceOfSolutions}
Consider \cref{Ass:ElasticMaterialLaw,Ass:PlasticMaterialLaw}. Let $(v_\mathrm{bc}, R_\mathrm{bc}) \in \R^3 \times \SO(3)$ and $q^h_0=(y^h_0,z^h_0) \in \mathcal{Q}^h_{(v_\mathrm{bc}, R_\mathrm{bc})}$, $q^0_0=(v^0_0, R^0_0, z^0_0) \in \mathcal{Q}^0_{(v_\mathrm{bc}, R_\mathrm{bc})}$ satisfying \eqref{Eq:InitialData}, as well as
\begin{subequations}
\begin{align}
 	z^h_0 &\wto z^0_0 &&\text{in } \Leb^2(\Omega, \R^{3\times 3}), \\
 	y^h_0 &\to v_0 &&\text{in } \SobH^1(\Omega, \R^3), \\
 	\grad[h]y^h_0 &\to R^0_0 &&\text{in } \Leb^2(\Omega, \R^{3\times 3}), \\
 	\mathcal{E}^h(0, q^h_0) &\to \mathcal{E}^0(0,q^0_0) < \infty.
\end{align}
\end{subequations}
Let $q^h = (y^h, z^h)$ be a sequence of approximate solutions to the ERIS $\mathcal{M}^h = (\mathcal{Q}^h_{(v_\mathrm{bc}, R_\mathrm{bc})}, \mathcal{E}^h, \mathcal{D}^h)$ with $q^h(0) = q^h_0$. There exists an energetic solution $q=(v,R,z)$ to $\mathcal{M}^0 = (\mathcal{Q}^0_{(v_\mathrm{bc}, R_\mathrm{bc})}, \mathcal{E}^0, \mathcal{D}^0)$ with $q(0) = q^0_0$, such that up to a subsequence (independent of $t$, not relabeled), 
\begin{subequations} \label{Eq:ConvergenceOfApproxSolutions}
\begin{align}
	z^h(t) &\wto z(t) && \text{in } \Leb^2(\Omega, \R^{3\times 3}) \text{ for all } t \in [0,T], \\
	\mathcal{E}^h(t,q^h(t)) &\to \mathcal{E}^0(t,q(t)) && \text{for all } t \in [0,T], \\
	\partial_t \mathcal{E}^h(t,q^h(t)) &\to \partial_t\mathcal{E}^0(t,q(t)) && \text{for a.e.\ } t \in [0,T] \text{ and in } \Leb^1([0,T]),\\
	\operatorname{Diss}_{\mathcal{M}^h}(z^h;[0,t]) &\to \operatorname{Diss}_{\mathcal{M}^0}(z;[0,t]) && \text{for all } t \in [0,T].
\end{align}
\end{subequations}
Moreover for any $t \in [0,T]$, up to a further, $t$-dependent subsequence (not relabeled),
\begin{subequations} \label{Eq:ConvergenceOfApproxSolutionsDeformation}
\begin{align}
	y^h(t) &\to v(t) && \text{in } \Leb^2(\Omega, \R^3), \\
	\grad[h]y^h(t) &\to R(t) && \text{in } \Leb^2(\Omega, \R^{3\times 3}).
\end{align}
\end{subequations}
\end{proposition}

\begin{proof}
\DeclarePairedDelimiterX{\sprod}[2]{\langle}{\rangle}{#1\,,\,\mathopen{} #2}
\stepemph{Step 1 -- Energy bound:} 
We denote by $\sprod*{\placeholder}{\placeholder}_Q$ the euclidean scalar product in $\Leb^2(\Omega, \R^3)$. From the compactness \cref{Thm:EvolutionaryGammaConvergence} (a) and \eqref{Eq:ApproxEnergyBalance}, we obtain
\begin{align*}
	&\norm*{y^h(t)}_{\SobH^1(\Omega)}^2 + \norm*{\grad[h]y^h(t)}_{\Leb^2(\Omega)}^2 + \norm*{z^h(t)}_{\Leb^2(\Omega)}^2 + \mathcal{E}^h_{\el+\pl}(y^h(t), z^h(t)) + \operatorname{Diss}_{\mathcal{M}^h}(z^h;[0,t]) \\
	&\qquad\overset{\mathllap{\eqref{Eq:Coercivity}}}{\leq} c_1 \left( 1 + \mathcal{E}^h_{\el+\pl}(y^h(t), z^h(t)) + \operatorname{Diss}_{\mathcal{M}^h}(z^h;[0,t]) \right) \\
	&\qquad= c_1 \left( 1 + \mathcal{E}^h(t, y^h(t), z^h(t)) + \operatorname{Diss}_{\mathcal{M}^h}(z^h;[0,t]) + \sprod*{l(t)}{y^h(t)}_Q \right) \\
	&\qquad\overset{\mathllap{\eqref{Eq:ApproxEnergyBalance}}}{\leq} c_1 \left( 1 + \kappa(h) + \mathcal{E}^h(0, q^h_0) - \int_0^t \sprod*{\partial_s l(s)}{y^h(s)}_Q \dd s + \sprod*{l(t)}{y^h(t)}_Q \right) \\
	&\qquad\leq c_2 \left( 1 + \int_0^t \norm*{y^h(s)}_{\SobH^1(\Omega)} \,\dd s + \norm*{y^h(t)}_{\SobH^1(\Omega)} \right).
\end{align*}
By absorbing the latter term into the left-hand side and using Gronwall's lemma, we obtain for some constant $c_3 > 0$ independent of $h$ and $t$,
\begin{equation*}\tag{$*$}
	\norm*{y^h(t)}_{\SobH^1(\Omega)}^2 + \norm*{\grad[h]y^h(t)}_{\Leb^2(\Omega)}^2 + \norm*{z^h(t)}_{\Leb^2(\Omega)}^2 + \mathcal{E}^h_{\el+\pl}(y^h(t), z^h(t)) + \operatorname{Diss}_{\mathcal{M}^h}(z^h;[0,t]) \leq c_3.
\end{equation*}

\stepemph{Step 2 -- Compactness:}
Using $(*)$ and $l \in \SobW^{1,1}([0,T], \SobH^1(\Omega, \R^3))$, we observe that the sequence $(\partial_t\mathcal{E}^h(\placeholder, q^h(\placeholder)))_h$ is bounded in $\Leb^1([0,T])$. Thus, using $(*)$ once more to apply a generalized version of Helly's selection principle, cf.\ \cite[Thm.~A.1]{MRS08}, we obtain a subsequence (independent of $t$, not relabeled), such that for all $t \in [0,T]$
\begin{align*}
	z^h(t) &\wto z(t) \text{ in } \Leb^2(\Omega, \R^{3\times 3}), &
	\operatorname{Diss}_{\mathcal{M}^h}(z^h;[0,t]) &\to: \delta^0(t), \\
	P^h := \sprod[\big]{\partial_t l(\placeholder)}{y^h(\placeholder)}_\Omega &\wto: P^* \text{ in } \Leb^1([0,T]),
\end{align*}
and $\delta^0:[0,T] \to \R$ satisfies
\begin{equation*}
	\operatorname{Diss}_{\mathcal{M}^0}(z;[s,t]) \leq \delta^0(t) - \delta^0(s) \quad \text{for all } 0 \leq s < t \leq T.
\end{equation*}
Let us fix $t \in [0,T]$. We define $E^-(t) := \liminf_{h \to 0} \mathcal{E}^h(t, q^h(t))$, $E^+(t) := \limsup_{h \to 0} \mathcal{E}^h(t, q^h(t))$ and $P^0(t) := \liminf_{h \to 0} P^h(t)$. Note that by Fatou, we have 
\begin{equation}\label{Eq:ProofConvergenceSolution:Fatou}
	P^*(s) \geq P^0(s) \text{ for a.e.\ } s \in [0,T].
\end{equation}
We restrict to a further $t$-dependent subsequence, indicated by $h'_t$, such that $P^{h'_t}(t) \to P^0(t)$.
Let us anticipate that in Step 5 we shall show that the whole (sub-)sequence $\left(\sprod*{\partial_s l(s)}{y^h(s)}_\Omega\right)_h$ is convergent for a.e.\ $s \in [0,T]$. Thus, the restriction to this subsequence is not necessary, which is important as discussed in \cref{Rem:ConvergenceSubsequences} below. In view of the compactness \cref{Thm:EvolutionaryGammaConvergence} (a), we find $(v(t),R(t)) \in \mathcal{A}_{\rm rod}$ and a further $t$-dependent subsequence (not relabeled), such that
\begin{equation}\label{Eq:ProofConvergenceSolution:ConvDeformation}
	y^{h'_t}(t) \to v(t) \text{ in } \Leb^2(\Omega, \R^3), \qquad
	\grad[h'_t] y^{h'_t}(t) \to R(t) \text{ in } \Leb^2(\Omega, \R^{3\times 3}),
\end{equation}
and $q(t) = (v(t),R(t),z(t)) \in \mathcal{Q}^0_{(v_\mathrm{bc}, R_\mathrm{bc})}$. We observe $q(0) = q^0_0$ and $P^0(s) = \sprod[\big]{\partial_s l(s)}{v(s)}_\Omega = \sprod*{\partial_s l^{\rm eff}(s)}{v(s)}_\omega$.
\medskip

\stepemph{Step 3 -- Stability:}
We show that the limit $q = (v,R,z)$ is stable, i.e.\ satisfies $q(t) \in \mathcal{S}_{\mathcal{M}^0}(t)$ for any $t \in [0,T]$. We prove this with the help of the mutual recovery sequence \cref{Thm:EvolutionaryGammaConvergence} (c). Indeed, fix $t \in [0,T]$ and let $(\hat{v},\hat{R},\hat{z}) \in \mathcal{Q}^0_{(v_\mathrm{bc}, R_\mathrm{bc})}$. By \cref{Thm:EvolutionaryGammaConvergence} (c), there exists a subsequence $(h_j)$ of $(h'_t)$, such that \eqref{Eq:ProofConvergenceSolution:ConvDeformation} holds and for which we can construct a mutual recovery sequence $\hat{q}^{h_j} = (\hat{y}^{h_j}, \hat{z}^{h_j})$ associated to $q(t)$, $\hat{q}$ and $q^{h_j}(t)$. We denote by $(t^h) \subset [0,T]$ a sequence with $t^h \to t$ satisfying \eqref{Eq:ApproxStability}. Note that from the convergences $y^{h_j}(t) \to v$ and $\hat{y}^{h_j} \to \hat{v}$ in $\SobH^1(\Omega, \R^3)$, we may infer
\begin{equation*}
	\sprod*{l(t^{h_j})}{y^{h_j}(t)}_\Omega \to \sprod[\big]{l(t)}{v(t)}_\Omega = \sprod*{l^{\rm eff}(t)}{v(t)}_\omega, \qquad
	\sprod*{l(t^{h_j})}{\hat{y}^{h_j}}_\Omega \to \sprod*{l^{\rm eff}(t)}{\hat{v}(t)}_\omega.
\end{equation*}
Thus, we obtain
\begin{multline*}
	\mathcal{E}^0(t, \hat{q}) + \mathcal{D}^0(z(t), \hat{z}) - \mathcal{E}^0(t, q(t)) \\
	\overset{\eqref{Eq:MututalLimsupInequality}}{\geq} \limsup_{j \to 0} \left( \mathcal{E}^{h_j}(t^{h_j}, \hat{q}^{h_j}) + \mathcal{D}^{h_j}(z^{h_j}(t), \hat{z}^{h_j}) - \mathcal{E}^{h_j}(t^{h_j}, q^{h_j}(t)) \right)
	\overset{\eqref{Eq:ApproxStability}}{\geq} \limsup_{j \to 0} (-\kappa(h_j)) = 0.
\end{multline*}
We claim that the stability implies $\mathcal{E}^0(t,q(t)) \leq E^-(t)$. Indeed, by appealing to a different subsequence of $(h)$, indicated by $h^\#_t$, we find $(v^*,R^*)$ with $(v^*,R^*, z(t)) \in \mathcal{Q}^0_{(v_\mathrm{bc}, R_\mathrm{bc})}$ satisfying
\begin{align*}
	y^{h^\#_t}(t) &\to v^* \text{ in } \Leb^2(\Omega, \R^3), &
	&\grad[h^\#_t] y^{h^\#_t}(t) \to R^* \text{ in } \Leb^2(\Omega, \R^{3\times 3}), \text{ and} &
	\mathcal{E}^{h^\#_t}(t, q^{h^\#_t}(t)) &\to E^-(t).
\end{align*}
The stability of $q(t)$ especially implies that $(v(t), R(t))$ minimizes $(v^*,R^*) \mapsto \mathcal{E}^0(t,v^*,R^*,z(t))$. Thus, the lower bound statement \eqref{Eq:LiminfInequalities} shows
\begin{equation} \label{Eq:ProofConvergenceSolution:LiminfEnergy}
	\mathcal{E}^0(t,q(t)) \leq \mathcal{E}^0(t,v^*,R^*, z(t)) \leq \liminf_{h^\#_t \to 0} \mathcal{E}^h(t, q^{h^\#_t}(t)) = E^-(t).
\end{equation}

\stepemph{Step 4 -- Energy balance:}
From $P^0(t) = \liminf_{h \to 0} P^h(t)$ we infer that $P^0$ is measurable and combined with $\abs*{P^h(t)} \leq c_4\norm*{\partial_1 l(t)}_{\Leb^2(\Omega)}$, we deduce $P^0 \in \Leb^1([0,T])$. It remains to show that $q$ satisfies the energy balance \eqref{Eq:GlobalEnergyBalance}. We split this statement by showing the lower and upper bound separately, starting with the lower bound. Indeed,
\begin{align*}
	&\mathcal{E}^0(t, q(t)) + \operatorname{Diss}_{\mathcal{M}^0}(z;[0,t]) 
	\leq E^-(t) + \delta^0(t) \leq E^+(t) + \delta^0(t) \\
	&\qquad= \limsup_{h \to 0} \left(\mathcal{E}^h(t, q^h(t)) + \operatorname{Diss}_{\mathcal{M}^h}(z^h;[0,t])\right) \\
	&\qquad\overset{\mathllap{\eqref{Eq:ApproxEnergyBalance}}}{\leq} \lim_{h \to 0} \left( \mathcal{E}^h(0, q^h_0) - \int_0^t \sprod*{\partial_s l(s)}{y^h(s)}\,\dd s + \kappa(h)\right) \\
	&\qquad\overset{\mathllap{\text{Step 2}}}{=} \mathcal{E}^0(0,q^0_0) - \int_0^t P^*(s) \,\dd s
	\leq \mathcal{E}^0(0,q^0_0) - \int_0^t \sprod*{\partial_sl^\mathrm{eff}(s)}{v(s)}_\omega \,\dd s.
\end{align*}
The upper bound is a direct consequence of the stability established in Step 3. We refer to \cite[Prop.~2.4]{MRS08} and \cite[Prop.~2.1.23]{MR15} (see also the comments in the proof of \cite[Thm.~2.4.10]{MR15}) for a detailed proof.

\stepemph{Step 5 -- Improved convergence:}
We show the remaining convergences in \eqref{Eq:ConvergenceOfApproxSolutions}. From Step 4, we infer
\begin{multline*}
	\mathcal{E}^0(0,q^0_0) - \int_0^t \sprod*{\partial_sl^\mathrm{eff}(s)}{v(s)}_\omega \,\dd s
	= \mathcal{E}^0(t, q(t)) + \operatorname{Diss}_{\mathcal{M}^0}(z;[0,t])
	\leq E^-(t) + \delta^0(t) \\
	\qquad \leq E^+(t) + \delta^0(t) 
	\leq \mathcal{E}^0(0,q^0_0) - \int_0^t P^*(s) \,\dd s
	\leq \mathcal{E}^0(0,q^0_0) - \int_0^t \sprod*{\partial_sl^\mathrm{eff}(s)}{v(s)}_\omega \,\dd s.
\end{multline*}
Thus, the inequalities are in fact equalities. Recall that for all $t \in [0,T]$ and a.e.\ $s \in [0,T]$, $\mathcal{E}^0(t,q(t)) \leq E^-(t) \leq E^+(t)$, $\operatorname{Diss}_{\mathcal{M}^0}(z;[0,t]) \leq \delta^0(t)$ and $P^*(s) \geq \sprod*{\partial_sl^\mathrm{eff}(s)}{v(s)}_\omega$. Thus, we obtain,
\begin{equation*}
	\mathcal{E}^0(t, q(t)) = E^-(t) = E^+(t), \qquad
	\operatorname{Diss}_{\mathcal{M}^0}(z;[0,t]) = \delta^0(t) \quad\text{and}\quad
	\sprod*{\partial_sl^\mathrm{eff}(s)}{v(s)}_\omega = P^*(s).
\end{equation*}
The last equation shows that the weak limit of $\left(\sprod*{\partial_t l(\placeholder)}{y^h(\placeholder)}_\Omega\right)_h$ coincides with the point-wise $\liminf$. By \cref{Lem:PointwiseConvergenceFromWeakConvergence} we conclude a.e.\ convergence of $\partial_t \mathcal{E}^h(\placeholder, q^h(\placeholder)) = -\sprod*{\partial_t l(\placeholder)}{y^h(\placeholder)}_\Omega$. Moreover, since $\partial_t\mathcal{E}^h(\placeholder, q^h(\placeholder))$ admits a dominating map in $\Leb^1([0,T])$, we also obtain strong convergence in $\Leb^1([0,T])$ and conclude \eqref{Eq:ConvergenceOfApproxSolutions}.
\end{proof}

\begin{remark}~ \label{Rem:ConvergenceSubsequences}
\begin{enumerate}[(i)]
	\item Since, provided \eqref{Eq:InitialData}, approximate solutions always exist, one consequence of \cref{Prop:ConvergenceOfSolutions} is the existence of solutions to the limiting ERIS.
	
	\item If there exists a unique solution $(v,R,z)$ to the limiting ERIS, then the convergence statements \eqref{Eq:ConvergenceOfApproxSolutions} and \eqref{Eq:ConvergenceOfApproxSolutionsDeformation} hold for the entire sequence, since for any subsequence there exists a further subsequence such that \eqref{Eq:ConvergenceOfApproxSolutions} and \eqref{Eq:ConvergenceOfApproxSolutionsDeformation} hold and the limits are independent of the chosen subsequence.
	
	\item That \eqref{Eq:ConvergenceOfApproxSolutionsDeformation} holds up to subsequence is especially due to the fact that a posteriori it is not necessary to restrict to a subsequence, where $\liminf_{h \to 0} P^h(t)$ is attained as a limit as discussed in Step 2 of the proof.
	
	\item If solutions to the limiting ERIS are not necessarily unique but at least for a given $z \in \Leb^2(\Omega, \R^{3\times 3})$, there exists at most one stable state $(v,R,z)$ for the limiting model, then \eqref{Eq:ConvergenceOfApproxSolutionsDeformation} holds without extracting a further $t$-dependent subsequence, since the limit $(v(t), R(t))$ is uniquely defined from $z(t)$. A sufficient condition for uniqueness of stable states is that $(v,R) \mapsto \mathcal{E}^0(t,v,R,z)$ admits a unique minimizer for any $t \in [0,T]$ and $z \in \Leb^2(\Omega, \R^{3\times 3})$, cf.\ \cite[Thm.~2.4.13]{MR15}.	
\end{enumerate}
\end{remark}

\section{Discussion of the rod model}
\label{Sec:Example}

The limiting ERIS $(\mathcal{Q}^0_{(v_\mathrm{bc}, R_\mathrm{bc})}, \mathcal{E}^0, \mathcal{D}^0)$ is non-convex. 
In particular, the energy functional $\mathcal{E}^0$ and the state space $\mathcal{Q}^0_{(v_\mathrm{bc}, R_\mathrm{bc})}$ are non-convex. 
The existence of an energetic solutions follows from our convergence result Theorem~\ref{Thm:MainResultConvergenceOfApproximateSolutions} but it can also be deduced directly from the standard construction via time-incremental problems.
In view of the non-convexity, we cannot expect uniqueness and a priori temporal regularity of the non-dissipative component $(v,R)$ fails, as solutions may jump arbitrarily between different stable configurations and may thus even not be measurable. 
As a consequence the passage to a $t$-dependent subsequence for convergence of the non-dissipative component $y^h$ cannot be avoided in general. 
\medskip

In this section we discuss an example where a unique solution to the limiting problem is expected. 
As discussed in \cref{Rem:ConvergenceSubsequences}, in this case we observe convergence of the whole sequence of (approximate) solutions to the 3D problem. 
Furthermore, in this example we obtain temporal regularity of the solution.
For the example we consider a highly symmetric situation, where neither the rod's cross-section nor the elastic and plastic material law prefers a bending direction.
In particular, we assume that the rod's cross-section is circular,
\begin{equation}
	\Omega = \omega \times S, \qquad
	\omega = (0,1), \qquad
	S = B_{\pi^{-1/2}}(0),
\end{equation}
and that the elastic material law is isotropic with vanishing Poisson ratio,
\begin{equation}
	Q_\el(G) := \mu \abs*{\sym G}^2,
\end{equation}
for some $\mu > 0$. Moreover, for some $\rho, \delta > 0$, we assume that the linearized plastic material law satisfies
\begin{equation}
	Q_\pl(G) := \rho \abs*{G}^2, \qquad
	\widetilde{\mathscr{R}}_\pl(\dot{F}) := \delta \abs*{\dot{F}}.
\end{equation}
We consider the situation of a one-sided horizontally clamped rod, i.e., boundary and initial conditions 
\begin{alignat}{3}\label{example:ass:bc}
	v(t,0) &= v_{\rm bc}=0,	&\quad R(t,0) 		&= R_{\rm bc}=I &&\quad\text{ for all }t\in[0,T],\\
	v(0,x_1) &= x_1, 				&\quad R(0,x_1) 	&= I 						&&\quad\text{ for all }x_1\in\omega,
\end{alignat}
and assume that the rod is initially not plastically strained, i.e.,
\begin{equation*}
	 z(0,x) = 0\text{ for a.e. }x \in \Omega.
\end{equation*}
The loading term is given by a force in direction $-e_2$ as
\begin{equation}
	l^{\rm eff}(t, x_1) := -\beta(t) f(x_1) e_2,
\end{equation}
where $\beta \in \Cont^\infty(\R_+)$ is increasing with $\beta(0) = 0$ and $f \in \Cont^\infty(\overline{\omega})$ with $f > 0$ in $\omega$.
\medskip

Because of the symmetry of the situation, it is natural to expect that the rod remains in the $(x_1,x_2$)-plane, i.e.,
\begin{equation}\label{example:ass:planar}
	v\cdot e_3=0\qquad\text{a.e. in }[0,T]\times\omega.
\end{equation}
In the following we shall assume \eqref{example:ass:planar} to simplify the upcoming discussion.
We note that a rod configuration $(v,R) \in \mathcal{A}_{\mathrm{rod}}$ satisfying \eqref{example:ass:planar} and the boundary condition \eqref{example:ass:bc} admits a unique representation 
\begin{equation} \label{Eq:Planar_rod}
	v(x_1) = \begin{pmatrix} \int_0^{x_1} \cos\alpha(s) \,\dd s \\ \int_0^{x_1} \sin\alpha(s) \,\dd s \\ 0\end{pmatrix}, \qquad
	R(x_1) = \begin{pmatrix} 
		\cos\alpha(x_1) & -\sin\alpha(x_1) & 0 \\
		\sin\alpha(x_1) & \cos\alpha(x_1)  & 0 \\
		0					 				& 0					 				 & 1
	\end{pmatrix},
\end{equation}
for a unique angle $\alpha \in \SobH^1(\omega)$ with $\alpha(0)=0$. As we shall explain below, this allows us to equivalently express the ERIS in terms of a state variable $q$ consisting of the angle $\alpha$ and the plastic strain $z$.

\medskip

We shall prove the following observations:
\begin{enumerate}[(a)]
	\item (Existence of an elastic regime). Consider an energetic solution $q$ with plastic component $z$. Denote by $t^*\geq 0$ the first point in time when plastic deformation occurs, that is,
	\begin{equation} \label{Eq:ExampleTstar}
		t^*:=\sup\set*{\tau\geq 0 \given z=0\text{ a.e.\ in }[0,\tau]\times\Omega}.
	\end{equation}
	We show that $t^*$ only depends on $(\mathcal{Q}^0_{(v_\mathrm{bc}, R_\mathrm{bc})}, \mathcal{E}^0, \mathcal{D}^0)$ and satisfies $0<t^*<\infty$, cf.\ \cref{Lem:Example:ElasticRegime}. In particular, $t^*$ is independent of the specific energetic solution in the definition of $t^*$.
	
	\item (Uniqueness in the elastic regime). In the elastic regime $[0,t^*]$ the energetic solutions are unique and the angle satisfies 
	\begin{equation*}
		\alpha(t)\in[-\tfrac\pi2,0]\qquad\text{for all }t\in[0,t_*].
	\end{equation*}
	In fact, we show that $\alpha$ solves an ODE, see \eqref{Eq:ExampleELG} below. See \cref{Lem:Example:UniqueSolutionElastic} for details.
	
	\item (Conditional uniqueness). Suppose $q$ is a solution to the $(\mathcal{Q}^0_{(v_\mathrm{bc}, R_\mathrm{bc})}, \mathcal{E}^0, \mathcal{D}^0)$ with an angle satisfying
	\begin{equation}\label{example:anglecondition}
		\alpha(t) \in [-\tfrac\pi2,0] \qquad\text{for a.e.\ } t \in [0,T].
	\end{equation}
	Then, $q \in W^{1,\infty}([0,T], \mathcal{Q}^0)$ and any other energetic solution $\tilde q$ whose angle satisfies \eqref{example:anglecondition} is equal to $q$.
\end{enumerate}
As we shall argue below, it is natural to expect that energetic solutions for the specific setting under consideration satisfy \eqref{example:anglecondition}. However, we do not prove this property (except in the elastic regime).

\paragraph{Equations for the angle.}
The argument for these claims is split into several lemmas.  First, we rewrite the limiting ERIS $(\mathcal{Q}^0_{(v_\mathrm{bc}, R_\mathrm{bc})}, \mathcal{E}^0, \mathcal{D}^0)$ as an ERIS ($\mathcal{Q}, \mathcal{E}, \mathcal{D}$) with state variable $q=(\alpha, z)$. To this end we use \eqref{example:ass:planar} and various formulas for $Q^{\rm eff}_{\rm el}$, $K^{\rm eff}$ and $z^{\rm res}$ that hold in the specific setting under consideration and that have been obtained in \cite{BGKN22-PP}. We note that for rod configurations $(v,R)$ satisfying \eqref{example:ass:planar} the associated bending-torsion strain $K_R = R^T \partial_1 R$ takes the form
\begin{equation}
	K_R = -\alpha' \sqrt{2}K_2,
\end{equation}
where $\alpha'$ denotes the (weak) derivative of $\alpha$ with respect to $x_1$. 
Based on this, we can write the energy $\mathcal{E}^0$ in terms of $\alpha$. The boundary conditions reduce to $\alpha(t,0) = 0$ for all $t \in [0,T]$ and the initial conditions to $\alpha(0,x_1) = 0$, $z(0,x) = 0$ for a.e.\ $x \in \Omega$. Let us first calculate the effective quantities using \cref{Prop:EffectiveQuantities} and \cite[Lem.~2.11]{BGKN22-PP}. Since $\mathcal{D}^0(z,\hat{z}) = \infty$ if $\hat{z} - z \nin \Leb^2(\Omega,\R^{3\times 3}_{\rm dev})$, it suffices to consider $z \in \Leb^2(\Omega,\R^{3\times 3}_{\rm dev})$. We obtain 
\begin{align}
	Q_\el^{\rm eff}(K) &= \frac{\mu}{8\pi} \sum_{i=1}^{3} \abs*{K \cdot K_i}^2, \\
	K^{\rm eff}(z) &= \begin{aligned}[t]
		\left[ 2\pi {\textstyle\fint_S} x_3 z_{12} + x_2z_{13} \,\dd \bar{x}\right] \sqrt{2}K_1 +
		\left[ 4\pi {\textstyle\fint_S} x_2 z_{11} \,\dd \bar{x}\right] &\sqrt{2}K_2 \\
		+ \left[ 4\pi {\textstyle\fint_S} x_3 z_{11} \,\dd \bar{x}\right] &\sqrt{2}K_3,
	\end{aligned} \\
	z^{\rm res} &= \begin{aligned}[t]
		z - \sym\compose*{-{\textstyle\fint_S} z_{11} e_1}{\bar{\grad}\varphi_z} +
		\left[ 2\pi {\textstyle\fint_S} x_3 z_{12} + x_2z_{13} \,\dd \bar{x}\right] &\sqrt{2}\Psi_1 \\
		+ \left[ 4\pi {\textstyle\fint_S} x_2 z_{11} \,\dd \bar{x}\right] \sqrt{2}\Psi_2 +
		\left[ 4\pi {\textstyle\fint_S} x_3 z_{11} \,\dd \bar{x}\right] &\sqrt{2}\Psi_3,
	\end{aligned}
\end{align}
where $K_1$, $K_2$ and $K_3$ are defined in \eqref{Eq:BasisSkewMatrices},
\begin{equation}
	\Psi_1(\bar{x}) := \frac{1}{\sqrt{2}}  \sym
		\left(\begin{smallmatrix}
			0 & 0 & 0 \\
			x_3 & 0 & 0 \\
			-x_2 & 0 & 0
		\end{smallmatrix}\right), \qquad
	\Psi_2(\bar{x}) := \frac{x_2}{\sqrt{2}}
		\left(\begin{smallmatrix}
			1 & 0 & 0 \\
			0 & 0 & 0 \\
			0 & 0 & 0
		\end{smallmatrix}\right), \qquad
	\Psi_3(\bar{x}) := \frac{x_3}{\sqrt{2}}
		\left(\begin{smallmatrix}
			1 & 0 & 0 \\
			0 & 0 & 0 \\
			0 & 0 & 0
		\end{smallmatrix}\right),
\end{equation}
and defining $\tilde{z} \in \Leb^2(\Omega, \R^{3\times 3}_{\rm sym})$ by $\tilde{z}_{11} := 0$ and $\tilde{z}_{ij} := z_{ij}$ else, $\varphi_z(x_1,\placeholder)$ minimizes
\begin{equation}
	\SobH^1(S, \R^3) \ni \varphi \mapsto \int_S \abs*{\sym\compose*{0}{\bar{\grad}\varphi} + \tilde{z}(x_1,\placeholder)}^2.
\end{equation}
Furthermore, we calculate the loading term using integrating by parts to get
\begin{equation}
	\int_0^1 l^{\rm eff}(t) \cdot v = -\beta(t) \int_0^1 F(x_1) \sin\alpha(x_1)\,\dd x_1,
\end{equation}
where $F(x_1) := \int_{x_1}^1 f$. Note that $F$ satisfies $F > 0$ on $[0,1)$ and $F(1) = 0$. Using this, it is not hard to see that any solution satisfies $z_{12} = z_{13} = 0$ and the ERIS reduces to
\begin{align}
	\mathcal{E}(t,\alpha,z)
	&:= \begin{aligned}[t]
		&\frac{\mu}{4\pi} \int_0^1 \abs*{\alpha' - \left[4\pi {\textstyle\fint_S} x_2 z_{11} \,\dd \bar{x}\right]}^2 \,\dd x_1 + \frac{\mu}{4\pi} \int_0^1 \abs*{{\textstyle\fint_S} x_3 z_{11} \,\dd \bar{x}}^2 \,\dd x_1\\
		&\qquad + \mu \int_\Omega \abs*{z_{11} + {\textstyle\fint_S} z_{11} + \left[4\pi {\textstyle\fint_S} \tilde{x}_2 z_{11} \,\dd \tilde{x}\right]x_2 + \left[4\pi {\textstyle\fint_S} \tilde{x}_3 z_{11} \,\dd \tilde{x}\right]x_3}^2 \,\dd x \\
		&\qquad + \mu \int_\Omega \sum_{i=2}^3 \abs*{z_{ii} - \partial_i(\varphi_z)_i}^2 + 2\abs*{z_{23} - \tfrac{1}{2}(\partial_3(\varphi_z)_2 + \partial_2 (\varphi_z)_3)}^2 \,\dd x  \\
		&\qquad + \rho \int_\Omega \sum_{i=1}^3 \abs*{z_{ii}}^2 + 2\abs*{z_{23}}^2\\
		&\qquad + \beta(t) \int_0^1 F\sin \alpha,
	\end{aligned} \\
	\mathcal{D}(z, \hat{z}) &:= \delta \int_\Omega \sqrt{\sum_{i=1}^3 \abs*{\hat{z}_{ii} - z_{ii}}^2 + 2\abs*{\hat{z}_{23} - z_{23}}^2}, \\
	\mathcal{Q} &:= \set*{\alpha \in \SobH^1(\omega) \given \alpha(0) = 0} \times \set*{z \in \Leb^2(\Omega, \R^{3\times 3}_{\rm dev}) \given z_{12} = z_{13} = 0} =: \mathcal{Q}_\alpha \times \mathcal{Q}_z.
\end{align}

\paragraph{Uniqueness in the elastic regime.}
The ERIS ($\mathcal{Q}, \mathcal{E}, \mathcal{D}$) describes the evolution of a horizontally clamped planar rod under an increasing downwards directed force (e.g.\ gravity). For such a problem it is reasonable to assume that a solution $q^*=(\alpha^*, z^*)$ satisfies $\alpha^*(t) \in [-\frac{\pi}{2}, 0]$ in $\omega$ as \cref{Fig:Elastic_rod} suggests. Indeed, we show that this can be rigorously justified for the associated elastic problem of minimizing $\mathcal{Q}_\alpha \ni \alpha \mapsto \mathcal{E}(t,\alpha,0)$.

\begin{figure}[htp]
\centering
\includegraphics[width=0.72\linewidth]{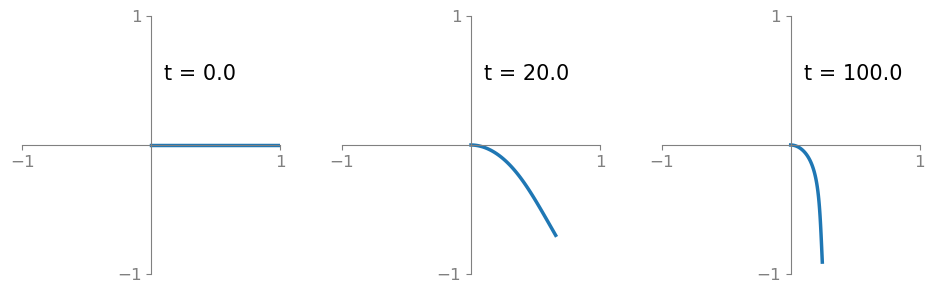}
\caption{The graphs depict the deformation $v$ associated via \eqref{Eq:Planar_rod} to the solution to \eqref{Eq:ExampleELG} subject to $\mu = 4\pi$, $\beta(t) = t$ and $F(x_1) = 1 - x_1$ (that is $f(x_1) \equiv 1$) for different points in time. The solution is calculated via a straight-forward finite difference scheme using the linearization $\cos \alpha \approx \cos \alpha_{\rm ap} - (\alpha - \alpha_{\rm ap})\sin\alpha_{\rm ap}$, where we choose the approximation $\alpha_{\rm ap}$ as the solution from the previous time step.}
\label{Fig:Elastic_rod}
\end{figure}

\begin{lemma} \label{Lem:Example:UniqueSolutionElastic}
Consider the energy functional
\begin{equation}
	\mathcal{E}(t,\alpha) := \frac{\mu}{4\pi}\int_0^1 \abs*{\alpha'}^2 + \beta(t)\int_0^1 F \sin \alpha,
\end{equation}
for $\alpha \in \SobH^1(\omega)$ subject to $\alpha(0) = 0$. There exists a unique minimizer $\alpha_t$ of $\mathcal{E}(t,\placeholder)$. It satisfies $\alpha_t \in \Cont^2(\overline{\omega})$, $\alpha_t \in [-\frac{\pi}{2},0]$ in $\omega$ and is a classical solution to the ordinary differential equation
\begin{equation} \label{Eq:ExampleELG}
		-\tfrac{\mu}{2\pi} \alpha'' + \beta(t)F\cos\alpha = 0, \qquad
		\text{subject to } \alpha(0) = 0, \alpha'(1) = 0.
\end{equation}
\end{lemma}

\begin{proof}
As long as $\beta(t) = 0$, we have the unique solution $\alpha_t \equiv 0$. Thus, we may assume that $\beta(t) > 0$. Consider the map $g:\R\to [-\frac{\pi}{2}, 0]$ given by
\begin{equation*}
	g(\alpha) := \begin{cases}
		-\tfrac{\pi}{2} & \text{if }\abs*{\alpha} \geq \tfrac{\pi}{2}, \\
		-\abs*{\alpha} & \text{else}.
	\end{cases}
\end{equation*}
Using that $F > 0$ on $[0,1)$, it is not hard to show that for any $\alpha \in \SobH^1(\omega)$ with $\alpha(0) = 0$, we have $\mathcal{E}(t,g \circ \alpha) \leq \mathcal{E}(t,\alpha)$ and $\mathcal{E}(t,g \circ \alpha) < \mathcal{E}(t,\alpha)$ if $\alpha(x_1) \nin [-\frac{\pi}{2},0]$ for some $x_1 \in \omega$. Moreover, define $\widetilde{\sin}: \R \to \R$ by
\begin{equation} \label{Eq:ExampleSintilde}
	\widetilde{\sin}(\alpha) := \begin{cases}
		-1 & \text{if } \alpha \leq -\tfrac{\pi}{2}, \\
		\sin(\alpha) & \text{if } \alpha \in [-\tfrac{\pi}{2},0], \\
		\alpha & \text{else},
	\end{cases}
\end{equation}
and consider the energy functional $\widetilde{\mathcal{E}}$ defined as $\mathcal{E}$ with $\sin$ replaced by $\widetilde{\sin}$. We observe $\widetilde{\mathcal{E}}(t,\alpha) = \mathcal{E}(t,\alpha)$ if $\alpha \in \SobH^1(\omega)$ satisfies $\alpha \in [-\frac{\pi}{2},0]$ in $\omega$ and $\widetilde{\mathcal{E}}(t,g \circ\alpha) < \widetilde{\mathcal{E}}(t,\alpha)$ else. 
Since $\widetilde{\mathcal{E}}(t,\placeholder)$ is strictly convex, it admits a unique minimizer $\alpha_t$. We conclude that $\alpha_t$ is also the unique minimizer of $\mathcal{E}(t,\placeholder)$ and satisfies $\alpha_t \in [-\frac{\pi}{2},0]$ in $\omega$.
Moreover, we can characterize $\alpha_t$ as the unique solution to the Euler-Lagrange equation, which in its weak form reads
\begin{equation}
	\frac{\mu}{2\pi}\int_0^1 \alpha'\varphi' + \beta(t) \int_0^1 F \varphi\, \widetilde{\sin}' \alpha = 0, \qquad
	\text{for all } \varphi \in \SobH^1(\omega) \text{ with } \varphi(0) = 0.
\end{equation}
Testing the equation for fixed $x_1 \in (0,1]$ with $\varphi(s) := \min\set*{s,x_1}$ and using integration by parts, we obtain the integral equation
\begin{equation} \label{Eq:Example:IntegralEquation}
	\alpha(x_1) = -\frac{2\pi\beta(t)}{\mu} \int_0^{x_1} \int_s^1 F(\sigma) \widetilde{\sin}'(\alpha(\sigma)) \,\dd\sigma \dd s.
\end{equation}
But by the fundamental theorem of calculus this integral equation is equivalent to the strong form of the Euler-Lagrange equation which is given by \eqref{Eq:ExampleELG} with $\cos$ replaced by $\widetilde{\sin}'$. But this replacement of $\cos$ by $\widetilde{\sin}'$ can be reversed in view of $\alpha_t \in [-\frac{\pi}{2}, 0]$ in $\omega$.
\end{proof}

\paragraph{Conditional uniqueness.}
The ERIS ($\mathcal{Q}, \mathcal{E}, \mathcal{D}$) has a unique solution $q^* = (\alpha^*, z^*)$ under the additional assumption that it satisfies $\alpha^*(t) \in [-\frac{\pi}{2}, 0]$ in $\omega$. Indeed, this is due to the fact that $\mathcal{D}$ is convex and in this case $\mathcal{E}$ is smooth and $\lambda$-convex for some $\lambda > 0$ in the convex domain
\begin{equation}
	\widetilde{Q} := \set*{\alpha \in \mathcal{Q}_\alpha \given \alpha \in [-\tfrac{\pi}{2},0] \text{ in } \omega} \times \mathcal{Q}_z,
\end{equation}
in the sense that for all $(\alpha,z), (\hat{\alpha}, \hat{z}) \in \widetilde{Q}$ and $\theta \in [0,1]$,
\begin{equation}
\begin{aligned}
	&\mathcal{E}\of*{t, \theta\alpha + (1-\theta)\hat{\alpha}, \theta z + (1-\theta)\hat{z}} \\
	&\qquad \leq \theta \mathcal{E}(t,\alpha, z) + (1- \theta)\mathcal{E}(t, \hat{\alpha}, \hat{z})
	- \lambda\theta(1-\theta)\left(\norm*{\alpha - \hat{\alpha}}_{\SobH^1(\omega)}^2 + \norm*{z - \hat{z}}_{\Leb^2(\Omega)}^2\right).
\end{aligned}
\end{equation}
The uniqueness is shown in \cite[Sect.~3.4]{MR15} (especially Corollary~3.4.6 and Section~3.4.4). Moreover, the reference establishes temporal regularity in the sense that $q^* \in \SobW^{1,\infty}([0,T], \mathcal{Q})$ and satisfies for a.e.\ $t \in [0,T]$ the differential inclusion
\begin{equation}\label{Eq:Example:DifferentialInclusion}
	-D_z \mathcal{E}(t,q^*(t)) \in \partial\mathcal{R}(\dot{z}^*(t)), \qquad -\operatorname{D}_\alpha\mathcal{E}(t,q^*(t)) = 0,
\end{equation}
where $\mathcal{R}(\dot{z}) := \delta \int_\Omega \abs*{\dot{z}}$ is the dissipation potential and $\partial\mathcal{R}$ denotes its convex subdifferential. The same consequences hold for the convexified energy $\tilde{\mathcal{E}}$ where $\sin\alpha$ is replaced by \eqref{Eq:ExampleSintilde}. For the energetic solution $\tilde{q}^*$ to the convexified ERIS, we obtain through the temporal regularity, that the angle stays in $[-\frac\pi2, 0]$ for a small time beyond the elastic regime, i.e.\ $\tilde{\alpha}^*(t) \in [-\frac\pi2, 0]$ for a.e.\ $t \in [0,T]$ for some $T > t^*$, provided the minimizer in the elastic regime satisfies $\alpha_{t^*} > -\frac\pi2$. It is clear that in $[0,T]$ the associated differential inclusion for the convexified problem coincides with \eqref{Eq:Example:DifferentialInclusion}. However, it is not clear whether $\tilde{q}^*$ is also an energetic solution to the original problem.

\paragraph{Existence of an elastic regime.}
We conclude this section by showing that there is a period of time where the solution stays elastic but it is not elastic for all points in time (provided $\beta(t) \to \infty$ as $t \to \infty$). For this we estimate the point in time $t^* \geq 0$ given by \eqref{Eq:ExampleTstar}. In view of \eqref{Eq:Example:DifferentialInclusion}, we have
\begin{equation*}
	t^* = \sup \set{ \tau \geq 0 \given s \in \mathcal{T} \text{ for a.e.\ } s \in [0,\tau]},
\end{equation*}
with
\begin{equation*}
	\mathcal{T} := \set*{t \geq 0 \given -\operatorname{D}_z\mathcal{E}(t,\alpha_t,0) \in \partial\mathcal{R}(0), \text{ where } \alpha_t \text{ is the unique minimizer of } \mathcal{E}(t,\placeholder,0)}.
\end{equation*}

\begin{lemma} \label{Lem:Example:ElasticRegime}
The point in time $t^*$ given by \eqref{Eq:ExampleTstar} satisfies
\begin{equation}
	\sup\set*{t \geq 0 \given \beta(t) \leq \tfrac{\delta \sqrt{3}}{\norm*{F}_\infty \sqrt{32\pi}}} \leq t^* \leq \inf\set*{t \geq 0 \given \beta(t) \geq \tfrac{9\mu}{8(\pi - 3) F(\frac{1}{2})}\max\set*{\pi^2, \tfrac{3\pi\delta^2}{16\mu^2}}}.
\end{equation}
\end{lemma}

\begin{proof}
Using the definition $\partial\mathcal{R}(0) = \set*{\xi \in \mathcal{Q}_z' \given \mathcal{R}(z) \geq \mathcal{R}(0) + \langle \xi, z\rangle \text{ for all }  z \in \mathcal{Q}_z}$, we calculate
\begin{align}
	-\operatorname{D}_z\mathcal{E}(t,\alpha_t,0) \in \partial\mathcal{R}(0)
	&\Leftrightarrow \frac{\mu}{2\pi}\int_0^1 \alpha_t' \left[4\pi{\textstyle\fint_S x_2 z_{11}}\,\dd\bar{x}\right] \,\dd x_1 \geq -\delta \int_\Omega \abs*{z} \text{ for all } z \in \mathcal{Q}_z \notag\\
	&\Leftrightarrow \abs*{\int_0^1 \alpha_t' \left[{\textstyle\fint_S x_2 z_{11}}\,\dd\bar{x}\right] \,\dd x_1} \leq \frac{\delta}{2\mu} \int_\Omega \abs*{z} \text{ for all } z \in \mathcal{Q}_z \notag\\
	&\Leftrightarrow \abs*{\int_0^1 \alpha_t' \left[{\textstyle\fint_S x_2 z_{11}}\,\dd\bar{x}\right] \,\dd x_1} \leq \frac{\delta\sqrt{3}}{\mu\sqrt{8}} \int_\Omega \abs*{z_{11}} \text{ for all } z_{11} \in \Leb^2(\Omega). \label{Eq:Example:CharacterizationElasticRegime}
\end{align}

\stepemph{Step 1 -- Lower bound:} In view of the definition of $S$, a sufficient condition for \eqref{Eq:Example:CharacterizationElasticRegime} is $\norm*{\alpha_t'}_\infty \leq \frac{\delta\sqrt{3\pi}}{\mu\sqrt{8}}$. Using \eqref{Eq:Example:IntegralEquation} we obtain
\begin{equation*}
	\abs{\alpha_t'(t, x_1)} = \frac{2\pi\beta(t)}{\mu} \int_{x_1}^1 F(s) \cos \alpha_t(s) \,\dd s \leq \frac{2\pi\beta(t) \norm*{F}_\infty}{\mu}.
\end{equation*}
Hence, as long as $\frac{2\pi\beta(t) \norm*{F}_\infty}{\mu} \leq \frac{\delta\sqrt{3\pi}}{\mu\sqrt{8}}$ we find $s \in \mathcal{T}$ for a.e.\ $s \in [0,t]$,  which establishes the lower bound for $t^*$.
\medskip

\stepemph{Step 2 -- Upper bound:} Testing \eqref{Eq:Example:CharacterizationElasticRegime} with $z_{11}(x) := \mathds{1}_{[0,\eps]}(x_1)\mathds{1}_{[0,(2\pi)^{-1/2}]^2}(\bar{x})$ for all $\eps > 0$, we get that necessarily
\begin{equation} \label{Eq:Example:UpperBoundTstar}
	\abs*{\alpha_t'(0)} = \lim_{\eps \to 0} \fint_0^\eps\abs*{\alpha_t'} \leq \tfrac{\delta\sqrt{3\pi}}{\mu\sqrt{16}}.
\end{equation}
We provide a rough lower bound for $\abs*{\alpha_t'(0)}$. Suppose $\abs*{\alpha_t'(0)} \leq c$ for some $c \geq \pi$. \cref{Eq:ExampleELG} shows that $\alpha_t'$ is non-positive and increasing. Thus, we obtain
\begin{equation*}
	\alpha_t(x_1) \geq \max\set*{-c x_1, -\tfrac{\pi}{2}} = \begin{cases} -c x_1 & \text{if } x_1 \leq \tfrac{\pi}{2c}, \\ -\tfrac{\pi}{2} & \text{else}	\end{cases} =: \bar{\alpha}(x_1).
\end{equation*}
Consider $\widetilde{\alpha}(x_1) := \max\set*{-\tfrac{3c}{2}x_1, -\tfrac{\pi}{2}}$. Note that $\widetilde{\alpha} \leq \bar{\alpha} \leq \alpha_t$ and $\widetilde{\alpha} \neq \alpha_t$. We estimate the energy $\mathcal{E}(t,\alpha_t,0)$ against $\mathcal{E}(t,\widetilde{\alpha},0)$. Using the monotonicity of $\sin$ on $[-\frac{\pi}{2},0]$, we obtain
\begin{align*}
	\mathcal{E}(t,\widetilde{\alpha}, 0) &= \frac{\mu}{4\pi} \int_0^1 \abs*{\widetilde{\alpha}'}^2 + \beta(t) \int_0^1 F \sin\widetilde{\alpha} \\
	&= \frac{3c \mu}{16} + \beta(t) \left[ \int_0^{\frac{\pi}{3c}} F(x_1)\sin(-\tfrac{3c}{2} x_1) \,\dd x_1 - \int_{\frac{\pi}{3c}}^{\frac{\pi}{2c}} F - \int_{\frac{\pi}{2c}}^1 F\right] \\
	&\leq \frac{3c \mu}{16} - \beta(t)\int_{\frac{\pi}{3c}}^{\frac{\pi}{2c}} F(x_1) (1 + \sin(-cx_1)) \,\dd x_1 + \beta(t) \int_0^1 F\sin\bar{\alpha} \\
	&\leq \frac{3c \mu}{16} - \beta(t)\int_{\frac{\pi}{3c}}^{\frac{\pi}{2c}} F(x_1) (1 - \sin(cx_1)) \,\dd x_1 + \mathcal{E}(t,\alpha_t,0).
\end{align*}
Since $\alpha(t)$ is the unique minimizer of $\mathcal{E}(t,\placeholder,0)$, we observe
\begin{equation*}
	\frac{3c \mu}{16} > \beta(t)\int_{\frac{\pi}{3c}}^{\frac{\pi}{2c}} F(x_1) (1 + \sin(-cx_1)) \,\dd x_1 \geq \beta(t)F(\tfrac{1}{2}) \int_{\frac{\pi}{3c}}^{\frac{\pi}{2c}} 1 - \sin(cx_1) \,\dd x_1 =  \frac{(\pi - 3)\beta(t)F(\tfrac{1}{2})}{6c}.
\end{equation*}
Here, the second inequality holds, since $c \geq \pi$ and $F$ is decreasing. We conclude that
\begin{equation*}
	\abs*{\alpha_t'(0)} > \sqrt{ \frac{8(\pi - 3)\beta(t)F(\tfrac{1}{2})}{9\mu}} =: c^*(t), \qquad
	\text{if } c^*(t) \geq \pi.
\end{equation*}
Combining this result with \eqref{Eq:Example:UpperBoundTstar}, we conclude that $t \nin \mathcal{T}$, if $c^*(t) \geq \max\set*{\pi, \frac{\delta\sqrt{3\pi}}{\mu\sqrt{16}}}$, which establishes the upper bound for $t^*$.
\end{proof}


\appendix
\section{Appendix}

\subsection{Equivalence of strain gradient terms}

We show that the strain gradient terms $\int_{\Omega} H^C_\el\big(\grad \sqrt{C_y}\big)$ and $\int_{\Omega} H^R_\el\big(\grad R_y\big)$ can be replaced by $\int_{\Omega} H^C_\el\big(\grad C_y\big)$ and $\int_{\Omega} H^R_\el\big(\grad \grad[h]y\big)$ respectively, without changing limit and procedure for the proofs. This is due to the following lemma.

\begin{lemma} \label{Lem:EquivalenceRestTerms}
Let $(F^h), (\hat{F^h}) \subset \SobW^{1,p}(\Omega, \R^{3\times 3})$ with $\det(F^h), \det(\hat{F}^h) > 0$ a.e.\ in $\Omega$. Consider the polar decompositions $F^h = R^h_* \sqrt{C^h}$ and $\hat{F}^h = \hat{R}^h_* \sqrt{\hat{C}^h}$ where $C^h := (F^h)^TF^h$ and $\hat{C}^h := (\hat{F}^h)^T\hat{F}^h$.
\begin{enumerate}[(a)]
	\item We have
	\begin{equation}
		\limsup_{h \to 0} h^{-\alpha_C}\norm[\big]{\sqrt{C^h} - I}_{\Leb^\infty(\Omega)} < \infty
		\quad\Leftrightarrow\quad \limsup_{h \to 0} h^{-\alpha_C}\norm*{C^h - I}_{\Leb^\infty(\Omega)} < \infty.
	\end{equation}
	
	\item Suppose $\limsup_{h \to 0} h^{-\alpha_C}\norm[\big]{\sqrt{C^h} - I}_{\Leb^\infty(\Omega)} < \infty$. Then,
	\begin{subequations}
	\begin{align}
		\limsup_{h \to 0} h^{-\alpha_C}\norm[\big]{\grad\sqrt{C^h}}_{\Leb^p(\Omega)} < \infty
		&\quad\Leftrightarrow\quad \limsup_{h \to 0} h^{-\alpha_C}\norm*{\grad C^h}_{\Leb^p(\Omega)} < \infty, \label{Eq:EquivalenceRestTermsCh}
	\intertext{and if additionally one (and then both) of these statements hold, then,}
		\limsup_{h \to 0} h^{\alpha_R}\norm[\big]{\grad R^h_*}_{\Leb^p(\Omega)} < \infty
		&\quad\Leftrightarrow\quad \limsup_{h \to 0} h^{\alpha_R}\norm*{\grad F^h}_{\Leb^p(\Omega)} < \infty. \label{Eq:EquivalenceRestTermsRh}
	\end{align}
	\end{subequations}
	
	\item Suppose $\limsup_{h \to 0} h^{-\alpha_C}\Big(\norm[\big]{\sqrt{C^h} - I}_{\Leb^\infty(\Omega)} + \norm[\big]{\sqrt{\hat{C}^h} - I}_{\Leb^\infty(\Omega)} + \norm[\big]{\grad\sqrt{C^h}}_{\Leb^p(\Omega)}\Big) < \infty$. Then,
	\begin{subequations}
	\begin{align}
		\lim_{h \to 0} h^{-\alpha_C}\norm[\big]{\grad\sqrt{\hat{C}^h} - \grad\sqrt{C^h}}_{\Leb^p(\Omega)} = 0
		&\quad\Leftrightarrow\quad \lim_{h \to 0} h^{-\alpha_C}\norm*{\grad\hat{C}^h - \grad C^h}_{\Leb^p(\Omega)} = 0,
	\intertext{and if additionally $\limsup_{h \to 0} h^{\alpha_R}\norm*{\grad R^h_*}_{\Leb^p(\Omega)} < \infty$ and one (and then both) of these statements hold, then for arbitrary $(A^h), (\hat{A}^h) \subset \Leb^\infty(\Omega, \R^{3\times 3})$ bounded,}
		\lim_{h \to 0} h^{\alpha_R}\norm*{\hat{A}^h\grad\hat{R}^h_* - A^h\grad R^h_*}_{\Leb^p(\Omega)} = 0
		&\quad\Leftrightarrow\quad \lim_{h \to 0} h^{\alpha_R}\norm*{\hat{A}^h\grad\hat{F}^h - A^h\grad F^h}_{\Leb^p(\Omega)} = 0.
	\end{align}
	\end{subequations}
\end{enumerate}
\end{lemma}

\begin{proof}
We sketch the proof by providing the essential identities that establish the statements. For (a) we note
\begin{equation*}
	C^h - I = (\sqrt{C^h} - I)(\sqrt{C^h} + I), \qquad
	\sqrt{C^h} - I = \tfrac{1}{2}(C^h - I) + \oclass(\abs*{C^h - I}).
\end{equation*}
Here, the second equation is due to a Taylor expansion of the map $\R^{3\times 3}_{\sym} \ni C \mapsto \sqrt{C}$ at $I$. Indeed, one may appeal to the implicit function theorem to show that this map is continuously differentiable near $I$.

For (b) first note that for small enough $h \ll 1$, $C^h \in \SobW^{1,p}(\Omega, \R^{3\times 3})$, since $F^h$ is bounded in $\Leb^\infty(\Omega, \R^{3\times 3})$ by the assumption $\limsup_{h \to 0} h^{-\alpha_C}\norm[\big]{\sqrt{C^h} - I}_{\Leb^\infty(\Omega)} < \infty$. We utilize again the implicit function theorem to show that also $\sqrt{C^h} \in \SobW^{1,p}(\Omega, \R^{3\times 3})$ and the following equations hold:
\begin{equation*}
	\partial_i C^h = 2 \sym \big( \sqrt{C^h} \partial_i \sqrt{C^h} \big), \qquad
	\partial_i \sqrt{C^h} = \partial_i C^h + 2\sym\big( (I - \sqrt{C^h})\partial_i \sqrt{C^h}) \big), \qquad
	i=1,2,3.
\end{equation*}
From these equations one may prove the first part of (b) easily using standard estimation techniques. Furthermore, the second part follows from the identity
\begin{equation*}
	\partial_i F^h = R^h_* \partial_i \sqrt{C^h} + (\partial_i R^h_*) \sqrt{C^h}.
\end{equation*}
Indeed, one may use that a Taylor expansion of the map $\Gl_+(3) \ni F \mapsto F^{-1}$ at $I$ shows that
\begin{equation*}
	\limsup_{h \to 0} h^{-\alpha_C} \norm*{(\sqrt{C^h})^{-1} - I}_{\Leb^\infty(\Omega)} < \infty.
\end{equation*}

Finally, for (c) one may argue as for (b) using the identities
\begin{align*}
	\partial_i (\hat{C}^h - C^h) &= 2\sym \Big( \sqrt{\hat{C}^h}\partial_i\big(\sqrt{\hat{C}^h} - \sqrt{C^h}\big) + \big(\sqrt{\hat{C}^h} - I\big)\partial_i \sqrt{C^h} + \big(I - \sqrt{C^h}\big)\partial_i \sqrt{C^h} \Big), \\
	\partial_i (\sqrt{\hat{C}^h} - \sqrt{C^h}) &= \tfrac{1}{2}\partial_i(\hat{C}^h - C^h) + \sym \left( (I - \sqrt{\hat{C}^h})\partial_i (\sqrt{\hat{C}^h} - \sqrt{C^h}) + (\sqrt{C^h} - \sqrt{\hat{C}^h})\partial_i\sqrt{C^h}\right), \\
	\hat{A}^h \partial_i \hat{F}^h - A^h \partial_i F^h &= \begin{aligned}[t]
		\big( \hat{A}^h \partial_i \hat{R}^h_* - A^h \partial_i R^h_* \big)\sqrt{\hat{C}^h} + A^h(\partial_i R^h_*)(\sqrt{\hat{C}^h} - \sqrt{C^h}) \\
		+ \hat{A}^h \hat{R}^h_* (\partial_i \sqrt{\hat{C}^h}) - A^h R^h_* (\partial_i \sqrt{C^h}).
	\end{aligned}
\end{align*}
These equations are consequences of the ones proposed in the argumentation for (b). 
\end{proof}

\subsection{Weak and a.e.\ convergence in Lebesgue spaces}

\begin{lemma}[{cf.\ \cite[Prop.~A.2]{FM06}}] \label{Lem:PointwiseConvergenceFromWeakConvergence}
Let $1 \leq p < \infty$, $U \subset \R^d$ measurable, and $(f_k) \subset \Leb^p(U)$, $f \in \Leb^p(U)$, such that $f_k \wto f$ in $\Leb^p(U)$ and $f(x) = \limsup_{k \to \infty} f_k(x)$ (or $f(x) = \liminf_{k \to \infty} f_k(x)$, respectively) for a.e.\ $x \in U$. Then, $f_k(x) \to f(x)$ for a.e.\ $x \in U$.
\end{lemma}

\begin{proof}
Assume $f(x) = \limsup_{k \to \infty} f_k(x)$ for a.e.\ $x \in U$. The claim for $f(x) = \liminf_{k \to \infty} f_k(x)$ can be obtained by considering $-f_k$ and $-f$. Consider $\tilde{f}(x) := \liminf_{k \to \infty} f_k(x)$. Using Fatou's lemma (with uniformly integrable sequence of minorants, cf.\ \cite[Thm.~B.3.5]{MR15}), we obtain $\tilde{f} \in \Leb^1(U)$. Using Fatou's lemma again and $f_k \wto f$, we calculate for any $\varphi \in \Cont^\infty_c(U)$,
\begin{align*}
	\int_U (\tilde{f} - f) \varphi
	= \lim_{k \to \infty} \int_U (\tilde{f} - f_k) \varphi
	\geq \int_U \liminf_{k \to \infty} (\tilde{f}(x) - f_k(x)) \varphi(x) \,\dd x
	= 0.
\end{align*}
Thus, $f(x) \leq \tilde{f}(x)$ for a.e.\ $x \in U$. Hence, using $f(x) = \limsup_{k \to \infty} f_k(x)$, we conclude the claim from,
\begin{equation*}
	\limsup_{k \to \infty} f_k(x) = f(x) \leq \tilde{f}(x) = \liminf_{k \to \infty} f_k(x) \leq \limsup_{k \to \infty} f_k(x).
\qedhere
\end{equation*}
\end{proof}

\clearpage
\printbibliography[heading=bibintoc]

@Article{BNS20,
  author    = {Bauer, Robert AND Neukamm, Stefan AND Sch{\"a}ffner, Mathias},
  journal   = {Journal of Elasticity},
  title     = {Derivation of a Homogenized Bending--Torsion Theory for Rods with Micro-Heterogeneous Prestrain},
  year      = {2020},
  issn      = {0374-3535},
  number    = {1},
  pages     = {109--145},
  volume    = {141},
  doi       = {10.1007/s10659-020-09777-6},
  publisher = {Springer},
}

@Article{MM03,
  author    = {Mora, Maria Giovanna AND M{\"u}ller, Stefan},
  journal   = {Calculus of Variations and Partial Differential Equations},
  title     = {Derivation of the nonlinear bending-torsion theory for inextensible rods by $\Gamma$ -convergence},
  year      = {2003},
  issn      = {1432-0835},
  month     = nov,
  number    = {3},
  pages     = {287--305},
  volume    = {18},
  doi       = {10.1007/s00526-003-0204-2},
  publisher = {Springer Science and Business Media LLC},
}

@Article{MR20,
  author    = {Mielke, Alexander and Roub\'{i}\v{c}ek, Tom\'{a}\v{s}},
  journal   = {Archive for Rational Mechanics and Analysis},
  title     = {Thermoviscoelasticity in Kelvin--Voigt Rheology at Large Strains},
  year      = {2020},
  number    = {1},
  pages     = {1--45},
  volume    = {238},
  doi       = {10.1007/s00205-020-01537-z},
  publisher = {Springer},
}

@Article{BFK23,
  author    = {Rufat Badal and Manuel Friedrich and Kru\v{z}\'{i}k, Martin},
  journal   = {Archive for Rational Mechanics and Analysis},
  title     = {Nonlinear and Linearized Models in Thermoviscoelasticity},
  year      = {2023},
  issn      = {0003-9527},
  month     = {jan},
  number    = {1},
  volume    = {247},
  doi       = {10.1007/s00205-022-01834-9},
  publisher = {Springer Science and Business Media {LLC}},
}

@Article{NR24-PP,
  author       = {Neukamm, Stefan and Richter, Kai},
  title        = {Linearization and Homogenization of nonlinear elasticity close to stress-free joints},
  year         = {2024},
  copyright    = {arXiv.org perpetual, non-exclusive license},
  doi          = {10.48550/arXiv.2406.04831},
  journaltitle = {arXiv preprint},
  publisher    = {arXiv},
}

@Article{BS24,
  author    = {Benoit-Maréchal, Lucas and Salvalaglio, Marco},
  journal   = {Modelling and Simulation in Materials Science and Engineering},
  title     = {Gradient elasticity in Swift–Hohenberg and phase-field crystal models},
  year      = {2024},
  issn      = {1361-651X},
  month     = may,
  number    = {5},
  pages     = {055005},
  volume    = {32},
  doi       = {10.1088/1361-651x/ad42bb},
  publisher = {IOP Publishing},
}

@Article{FK18,
  author    = {Friedrich, Manuel and Kružík, Martin},
  journal   = {SIAM Journal on Mathematical Analysis},
  title     = {On the Passage from Nonlinear to Linearized Viscoelasticity},
  year      = {2018},
  issn      = {1095-7154},
  month     = jan,
  number    = {4},
  pages     = {4426--4456},
  volume    = {50},
  doi       = {10.1137/17m1131428},
  publisher = {Society for Industrial & Applied Mathematics (SIAM)},
}

@Article{NV13,
  author  = {Neukamm, Stefan and Vel{\v{c}}i{\'c}, Igor},
  journal = {Mathematical Models and Methods in Applied Sciences},
  title   = {Derivation of a homogenized von-Kármán plate theory from 3D nonlinear elasticity},
  year    = {2013},
  number  = {14},
  pages   = {2701--2748},
  volume  = {23},
  doi     = {10.1142/S0218202513500449},
}

@Article{MT04,
  author    = {Mielke, Alexander AND Theil, Florian},
  journal   = {Nonlinear Differential Equations and Applications},
  title     = {On rate-independent hysteresis models},
  year      = {2004},
  pages     = {151--189},
  volume    = {11},
  doi       = {10.1007/s00030-003-1052-7},
  publisher = {Birkh{\"a}user-Verlag},
}

@Article{BK23,
  author    = {Bresciani, Marco and Kružík, Martin},
  title     = {A Reduced Model for Plates Arising as Low-Energy {$\Gamma$}-Limit in Nonlinear Magnetoelasticity},
  doi       = {10.1137/21m1446836},
  issn      = {1095-7154},
  number    = {4},
  pages     = {3108--3168},
  volume    = {55},
  journal   = {SIAM Journal on Mathematical Analysis},
  month     = jul,
  publisher = {Society for Industrial & Applied Mathematics (SIAM)},
  year      = {2023},
}

@Article{DSV09,
  author    = {dell’Isola, F. and Sciarra, G. and Vidoli, S.},
  journal   = {Proceedings of the Royal Society A: Mathematical, Physical and Engineering Sciences},
  title     = {Generalized Hooke’s law for isotropic second gradient materials},
  year      = {2009},
  issn      = {1471-2946},
  month     = apr,
  number    = {2107},
  pages     = {2177--2196},
  volume    = {465},
  doi       = {10.1098/rspa.2008.0530},
  publisher = {The Royal Society},
}

@Article{DKPS21,
  author    = {Davoli, Elisa and Kru\v{z}\'{i}k, Martin and Piovano, Paolo and Stefanelli, Ulisse},
  journal   = {Continuum Mechanics and Thermodynamics},
  title     = {Magnetoelastic thin films at large strains},
  year      = {2021},
  number    = {2},
  pages     = {327--341},
  volume    = {33},
  doi       = {10.1007/s00161-020-00904-1},
  publisher = {Springer},
}

@Article{BGNPP23,
  author    = {S{\"o}ren Bartels and Max Griehl and Stefan Neukamm and David Padilla-Garza and Christian Palus},
  journal   = {Mathematical Models and Methods in Applied Sciences},
  title     = {A nonlinear bending theory for nematic {LCE} plates},
  year      = {2023},
  month     = {may},
  number    = {07},
  pages     = {1437--1516},
  volume    = {33},
  doi       = {10.1142/s0218202523500331},
  publisher = {World Scientific Pub Co Pte Ltd},
}

@Article{Min64,
  author    = {Mindlin, R. D.},
  title     = {Micro-structure in linear elasticity},
  doi       = {10.1007/bf00248490},
  issn      = {1432-0673},
  number    = {1},
  pages     = {51--78},
  volume    = {16},
  journal   = {Archive for Rational Mechanics and Analysis},
  month     = jan,
  publisher = {Springer Science and Business Media LLC},
  year      = {1964},
}

@Article{BNPS22,
  author    = {Klaus Böhnlein and Stefan Neukamm and David Padilla-Garza and Oliver Sander},
  journal   = {Journal of Nonlinear Science},
  title     = {A Homogenized Bending Theory for Prestrained Plates},
  year      = {2022},
  number    = {1},
  volume    = {33},
  doi       = {10.1007/s00332-022-09869-8},
  publisher = {Springer Science and Business Media {LLC}},
}

@Article{GOW20,
  author    = {Griso, Georges and Orlik, Julia and Wackerle, Stephan},
  journal   = {Journal de Mathématiques Pures et Appliquées},
  title     = {Asymptotic behavior for textiles in von-Kármán regime},
  year      = {2020},
  issn      = {0021-7824},
  month     = dec,
  pages     = {164--193},
  volume    = {144},
  doi       = {10.1016/j.matpur.2020.10.002},
  publisher = {Elsevier BV},
}

@Article{FK20,
  author    = {Friedrich, Manuel and Kru\v{z}\'{i}k, Martin},
  journal   = {Archive for Rational Mechanics and Analysis},
  title     = {Derivation of von K{\'{a}}rm{\'{a}}n Plate Theory in the Framework of Three-Dimensional Viscoelasticity},
  year      = {2020},
  issn      = {0003-9527},
  month     = {jun},
  number    = {1},
  pages     = {489--540},
  volume    = {238},
  doi       = {10.1007/s00205-020-01547-x},
  publisher = {Springer Science and Business Media {LLC}},
}

@Article{Neu12,
  author    = {Neukamm, Stefan},
  journal   = {Archive for Rational Mechanics and Analysis},
  title     = {Rigorous Derivation of a Homogenized Bending-Torsion Theory for Inextensible Rods from Three-Dimensional Elasticity},
  year      = {2012},
  number    = {2},
  pages     = {645--706},
  volume    = {206},
  doi       = {10.1007/s00205-012-0539-y},
  publisher = {Springer},
}

@Article{Tou62,
  author    = {Toupin, R. A.},
  title     = {Elastic materials with couple-stresses},
  doi       = {10.1007/bf00253945},
  issn      = {1432-0673},
  number    = {1},
  pages     = {385--414},
  volume    = {11},
  journal   = {Archive for Rational Mechanics and Analysis},
  publisher = {Springer Science and Business Media LLC},
  year      = {1962},
}

@Article{FM23,
  author    = {Friedrich, Manuel and Machill, Lennart},
  title     = {One-dimensional viscoelastic von Kármán theories derived from nonlinear thin-walled beams},
  doi       = {10.1007/s00526-023-02525-3},
  issn      = {1432-0835},
  number    = {7},
  volume    = {62},
  journal   = {Calculus of Variations and Partial Differential Equations},
  month     = jul,
  publisher = {Springer Science and Business Media LLC},
  year      = {2023},
}

@Article{Tou64,
  author    = {Toupin, R. A.},
  title     = {Theories of elasticity with couple-stress},
  doi       = {10.1007/bf00253050},
  issn      = {1432-0673},
  number    = {2},
  pages     = {85--112},
  volume    = {17},
  journal   = {Archive for Rational Mechanics and Analysis},
  month     = jan,
  publisher = {Springer Science and Business Media LLC},
  year      = {1964},
}

@Article{MS13,
  author  = {Mielke, Alexander AND Stefanelli, Ulisse},
  journal = {Journal of the European Mathematical Society},
  title   = {Linearized plasticity is the evolutionary $\Gamma$-limit of finite plasticity},
  year    = {2013},
  number  = {3},
  pages   = {923--948},
  volume  = {15},
  doi     = {10.4171/JEMS/381},
}

@Book{KR19,
  author    = {Kru\v{z}\'{i}k, Martin and Roub\'{i}\v{c}ek, Tom\'{a}\v{s}},
  publisher = {Springer Cham},
  title     = {Mathematical Methods in Continuum Mechanics of Solids},
  year      = {2019},
  address   = {Cham},
  edition   = {1},
  isbn      = {978-3-030-02065-1},
  series    = {Interaction of Mechanics and Mathematics},
  doi       = {10.1007/978-3-030-02065-1},
}

@Book{MR15,
  author    = {Mielke, Alexander and Roub\'{i}\v{c}ek, Tom\'{a}\v{s}},
  publisher = {Springer New York},
  title     = {Rate-Independent Systems},
  year      = {2015},
  isbn      = {9781493927067},
  series    = {Applied Mathematical Sciences},
  doi       = {10.1007/978-1-4939-2706-7},
  issn      = {2196-968X},
  subtitle  = {Theory and Application},
}

@Article{FM06,
  author  = {Francfort, Gilles and Mielke, Alexander},
  journal = {Journal f{\"u}r die reine und angewandte Mathematik},
  title   = {Existence results for a class of rate-independent material models with nonconvex elastic energies},
  year    = {2006},
  number  = {595},
  pages   = {55--91},
  volume  = {2006},
  doi     = {10.1515/CRELLE.2006.044},
}

@Article{Lee69,
  author    = {Lee, E. H.},
  title     = {Elastic-Plastic Deformation at Finite Strains},
  doi       = {10.1115/1.3564580},
  issn      = {1528-9036},
  number    = {1},
  pages     = {1--6},
  volume    = {36},
  journal   = {Journal of Applied Mechanics},
  month     = mar,
  publisher = {ASME International},
  year      = {1969},
}

@Article{FJM06,
  author    = {Friesecke, Gero and James, Richard D. and M{\"u}ller, Stefan},
  journal   = {Archive for Rational Mechanics and Analysis},
  title     = {A Hierarchy of Plate Models Derived from Nonlinear Elasticity by Gamma-Convergence},
  year      = {2006},
  number    = {2},
  pages     = {183--236},
  volume    = {180},
  doi       = {10.1007/s00205-005-0400-7},
  publisher = {Springer},
}

@Article{MRS08,
  author    = {Mielke, Alexander and Roub\'{i}\v{c}ek, Tom\'{a}\v{s} and Stefanelli, Ulisse},
  journal   = {Calculus of Variations and Partial Differential Equations},
  title     = {$\Gamma$-limits and relaxations for rate-independent evolutionary problems},
  year      = {2008},
  number    = {3},
  pages     = {387--416},
  volume    = {31},
  doi       = {10.1007/s00526-007-0119-4},
  publisher = {Springer},
}

@Article{FJM02,
  author    = {Friesecke, Gero and James, Richard D. and M{\"u}ller, Stefan},
  journal   = {Communications on Pure and Applied Mathematics},
  title     = {A theorem on geometric rigidity and the derivation of nonlinear plate theory from three-dimensional elasticity},
  year      = {2002},
  number    = {11},
  pages     = {1461--1506},
  volume    = {55},
  doi       = {10.1002/cpa.10048},
  fjournal  = {Communications on Pure and Applied Mathematics: A Journal Issued by the Courant Institute of Mathematical Sciences},
  publisher = {Wiley Online Library},
}

@Article{Dav14_02,
  author  = {Davoli, Elisa},
  journal = {Mathematical Models and Methods in Applied Sciences},
  title   = {Quasistatic evolution models for thin plates arising as low energy $Gamma$-limits of finite plastictiy},
  year    = {2014},
  issn    = {0218-2025},
  number  = {10},
  pages   = {2085--2153},
  volume  = {24},
  doi     = {10.1142/s021820251450016x},
}

@Article{BGKN22-PP,
  author       = {Bartels, S{\"o}ren and Griehl, Max and Keck, Jakob and Neukamm, Stefan},
  title        = {Modeling and simulation of nematic LCE rods},
  year         = {2022},
  copyright    = {arXiv.org perpetual, non-exclusive license},
  doi          = {10.48550/arXiv.2205.15174},
  journaltitle = {arXiv preprint},
  publisher    = {arXiv},
}

@Article{OL24,
  author    = {van Oosterhout, Willem J. M. and Liero, Matthias},
  title     = {Finite‐strain poro‐visco‐elasticity with degenerate mobility},
  doi       = {10.1002/zamm.202300486},
  issn      = {1521-4001},
  number    = {5},
  volume    = {104},
  journal   = {Zeitschrift für Angewandte Mathematik und Mechanik},
  month     = mar,
  publisher = {Wiley},
  year      = {2024},
}

@Article{AD20,
  author    = {Virginia Agostiniani and Antonio DeSimone},
  journal   = {Mathematics and Mechanics of Solids},
  title     = {Rigorous derivation of active plate models for thin sheets of nematic elastomers},
  year      = {2020},
  number    = {10},
  pages     = {1804--1830},
  volume    = {25},
  doi       = {10.1177/1081286517699991},
  publisher = {{SAGE} Publications},
}
\end{document}